\documentclass{article}
\usepackage[journal=RMI,lang=american]{ems-journal-rmi} 

\numberwithin{equation}{section}

\usepackage{color}
\usepackage{subcaption}
\usepackage{comment}

\newcommand{\cor}{{({\rm c})}}
\newcommand{\side}{{({\rm s})}}

\newcommand{\dx}{\partial_x}
\newcommand{\dt}{\partial_t}
\newcommand{\dz}{\partial_z}

\newcommand{\tpsi}{\widetilde{\psi}}

\newcommand{\cE}{{\mathcal E}}
\newcommand{\mfh}{\mathfrak h}

\newcommand{\gr}{{\mathtt g}}

\newcommand{\RR}{{\mathbb R}}

\newcommand{\vertiii}[1]{{\left\vert\kern-0.25ex\left\vert\kern-0.25ex\left\vert #1 
    \right\vert\kern-0.25ex\right\vert\kern-0.25ex\right\vert}}

\newtheorem{proposition}{Proposition}
\newtheorem{theorem}{Theorem}
\newtheorem*{theorem*}{Theorem}
\newtheorem{lemma}{Lemma}
\newtheorem{corollary}{Corollary}
\newtheorem{assumption}{Assumption}
\newtheorem{definition}{Definition}

\theoremstyle{remark}
\newtheorem{remark}{Remark}

\begin{document}

\title{Well posedness of F. John's floating body problem for a fixed object}
\titlemark{Well-posedness of F. John's floating body problem}

\emsauthor{1}{
	\givenname{David}
	\surname{Lannes}
	\zblid{author.first}
	\mrid{1234567}
	\orcid{0000-0001-0002-0003}}{D.~Lannes}
\emsauthor{2}{
	\givenname{Mei}
	\surname{Ming}
	\mrid{885432}
	\zblid{ming.mei}
	\orcid{0000-0002-6496-9289}}{M.~Ming}

\Emsaffil{1}{
	\pretext{}
	\department{Institut de Mathématiques de Bordeaux}
	\organisation{CNRS UMR 5251 et Université de Bordeaux}
	\address{351 Cours de la Libération}
	\zip{33405}
	\city{Talence Cedex}
	\country{France}
	\posttext{}
	\affemail{david.lannes@math.u-bordeaux.fr}
	}

\Emsaffil{2}{
	\pretext{}
	\department{School of Mathematics and Statistics}
	\organisation{Yunnan University}
	\address{1}{Rm3422 Gewu Building, Yunnan University, East Outer Ring Road, Chenggong District}
	\zip{650500}
	\city{Kunming}
	\country{People's Republic of China}
	\posttext{}
	\affemail{mingmei@ynu.edu.cn}
	}

\classification[35S10,35J57]{35Q35}

\begin{abstract}
The goal of this paper is to prove the well-posedness of F. John's floating body problem in the case of a fixed object and for unsteady waves, in horizontal dimension $d=1$ and with a possibly emerging bottom. This problem describes the interactions of waves with a partially immersed object using the linearized Bernoulli equations. The fluid domain $\Omega$ therefore has corners where the object meets the free surface, which consists of various connected components. The energy space associated with this problem involves the space of traces on these different connected components  of all functions in the Beppo-Levi space $\dot{H}^1(\Omega)$; we characterize this space, exhibiting non local effects linking the different connected components. We prove the well-posedness of the Laplace equation in corner domains, with mixed boundary conditions and Dirichlet data in this trace space, and study several properties of the associated Dirichlet-Neumann operator (self-adjointness, ellipticity properties, etc.). This trace space being only semi-normed, we cannot use standard semi-group theory to solve F. John's problem: one has to choose a realization of the homogeneous space (i.e. choose an adequate representative in the equivalence class) we are working with. When the fluid domain is bounded, this realization is obtained by imposing a zero-mass condition; for unbounded fluid domains, we have to choose a space-time realization which can be interpreted as a particular choice of the Bernoulli constant. Well-posedness in the energy space is then proved. Conditions for higher order regularity in times are then derived, which yield some limited space regularity that can be improved through smallness assumption on the contact angles. We finally show that higher order regularity away from the contact points can be achieved through weighted estimates.
\end{abstract}

\maketitle

\section{Introduction}

\subsection{General setting}

The interaction of waves with floating structures is obviously important for navigation issues but also for renewable marine energies (floating wind turbines or wave energy converters) or the modeling of sea-ice. From a mathematical point of view, this is a very complex problem because in addition to all the difficulties met to describe ocean waves, one must understand their interaction with the floating objects. This interaction raises new problems, among which the (free boundary) dynamics of the contact line between the object, the air, and the surface of the water, and the fact that the fluid domain can no longer be assumed to be smooth since it has wedges (or corners in the two-dimensional case). If the object is allowed to move freely, one has in addition to study its motion under the force exerted by the fluid. 

A wide range of mathematical and numerical tools are used to describe wave structure interactions. Computational Fluid Dynamics (CFD) is used to get precise information such as vortex-induced motion of floaters or global performance in extreme waves; however, it has to be said that CFD is far less adapted for this kind of situation than it is, say, in aeronautics or car industry. This is due to several specific difficulties identified in \cite{Kim} and that include the fact that one has to work in an open ocean environment that requires computations of large volumes of fluid, and that one has to face a non-Gaussian stochastic environment that requires many computations to get reliable statistics; for these reasons, the total cost of a CFD project is comparable to the cost of physical model tests.

Faster, but of course less precise computations, can be achieved if one uses a simpler model for the description of the fluid. Fully Nonlinear Potential Flow (FNPF) methods \cite{Penalba} for instance assume that the fluid  is non viscous, incompressible and that the flow is irrotational. The fluid velocity therefore derives from a scalar velocity potential whose evolution is governed by the Bernoulli equations. In the absence of any floating object, this amounts to assuming that the flow is governed by the water waves equations (see \cite{ABZ2,HIT} for low regularity results that allow $C^{3/2}$ interfaces). Though faster than CFD computations, the FNPF approach remains extremely costly, for the same reasons (using them to describe the propagation of waves without floating object is already an achievement, even for one dimensional surfaces \cite{Benoit}).

Another approach, proposed in \cite{Lannes1} consists in replacing the equations for the fluid by simpler asymptotic wave models such as the Nonlinear Shallow Water (NSW) or the Boussinesq equations. Such models, that couple the evolution of the free surface with the vertically averaged horizontal velocity have proven very efficient to describe waves in coastal areas \cite{Lannes_SW}. The presence of a floating object can be accounted for by imposing a constraint on the surface elevation in the region where the object is located; the pressure exerted by the fluid on the object can then be understood as the Lagrange multiplier associated with this constraint. The problem can then be reduced to a non standard initial boundary value problem for the wave model cast in the exterior region (where the surface is not in contact with the object); see \cite{IguchiLannes}, \cite{Bocchi,Bocchi2} and \cite{IguchiLannes2023} for the $1d$ nonlinear, $2d$ radial and $2d$ NSW equations respectively, and \cite{BLM,BL} for the $1d$ Boussinesq equations. We refer also to \cite{Maity} for the viscous $1d$ NSW equations and also mention the so-called "soft congestion" approach (see the review by Perrin \cite{Perrin})  that relaxes the constraint on the surface elevation, hereby allowing the use of efficient low-Mach techniques \cite{Godlewski1,Godlewski2}. Such approaches based on asymptotic wave models have been numerically implemented for the $1d$ nonlinear shallow water \cite{BocchiHeVergara,Haidar} and Boussinesq equations \cite{BLW} and are numerically very efficient. The hope is that they can be integrated to the operational codes based on these wave models, which would allow to study the impact of farms of floating offshore wind turbines or wave energy converters and the wave fields, their incidence on submersion risks, etc. Their drawback is of course their limited range of validity (e.g. shallow water models cannot be used in deep water) that must be assessed numerically and experimentally.

Today, the most widely used approach to describe wave-structure interactions corresponds to the {\it linear} potential flow approach. In this approach, the variations of the fluid domain are neglected for the computation of the velocity potential. This removes in particular the difficult issue of the evolution of the contact line. This approach is used in the commercial codes used by engineers. From a mathematical point of view, this problem attracted a lot of interest at the middle of the XXth century, starting with a series of two papers by F. John \cite{John1,John2} who studied the interaction of waves with floating structures in free or forced motion, assuming the the fluid is governed by the linear Bernoulli equation. Many articles have been devoted to this since when, almost all of them focusing on time harmonic motions and on the possible existence and properties of trapped modes \cite{McIver,KMV}; let us mention however the early work of Ursell \cite{Ursell} who considered
the decay of the free motion of a floating symmetric body in a fluid initially at rest. Despite its physical importance, the analysis of unsteady waves in the presence of obstacles is only understood for completely submerged objects \cite{KMV}; one of the reasons why such configurations are easier to handle mathematically is that the fluid domain remains regular. As pointed in \cite{KMV} the analysis of the Cauchy problem for unsteady waves in the presence of a partially immersed object is open. This is the problem that we address in this paper, in the particular case where the floating object is fixed.

\subsection{Statement of the problem}
We consider the behavior of small perturbations of a fluid occupying at rest a two-dimensional domain $\Omega$ with finite depth, and in the presence of one or several partially immersed fixed solid objects. On the right and left extremities, the fluid can be delimited or not by an emerging bottom, as shown in Figure \ref{fig:images}. More precisely, we make the following assumption on $\Omega$ and its boundary $\Gamma$.

\begin{figure}[htbp]
        \begin{subfigure}[b]{0.45\textwidth}
                \centering
                \includegraphics[width=\textwidth]{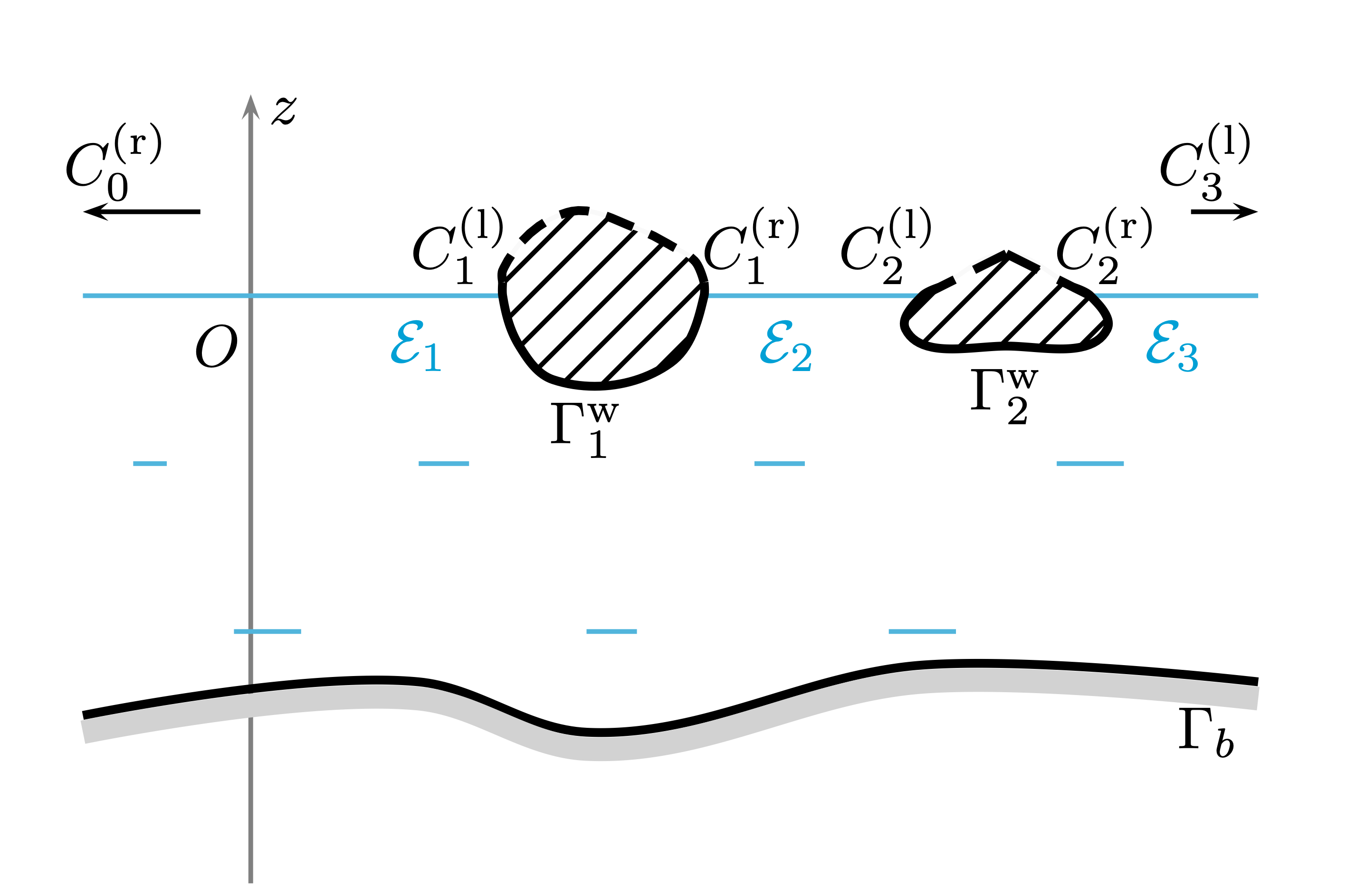}
                \caption{$N=2$, non-emerging bottom, $\Gamma^{\rm D}$ unbounded}
                \label{fig:image1}
        \end{subfigure}%
        \hfill
        \begin{subfigure}[b]{0.45\textwidth}
                \centering
                \includegraphics[width=\textwidth]{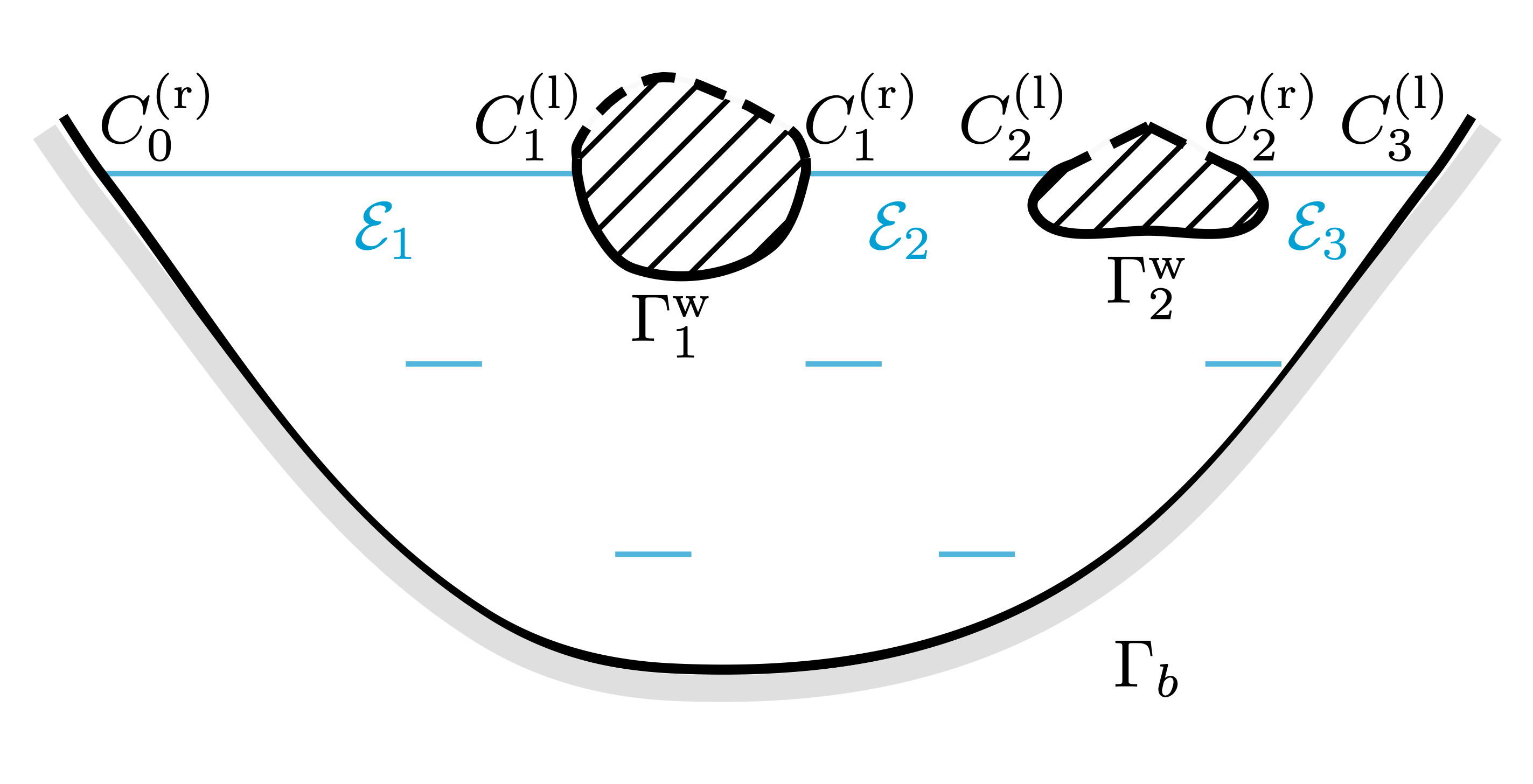}
                \caption{$N=2$, emerging bottom, $\Gamma^{\rm D}$ bounded}
                \label{fig:image2}
        \end{subfigure}

         \begin{subfigure}[b]{0.45\textwidth}
                \centering
                \includegraphics[width=\textwidth]{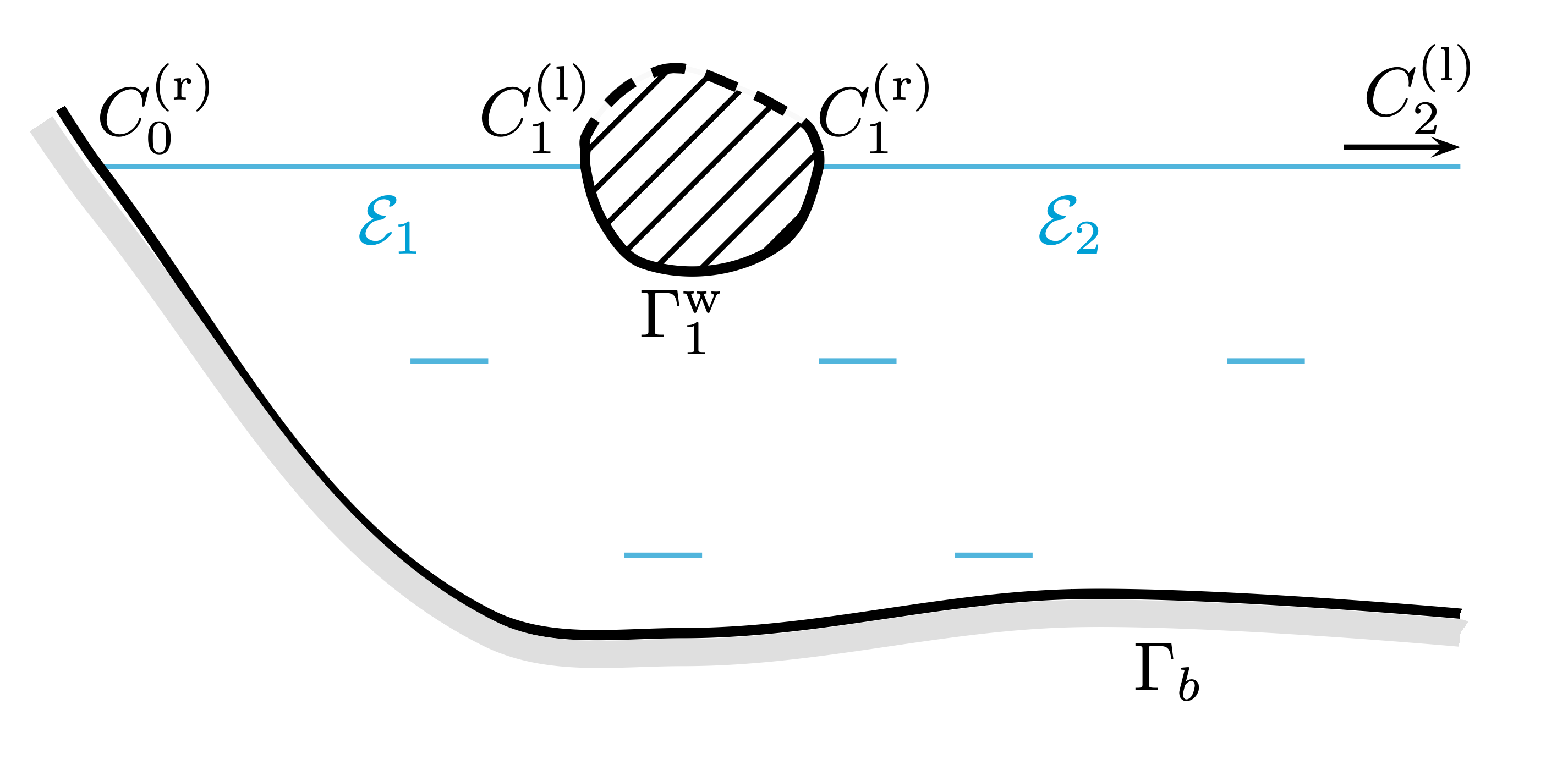}
                \caption{$N=1$, emerging bottom, $\Gamma^{\rm D}$ unbounded}
                \label{fig:image3}
        \end{subfigure}
        \caption{Admissible configurations}
        \label{fig:images}
\end{figure}

\begin{assumption}\label{assconfig}
Let $N\geq 1$ be the number of partially immersed solids. The fluid domain $\Omega$ is delimited from above and from below by two open curves $\Gamma^{(\rm top)}$ and $\Gamma^{(\rm b)}$ respectively, that do not intersect, and such that:
\begin{enumerate}
\item There exist $-\infty\leq x_0^{{\rm r}} <x^{{\rm l}}_1 < x^{\rm r}_1<\dots <x^{{\rm l}}_N<x^{\rm r}_N<x^{{\rm l}}_{N+1}\leq \infty$ such that $\Gamma^{(\rm top)}$ is a curved broken line with (possibly infinite) endpoints $C_0=(x^{\rm r}_0,0)$ and $C_{N+1}=(x^{\rm l}_{N+1},0)$ and vertices $C_j^{(\rm l)}=(x_j^{\rm l},0)$ and $C_j^{(\rm r)}=(x_j^{\rm r},0)$ for $1\leq j\leq N$. Moreover the segments $(C^{(\rm l)}_j,C_j^{(\rm r)})$, which correspond to the wetted part of the boundary of the objects, are smooth open curves denoted by $\Gamma^{({\rm w})}_j$ and contained in the lower half-plane $\{z<0\}$; all the other segments of   $\Gamma^{(\rm top)}$ are flat and contained in the axis $\{z=0\}$. 
\item The bottom  $\Gamma^{({\rm b})}$ is the graph on $(x^{\rm r}_0,x^{\rm l}_{N+1})$ of a smooth function $b$ such that $b(x^{{\rm r}}_0)=0$  if $-\infty<  x^{{\rm r}}_0$ and $\limsup_{x\to - \infty} b(x)<0$ if  $x^{{\rm r}}_0=- \infty$, and similar condition at the other endpoint $x_{N+1}^{\rm l}$. 
\item At each corner, the fluid domain forms an angle strictly larger than $0$ and strictly smaller than $\pi$ (no tangential contact).
\end{enumerate}
We denote by $\Gamma^{\rm D}=\underline{\Gamma}^{\rm D}\times\{0\}$ with $\underline{\Gamma}^{\rm D}=\bigcup_{j=1}^{N+1} {\mathcal E}_j$, where the ${\mathcal E}_j$ are the intervals defined as
$\cE_j=(x_{j-1}^{\rm r},x_j^{\rm l})$.
We also denote by $\Gamma^{\rm N}$ the union of all the other open curved segments of $\Gamma=\partial\Omega$ and by $\Gamma^*=\Gamma^{\rm D}\cup \Gamma^{\rm N}$.
\end{assumption} 
\begin{remark}
The notation $\Gamma^{\rm D}$ for the portion of the surface of the fluid in contact with the air is not natural at this point; it is due to the fact that for the elliptic problem solved by the velocity potential in $\Omega$, Dirichlet type boundary conditions are imposed on $\Gamma^{\rm D}$. Similarly, we denote $\Gamma^{\rm N}$ the portion of $\Gamma$ where the fluid is in contact with a solid boundary because Neumann boundary conditions are imposed on these segments.
\end{remark}

In order to state the linear floating problem formulated by F. John \cite{John1} in the case of a fixed object, let us first introduce some notations. Since the flow is potential, and because the floating objects and the bottom are assumed to be impermeable and fixed, the velocity $U(t,\cdot)$ in the fluid domain at time $t$ can be written $U(t,\cdot)=\nabla \phi(t,\cdot)$, where $\phi(t,\cdot)$ is a scalar function, harmonic in $\Omega$, and such that its normal derivative satisfies $\partial_{\rm n}\phi=0$ on $\Gamma^{\rm N}$. It is therefore fully determined by its trace on $\Gamma^{\rm D}$ (see Proposition \ref{propharmonic} for a precise statement), which we denote by $\psi(t,\cdot)$. This latter quantity therefore determines also the normal derivative $\partial_n\phi$ on $\Gamma^{\rm D}$, which we denote by $G_0\psi$; since $\Gamma^{\rm D}$ is horizontal here, we have equivalently $G_0\psi:=(\partial_z\phi )_{\vert_{\Gamma^{\rm D}}}$; the operator $G_0$ is called the Dirichlet-Neumann operator. Denoting also by $\zeta(t,\cdot)$ the surface elevation above $\Gamma^{\rm D}$ at time $t$, F. John's floating body problem can be stated as
\begin{equation}\label{F1}
\begin{cases}
\dt \zeta -G_0 \psi &= f ,\\
\dt \psi + {\mathtt g}\zeta &=g, 
\end{cases}
\quad \mbox{ on }\quad 
\RR^+_t\times {\Gamma}^{\rm D},
\end{equation}
where ${\mathtt g}>0$ is the acceleration of gravity, and the function $f$ and $g$ are given source terms. The goal of this paper is to prove the well-posedness of the Cauchy problem formed by \eqref{F1} and initial conditions of the form
\begin{equation}\label{F2}
(\zeta,\psi)_{\vert_{t=0}}=(\zeta^{\rm in},\psi^{\rm in}),
\end{equation}
for some given functions $\zeta^{\rm in}$ and $\psi^{\rm in}$.

\subsection{Presentation of the results}

In the presentation of our results, we identify four questions that can be of independent interest. Firstly, in order to identify the energy space for \eqref{F1}, we are led to characterize the range of the boundary trace mapping for functions in the homogeneous Beppo-Levi space $\dot{H}^1(\Omega)$. Secondly, since the energy space is not homogeneous and of critical regularity, the question of its realization is not trivial, especially when the fluid domain is unbounded, in which case we can relate it to the choice of the Bernoulli constant; once a good realization is found, we can prove the well-posedness of F. John's problem in the energy space. Since the regularity issues raised by the corners of the fluid domain play a central role; we review the literature on this subject and try to explain the difficulties at stake by discussing an apparent paradox related to corner singularities for water waves; we can finally present our approach to get higher order regularity estimates for solutions to \eqref{F1}-\eqref{F2} away from the corners.

\subsubsection{The energy space and the boundary trace of $\dot{H}^1(\Omega)$ functions}

We consider in this paper waves of finite mechanical energy, that is, we want the surface elevation $\zeta$ and the velocity $U=\nabla\phi$ to be such that
$$
E:= \frac{1}{2}\rho {\mathtt g}\int_{\Gamma^{\rm D}} \zeta^2+\frac{1}{2}\rho \int_\Omega \vert \nabla \phi\vert^2 <\infty,
$$
where $\rho$ is the constant density of the fluid.
Quite obviously, $L^2(\Gamma^{\rm D})$ is the natural functional space for $\zeta$; for $\psi$ however, the choice is much less clear. We should consider the space of all the functions $\psi$ that can be obtained as the trace on $\Gamma^{\rm D}$ of an harmonic function $\phi$ of finite Dirichlet energy, that is, $\int_\Omega \vert \nabla \phi\vert^2 <\infty$. Since every function $\phi$ in the Beppo-Levi space $\dot{H}^1(\Omega):=\{ \phi \in L^1_{\rm loc}(\Omega), \nabla \phi \in L^2(\Omega)\}$ can be uniquely decomposed as $\phi=\phi_1+\phi_2$, with $\phi_1\in H^1(\Omega)$ is such that $(\phi_1)_{\vert_{\Gamma^{\rm D}}}\equiv 0$ and $\phi_2$ is harmonic with homogeneous Neumann condition on $\Gamma^{\rm N}$, the problem of finding the correct functional space for $\psi$ can be reformulated as the following question, of independent interest.\\

\medbreak
\noindent
{\bf Question 1.} Can we characterize the range of the mapping ${\mbox Tr}^{\rm D}: \phi \in \dot{H}^1(\Omega)\mapsto \phi_{\vert_{\Gamma^{\rm D}}}$ ? And if so, does it admit a continuous right inverse ?
\medbreak

Recently, a very similar question was raised in \cite{Gaudin1,Gaudin2}, but for the trace of homogeneous space of the form $\dot{H}^{s+1/2}(\Omega)$ with $s<1/2$; the case we have to deal with here is therefore the critical case $s=1/2$. The critical case was addressed recently by Strichartz \cite{Strichartz} in the case where $\Omega$ is a flat strip and $\Gamma^{\rm D}$ consisting of the union of the two boundaries. He showed that the structure of the range of the trace mapping has some interesting structures that led Leoni and Tice \cite{LeoniTice} to develop a general theory for {\it screened} Sobolev spaces. The challenge here is to manage the fact that we are in a corner domain and that $\Gamma^{\rm D}$ now has boundaries.

We answer by the affirmative to both parts of Question 1 in Section \ref{sectionTraces}, and denote by $\dot{H}^{1/2}(\Gamma^{\rm D})$ the range of the trace mapping. This space is only a semi-normed space, and its norm contains a nonlocal effect relating the different connected components of $\Gamma^{\rm D}$. A rough statement of the result presented in Theorems \ref{theortrace} and \ref{theoremsurj} is the following.
\begin{theorem*}
We characterize a homogeneous space $\dot{H}^{1/2}(\Gamma^{\rm D})$ such that the trace mapping ${\mbox Tr}^{\rm D}: \dot{H}^1(\Omega)\to  \dot{H}^{1/2}(\Gamma^{\rm D})$ is well defined, continuous, onto, and admits a continuous right-inverse.
\end{theorem*}
Note that we also prove in Proposition \ref{propharmonic} that the Laplace equation in $\Omega$, with homogeneous Neumann conditions on $\Gamma^{\rm N}$ and Dirichlet data on $\Gamma^{\rm D}$ in $\dot{H}^{1/2}(\Gamma^{\rm D})$  is well posed, which extends the classical result where the Dirichlet data are in ${H}^{1/2}(\Gamma^{\rm D})$; this result is optimal since we have the equivalence of semi-norms $\vert \psi \vert_{\dot{H}^{1/2}(\Gamma^{\rm D})}\sim \big(\int_\Omega \vert\nabla \phi\vert^2\big)^{1/2}$ (which is of course false if we replace $\vert \psi \vert_{\dot{H}^{1/2}(\Gamma^{\rm D})}$ by $\vert \psi \vert_{{H}^{1/2}(\Gamma^{\rm D})}$).

It also follows from these considerations and the definition of the energy above, that the natural energy space for \eqref{F1} is the {\it semi-normed} space ${\mathbb X}$ defined as
$$
{\mathbb X}:=L^2(\Gamma^{\rm D})\times \dot{H}^{1/2}(\Gamma^{\rm D}).
$$

\subsubsection{Realization of the energy space and well-posedness in the energy space}

In order to prove the well-posedness of \eqref{F1}-\eqref{F2} in the energy space, one has to prove that the evolution operator in \eqref{F1} is skew-adjoint. This is achieved via a careful analysis of the Dirichlet-Neumann operator on corner domains, but this is not enough to conclude by standard tools because the energy space ${\mathbb X}$, being only a semi-normed space, is not a Banach space. In such situations, it is convenient to find a {\it realization} of the semi-normed space ${\mathbb X}$, that is, a convenient choice of a particular representative of each element of the quotient space ${\mathbb X}/K$, $K$ being the adherence of zero for the semi-norm. A well-known example is Chemin's space ${\mathcal S}'_h$ of homogeneous distributions \cite{BCD} (see also \cite{Cobb}), but this approach based on Fourier analysis is not directly applicable in our case where $\Gamma^{\rm D}$ has boundaries and more importantly, it works for homogeneous spaces of the form $\dot{H}^s$ with $s<1/2$; here again, the case $s=1/2$ of interest here is critical.

\medbreak
\noindent
{\bf Question 2.} Can we find a good realization of the energy space ${\mathbb X}$ ?
\medbreak

When $\Gamma^{\rm D}$ is bounded, it is quite easy to answer this question by the affirmative. Indeed, the space ${\mathcal L}^2(\Gamma^{\rm D})\times \dot{\mathcal H}^{1/2}(\Gamma^{\rm D})$ consisting of all $(\zeta,\psi)\in {\mathbb X}$ such that $\int_{\Gamma^{\rm D}}\zeta = \int_{\Gamma^{\rm D}}\psi=0$ is an appropriate realization of ${\mathbb X}$ since the equations propagate these zero mass conditions. This realization is also a standard Hilbert space in which we can apply standard tools to prove the well-posedness of \eqref{F1}-\eqref{F2}.\\
When $\Gamma^{\rm D}$ is unbounded, this realization does not apply because the integrals $\int_{\Gamma^{\rm D}}\zeta$ and $\int_{\Gamma^{\rm D}}\psi$  are not necessarily well defined. Our approach is then to construct a solution by a duality method in the semi-normed space $L^2([0,T];{\mathbb X})$; we then show that it is possible to choose a realization of this homogeneous space which depends on space {\it and time}. This realization turns to have a physical meaning: it consists in choosing the Bernoulli constant to be zero in the second equation of \eqref{F1}. The well-posedness in the energy space can then be established. A simplified statement of Theorems \ref{theoWPbounded} and \ref{theoWPunbounded} is the following.
\begin{theorem*}
The initial value problem \eqref{F1}-\eqref{F2} is well-posed in the energy space.
\end{theorem*}

\subsubsection{Regularity issues and the corner angle singularity paradox for water waves}\label{ssectparadox}

One of the difficulties one has  to face with F. John's floating body problem is that the fluid domain has corners. Due to these singularities, the velocity potential is not necessarily in $H^{s+1}(\Omega)$, with $s\geq 0$, when the Dirichlet data $\psi$ belongs to $H^{s+1/2}(\Gamma^{\rm D})$. While such a property is true for all $s\geq 0$ in smooth domains, there is a threshold value $s_0$ for $s$ above which this elliptic regularity property does not hold when $\Omega$ has corners \cite{Grisvard,Dauge,KMR}. Consequently, above this threshold, the Dirichlet-Neumann operator does not map $H^{s+1/2}(\Gamma^{\rm D})$ into $H^{s-1/2}(\Gamma^{\rm D})$ and the techniques developed to deal with the standard water waves equations become useless because they strongly use the fact that $G_0$ is an operator of order $1$. Since the regularity threshold $s_0$ is larger when the angles of the corners of the fluid domain are small, a way to overcome this difficulty is to assume that the angles are small. This is the approach proposed by T. de Poyferré \cite{Poyferre} (see also Ming and Wang \cite{MW2,MW3,MW4} in the presence of surface tension) to derive a priori estimates for the nonlinear water waves equations in the presence of an emerging bottom, assuming that the so called Rayleigh-Taylor coefficient is positive, and using the geometrical framework developed by Shatah and Zeng \cite{SZ}.

The role of this coefficient can be seen by writing the equations satisfied by the trace of the horizontal velocity at the free surface, which we denote  $\underline{v}$ ; in horizontal dimension $d=1$ (for the sake if simplicity), this equation is given by
$$
\dt \underline{v}+\underline{v}\dx \underline{v}+{\mathfrak a}\dx \zeta =0,
$$
where ${\mathfrak a}:= -(\partial_y P)_{\vert_{\Gamma^{\rm D}}}$ is the Rayleigh-Taylor coefficient. Its positiveness ensures the hyperbolicity of the equations; Wu showed that it is always positive in infinite depth \cite{Wu1997} and Lannes in finite depth with a flat bottom \cite{Lannes} when the surface is sufficiently regular. But in the presence of corner singularities this coefficient could a priori vanish. Assuming that it is positive yields some compatibility conditions that are quite obvious when one considers water waves on a half-line delimited by a vertical wall at $x=0$. As pointed out by Alazard, Burq and Zuily \cite{ABZ3}, since $\underline{v}$ vanishes at $x=0$, and if ${\mathfrak a}>0$ at this point, one must necessarily have  $\dx\zeta =0$ as a consequence of the equation satisfied by $\underline{v}$; in other words, the contact angle must necessarily be $90$ degrees. 

In order to have more freedom on the angle at the singularity of the fluid domain, it is possible to relax the assumption that ${\mathfrak a}>0$. This approach has been considered by several authors interested by the existence of possible surface waves with a corner singularity. The difficulty is then to handle the fact that the hyperbolicity of the equations degenerates if ${\mathfrak a}$ vanishes.  Kinsey and Wu \cite{KinseyWu,Wu2015} allowed the Rayleigh-Taylor coefficient to degenerate at the corner points and derived weighted energy estimates for angles at the crest that are less than $\pi/2$; Wu \cite{Wu2019} then established a local existence result for a class of initial data which allows such singularities, and Agrawal \cite{Agrawal} showed that for this class of solutions the corner singularities are preserved by the evolution in time, and that the angle of the crest does not change with time. Cordoba, Enciso and Grubic \cite{CEG1} were then able to establish local-well-posedness in a class of weighted Sobolev spaces allowing corner singularities, but this time with an angle that may vary in time; however, they had to impose several symmetries that imposed them to work in the zero gravity case. Recently, in \cite{CEG2}, the same authors were able to remove the symmetry assumptions and therefore to keep the gravity term. Similarly, for singularities caused by an emerging bottom, it is also possible to consider configurations where the Rayleigh-Taylor coefficient vanishes at the singularity, as done recently by Ming \cite{M2}.

Unfortunately, these approaches are not suited for our present purpose since, in the linear case, one has ${\mathfrak a}={\mathtt g}$ which is strictly positive. We therefore potentially have to deal with the constraints on the angle illustrated by the example of Alazard, Burq and Zuily mentioned above, and that looks to be valid in our situation too. Indeed, if we consider a floating object with a boundary that intersects vertically the free surface, one has $\dx\psi=0$ at the contact points and it follows from \eqref{F1} (without source term), that ${\mathtt g}\dx \zeta=0$ too. The fact that this argument should be handled with care is well illustrated by the fact that it raises a paradox when applied to the famous extreme Stokes wave, which is a progressive periodic solution of the water waves equation of maximum amplitude and is known to form an angle of $120$ degrees at its crest \cite{Stokes1,Stokes2,Toland}. In the frame moving at the velocity $c$ of the wave, the equation for the surface velocity becomes
$$
(\underline{v}-c)\dx \underline{v}+{\mathfrak a}\dx \zeta =0.
$$
It is also known since Stokes that the crest is a stagnation point, meaning that $\underline{v}=c$ at the crest; moreover, the Rayleigh-Taylor coefficient is ${\mathfrak a}=\frac{3}{4}{\mathtt g}$ at the crest and therefore positive \cite{Lyons}. For the same reason as above, this seems to imply that $\dx\zeta=0$, and therefore that the angle at the crest should be $2\times 90$ degrees rather than $120$ degrees, hence the following question.

\medbreak
\noindent
{\bf Question 3.} How can one explain the angle corner singularity paradox, and what happens at the contact points?
\medbreak

For the Stokes wave, this paradox can be explained by a singular behavior near the corner. For a Stokes wave of amplitude $a$, with crest located at $x=0$, one has near the crest $\zeta \sim a-\frac{1}{\sqrt{3}}\vert{x}\vert$ and $\underline{v}\sim c+(\frac{\sqrt{3}}{2}{\mathtt g})^{1/2} \sqrt{\vert x \vert}$; one has therefore $(\underline{v}-c)\dx \underline{v}\sim \frac{\sqrt{3}}{4} {\mathtt g}$ which is non zero, so that $\dx\zeta$ does not vanish, hereby explaining the paradox\footnote{Note that in the configuration considered in \cite{ABZ3} there is no paradox since by symmetry reasons, the condition $\dx\zeta=0$ at the contact point is  propagated from the initial data so that the claim of the authors is of course correct.}.\\
For F. John's floating body problem with an object intersecting vertically the free surface, which is a linear problem, this nonlinear explanation does not hold, but there still can be obstructions to apply the argument of Alazard, Burq and Zuily. Indeed, in order to infer that $\dx\psi=0$ at the contact points from the homogeneous Neumann condition on the object, some minimal regularity on $\phi$ is needed. For instance, in  \cite{Su}, Su, Tucsnak and Weiss considered the linear water waves equation in a rectangular basin (so that Fourier decomposition can be applied) with a regularity $H^{1/2}$ on $\zeta$ and $H^1$ on $\psi$; at such regularity, the compatibility condition $\dx\psi=0$ and $\dx \zeta=0$ at the corners do not appear. It is however not clear whether the compatibility conditions appear automatically if $\psi$ is smoother since it is known that the regularity of the velocity potential $\phi$ is limited by the presence of the corner.  The critical regularity for $\phi$ above which compatibility conditions arise is $H^2(\Omega)$. Indeed, if  $\phi\in H^{2+\epsilon}(\Omega)$ with $\epsilon>0$, one has $\dx\phi \in H^{1+\epsilon}(\Omega)$ and by the trace theorem $(\dx\phi)_{\vert_{\Gamma}}\in H^{1/2+\epsilon}(\Gamma)$; since $(\dx \phi)_{\vert_{\Gamma^{\rm N}}}$ vanishes near the contact point, this implies that $\dx\psi =(\dx \phi)_{\vert_{\Gamma^{\rm D}}}$ vanishes at the contact points. It follow that in the standard Sobolev scale $H^s(\Gamma^{\rm D})$, the critical regularity for $\dx\psi$ should be $H^{1/2}(\Gamma^{\rm D})$.\\
We explain in the next section that using weighted estimates one can however reach higher order regularity. 

\subsubsection{Time regularity, space regularity and weighted estimates}

As explained above, we manage to prove the well-posedness of F. John's problem \eqref{F1}-\eqref{F2} in $C([0,T];{\mathbb X})$. Under appropriate assumptions on the initial datas, one can prove that the solution is more regular in time, say, $C^n([0,T];{\mathbb X})$, with $n \in {\mathbb N}$. For the standard linear water waves equations, such an information is sufficient to deduce space regularity for the solution, using ellipticity properties of the Dirichlet-Neumann operator. Here, this is only partially true due to the singularities of the fluid domain. We are able to establish ellipticity properties for the Dirichlet-Neumann operator, but only up to a given regularity threshold. As a consequence, even if we had infinite time regularity we could only deduce finite space regularity. Without smallness assumption on the angles, one cannot get better than $\zeta\in H^{3/2}(\Gamma^{\rm D})$ and $\dx \psi\in \dot{H}^{1/2}(\Gamma^{\rm D})$, which is consistent with the comments made above to explain the corner singularity paradox. 

Let us mention here that Guo and Tice faced the same difficulty in their study of the stability of the contact line at smaller scales, where both viscosity and capillarity must be taken into account \cite{GuoTice1,GuoTice2}; to handle this issue they developed a theory based on the functional calculus for the capillary gravity operator which, unfortunately, cannot be used here in the absence of surface tension, and also because we work in a possibly unbounded domain.
In  order to get higher order regularity in the configuration under consideration here,   one can try to use weighted derivatives of the form $\rho\dx$, with $\rho$ a bounded function behaving in the neighborhood of each corner as the distance to this corner. This approach has been used for instance in \cite{KinseyWu,Wu2015,Wu2019,CEG1,CEG2,M2} where the authors work in weighted Sobolev spaces, which is a natural framework for elliptic problems in corner domains, see for instance Kozlov, Maz'ya and Rossmann \cite{KMR,KMR2001}.
The issue is that the Dirichlet-Neumann operator operates on such weighted Sobolev spaces only under smallness assumptions on the angle; typically, with the space $V_\beta^{l-1/2}(\Gamma^{\rm D})$ defined as in Section 6.2.1 of \cite{KMR} (where $l$ is the order of derivatives and $\beta$ is some power of the weight) the Dirichlet-Neumann operator is a continuous mapping from $V_\beta^{l-1/2}(\Gamma^{\rm D})$ to $V_\beta^{l-3/2}(\Gamma^{\rm D})$ provided that $l\geq 2$ is an integer, $\beta\in \RR$ and that $\vert l-\beta-1\vert < \frac{\pi}{2\omega_0}$, where $\omega_0$ is the largest corner angle of the domain (this is a consequence of Theorems 1.4.3 and Corollary 6.3.1 of \cite{KMR2001}).

\medbreak
\noindent
{\bf Question 4.} Can we use weighted derivatives without making any smallness assumption on the angles ?
\medbreak

Here again, we answer this question by the affirmative. The strategy is to carefully look at the structure of the commutator $[G_0,\rho\dx]$. If $\Omega$ were an angular sector and $\rho=r$, we could compute explicitly this commutator using polar coordinates; more precisely, we would have $[G_0,\rho\dx]=G_0$. Formally applying $\rho\dx$ to \eqref{F1} (without source term), we would find that $(\rho\dx \zeta,\rho\dx\psi)$ would solve 
$$
\begin{cases}
\dt (\rho\dx\zeta) -G_0 \rho\dx\psi &= G_0\psi ,\\
\dt (\rho\dx\psi) + {\mathtt g}\rho\dx\zeta &=0 ,
\end{cases}
\quad \mbox{ on }\quad 
\RR^+_t\times {\Gamma}^{\rm D}.
$$
Using the well-posedness in the energy space mentioned above, we would get a control of $(\rho\dx \zeta,\rho\dx\psi)$ in ${\mathbb X}$ provided that the source term $G_0\psi$ is in $L^2$, which is equivalent to say, by \eqref{F1}, that $\dt\zeta \in L^2$; therefore the weighted regularity could be deduced from the time regularity and the commutator identity $[G_0,\rho\dx]=G_0$. If the domain $\Omega$ is as in Assumption \ref{assconfig} then this identity is no longer true, but this commutator can still be controlled in terms of time derivatives (and lower order terms).

For all $n\in {\mathbb N}$, we can then identify a space ${\mathbb Y}^n \subset{\mathbb X}$ whose norms controls $( (\rho\dx)^j\zeta,(\rho\dx)^j \psi )$ in ${\mathbb X}$ for all $0\leq j\leq n$ as well as $(\zeta,\psi)$ in $H^{j/2}\times \dot{H}^{(j+1)/2}$ for $0\leq j\leq \min\{2,n\}$, and prove the following result (we refer to Theorem \ref{maintheo} for a precise statement).
\begin{theorem*}
For all $n\in {\mathbb N}$, the initial value problem \eqref{F1}-\eqref{F2} is well-posed for data in ${\mathbb Y}^n$.
\end{theorem*}

\subsection{Organization of the paper}

In Section \ref{sectionTraces}, we identify the range of the mapping that takes the trace on $\Gamma^{\rm D}$ of functions in $\dot{H}^{1}(\Omega)$. We first consider in \S \ref{sectstrip} the case where $\Omega$ is a flat strip, in which case the range of the trace mapping, can be characterized in terms of the space $\dot{H}^{1/2}(\RR)$, which we define and characterize in terms of screened Sobolev spaces that are described in \S \ref{sectscreen}.  We then define in \S \ref{sectinthalf} the functional space $\dot{H}^{1/2}(I)$ when $I$ is a finite interval or a half-line; the case of $\dot{H}^{1/2}(\Gamma^{\rm D})$ is more complex because $\Gamma^{\rm D}$ has several connected components; we show in \S \ref{sectdotH12} that the corresponding semi-norm is non local in the sense that it includes terms relating the different connected components. We then show in \S \ref{secttracecont}  that the trace mapping $\mbox{Tr}^{\rm D}:\dot{H}^1(\Omega)\to \dot{H}^{1/2}(\Gamma^{\rm D})$ is well defined and continuous, and in \S \ref{secttracesurj} that it is onto and admits a continuous right inverse. Further properties of the space $\dot{H}^{1/2}(\Gamma^{\rm D})$ are then investigated in \S \ref{sectdensity} and \S \ref{sectdotH012}. 

Section \ref{sectLaplace} is then devoted to the analysis of the Laplace equation in $\Omega$, with homogeneous Neumann conditions on $\Gamma^{\rm D}$ and Dirichlet data in $\dot{H}^{1/2}(\Gamma^{\rm D})$ on $\Gamma^{\rm D}$. We prove  in \S \ref{sectharmext} that this problem is well posed, which allows us to rigorously define the Dirichlet-Neumann operator $G_0$ in \S \ref{sectDN}.  This operator is constructed as a mapping from $\dot{H}^{1/2}(\Gamma^{\rm D})$ with values in its dual but we also show that it is a continuous elliptic operator from $\dot{H}^1(\Gamma^{\rm D})$ to $L^2(\Gamma^{\rm D})$ and prove several other properties such as self-adjointness and elliptic regularity.

Section \ref{sectWP} is devoted to the well-posedness of F. John's problem in the energy space. We prove in \S \ref{subsectskew}  that the evolution operator in \eqref{F1} is skew-adjoint if $\Gamma^{\rm D}$ is bounded and if a zero mass assumption is made on $\zeta$ and $\psi$; well-posedness for \eqref{F1}-\eqref{F2} then follows from standard semi-group theory.  When $\Gamma^{\rm D}$ is unbounded, as explained above for Question 2, we construct in \S \ref{subsectWPunbnded} a solution by a duality method in a semi-normed functional space which we realize (in the sense of \cite{Bourdaud}) through a convenient choice of the Bernoulli constant; well-posedness is then established. Higher order time regularity is then studied in \S \ref{sectHOTR}  and we show in \S \ref{sectlimreg} how to deduce a limited amount of space regularity from time regularity.  

We then proceed in Section \ref{sectcommutator} to study commutator terms of the form $[(\rho\dx)^j,G_0]$ (see Question 4 above). We first construct in \S \ref{sectnotref} and \S \ref{sectweight} a convenient weight function $\rho$ on $\Omega$, and then, in \S \ref{sectTN}, we construct an extension $T$ of the tangent vector ${\bf t}$ defined on $\Gamma^*$; this extension is used to construct an extension $X=\rho T\cdot \nabla$ defined on $\Omega\cup \Gamma^*$ of the weighted derivative $\rho\dx$.  In \S \ref{sectHOHE}, we measure the regularity of the harmonic extension of $\psi$ using the vector field $X$; finally, the commutator estimates are proved in \S \ref{sectCEDN}.

The final section is devoted to the proof of the well-posedness of F. John's problem in spaces that can measure higher order regularity. These partially weighted  spaces are introduced in \S \ref{sectmixspaces}. We then prove in  \S \ref{sectkeyprop},  a key result of transfer of regularity, which roughly states that an additional $\rho\dx$ derivative of the solution can be controlled at the cost of a time derivative; this is a consequence of the commutator estimates of Section \ref{sectcommutator}, which are themselves the equivalent in $\Omega$ of the commutation property proved in a sector in the comments related to Question 4 above. A well-posedness result in such partially weighted spaces is then proved in \S \ref{sectMR}, without smallness assumption on the angles.

\medbreak

\noindent
{\bf Acknowledgment.} The authors thank T. de Poyferré for discussions about this work, and M.Ming wants to thank Chongchun Zeng and Chao Wang for some discussions on the water waves problem.

\subsection{Notations}
We provide here some  notations used throughout this article, with a particular emphasis on the function spaces that we have to use.
\subsubsection{General notations}
- We use $C$ as a generic notation for a strictly positive constant of no importance; the value of $C$ may change from one line to another.\\
- The notations $a\lesssim b$ stands for $a\leq C b$.\\
- We denote by $x$ and $z$  the horizontal and vertical coordinates and by $\partial_x$ or $\partial_1$ and $\partial_z$ or $\partial_2$ the corresponding partial derivatives; we also write $\nabla=(\partial_x,\partial_z)^{\rm T}$.
\subsubsection{Function spaces}
Here are some notations used throughout this article for function spaces. We first provide the notations for the function spaces on the fluid domain $\Omega$. Since, with the notations of Assumption \ref{assconfig}, $\Gamma^{\rm D}$ can be identified with a finite union of finite or semi-infinite intervals $\cE_j$, we then introduce function spaces over such intervals; finally, we introduce the function spaces on the non-connected  portion $\Gamma^{\rm D}$ of the boundary.  Note that when the context is clear, we often use shortened notations, for instance $\Vert \phi \Vert_{L^2(\Omega)}$ may be written $\Vert \phi \Vert_{L^2}$ or $\Vert \phi \Vert_2$.

\medskip
\noindent $\bullet$ \underline{Function spaces on $\Omega$.}\\
- We denote by $L^2(\Omega)$ the standard Lebesgue space, and by $\Vert \cdot \Vert_{L^2(\Omega)}$ its canonical norm.\\  
- We denote by $H^1(\Omega)$ the standard non-homogeneous Sobolev space with canonical norm $\Vert \phi \Vert_{H^1(\Omega)}=\Vert \phi\Vert_{L^2}+\Vert \nabla \phi\Vert_{L^2}$.\\
- We denote by $H^{1}_{\rm D}(\Omega)$ the set of functions in $H^1(\Omega)$ whose trace on $\Gamma^{\rm D}$ vanishes, that is, the completion of ${\mathcal D}(\Omega\cup \Gamma^{\rm N})$ for the $H^1(\Omega)$ norm, where ${\mathcal D}(\Omega\cup \Gamma^{\rm N})$ is the set of infinitely smooth functions on $\RR^2$ supported in a compact set included in $\Omega\cup \Gamma^{\rm N}$.\\
- We denote by $\dot{H}^1(\Omega)$ the homogeneous (or Beppo-Levi) Sobolev space of order $1$
$$
\dot{H}^1(\Omega)=\{\phi\in L^1_{\rm loc}(\Omega), \nabla\phi\in L^2(\Omega)^2\} ,
$$
endowed with the semi-norm $\Vert \phi\Vert_{\dot{H}^1(\Omega)}:=\Vert \nabla\phi \Vert_{L^2}$, see \eqref{defdotH1}.

\medskip
\noindent $\bullet$ \underline{ Function spaces on a non-empty open interval $I\subsetneq \RR$.}\\
- The set of smooth functions compactly supported in $I$ is denoted ${\mathcal D}(I)$.\\
- We denote by $\widetilde{L}^1_{\rm loc}(I):=\{ f\in L^1_{\rm loc}(I), \exists \widetilde{f}\in L^1_{\rm loc}(\RR), \widetilde{f}_{\vert_I}=f\}$ (the notation $L^1_{\rm loc}(\overline{I})$ is sometimes used in the literature for these spaces); in particular, if $I$ is bounded, $\widetilde{L}^1_{\rm loc}(I)=L^1(I)$.\\
- We denote by $H^{1/2}(I)$ the standard Sobolev space of order $1/2$, 
$$
H^{1/2}(I)=\{ f\in L^2(I), \vert f\vert_{H^{1/2}(I)}<\infty  \},
$$
where
$  \vert f\vert_{H^{1/2}(I)}=\vert f\vert_{L^2(I)}+\big(\iint_{x,y\in I}\frac{( f(y)-f(x) )^2}{(y-x)^2}{\rm d}y{\rm d}x\big)^{1/2}$.\\
- We denote by  $H_0^{1/2}(I)$ the completion  of ${\mathcal D}(I)$ in $H^{1/2}(I)$. It is classical that $H_0^{1/2}(I)=H^{1/2}(I)$, both algebraically (these two sets are the same) and topologically (their norms are equivalent).\\
- We denote by $H_{00}^{1/2}(I)$ the set of functions in $H^{1/2}(I)$ whose extension by zero in $\RR\backslash I$ belongs to $H^{1/2}(\RR)$. This set is strictly included in $H^{1/2}(I)=H^{1/2}_0(I)$.\\
- The homogeneous Sobolev space of order $1/2$, denoted ${H}_{\rm hom}^{1/2}(I)$  is defined as
$$
\dot{H}^{1/2}_{\rm hom}(I):=\{f\in \widetilde{L}^1_{\rm loc}(I), \vert f\vert_{\dot{H}^{1/2}_{\rm hom}(I)}<\infty \} ,
$$
with $  \vert f\vert_{\dot{H}^{1/2}_{{\rm hom}}(I)}:= 
\big(\iint_{x,y\in I}\frac{( f(y)-f(x) )^2}{(y-x)^2}{\rm d}y{\rm d}x\big)^{1/2}$, so that $\vert f \vert_{H^{1/2}(I)}=\vert f\vert_{L^2(I)}+ \vert f\vert_{\dot{H}^{1/2}_{{\rm hom}}(I)}$.\\
- The {\it screened}  homogeneous Sobolev space of order $1/2$, denoted $\dot{H}^{1/2}(I)$ is defined as
$$
\dot{H}^{1/2}(I):=\{f\in \widetilde{L}^1_{\rm loc}(I), \vert f\vert_{\dot{H}^{1/2}(I)}<\infty \}  ,
$$
with $  \vert f\vert_{\dot{H}^{1/2}}:= 
\big(\iint_{x,y\in I, \vert y-x\vert\leq 1}\frac{( f(y)-f(x) )^2}{(y-x)^2}{\rm d}y{\rm d}x\big)^{1/2}$. This space is defined and studied in \S \ref{sectinthalf}; in particular, one has $\dot{H}^{1/2}(I)={{H}^{1/2}_{\rm hom}(I)}$ when $I$ is finite, but ${{H}^{1/2}_{\rm hom}(I)\subsetneq \dot{H}^{1/2}(I)} $ when $I$ is unbounded.\\
- The homogeneous Sobolev space of order $s\geq 1$ is $\dot{H}^s(I):=\{f\in \widetilde{L}^1_{\rm loc}, \dx f \in H^{s-1}(I)\}$ and is endowed with the canonical semi-norm $\vert f \vert_{\dot{H}^s}=\vert \dx f \vert_{H^{s-1}}$.

\medskip
\noindent $\bullet$\underline{ Function spaces on $\Gamma^{\rm D}$.}\\
- We denote by $\widetilde{L}^1_{\rm loc}(\Gamma^{\rm D}):=\{ f\in L^1_{\rm loc}(\Gamma^{\rm D}), \exists \widetilde{f}\in L^1_{\rm loc}(\RR), \widetilde{f}_{\vert_I{\Gamma^{\rm D}}}=f\}$; in particular, if $\Gamma^{\rm D}$ is bounded, $\widetilde{L}^1_{\rm loc}(\Gamma^{\rm D})=L^1(\Gamma^{\rm D})$;\\
- We define $X(\Gamma^{\rm D})$, with $X={\mathcal D}$, $L^2$, $H^{1/2}$ or $H_{00}^{1/2}$ as $X(\Gamma^{\rm D})=\prod_{j=1}^{N+1}X(\cE_j)$, with $X(\cE_j)$ defined as above.\\
- For all $s\geq 0$, we denote classically by $H^{s}(\Gamma^{\rm D})$ the set of functions whose restriction to $\cE_j$ belongs to $H^{s}(\cE_j)$ for all $1\leq j\leq N+1$; we identify therefore $H^{s}(\Gamma^{\rm D})=\prod_{j=1}^{N+1}H^{s}(\cE_j)$.\\
- The spaces $\dot{H}^{s}(\Gamma^{\rm D})$ with $s=1/2$ or $s\geq 1$  coincide algebraically with $\prod_{j=1}^{N+1}\dot{H}^{s}(\cE_j)$, with $\dot{H}^s(\cE_j)$ as defined above, but are topologically different as they are endowed with the semi-norm
$$
\vert f \vert_{\dot{H}^s(\Gamma^{\rm D})}=\sum_{j=1}^{N+1} \vert f_{\vert_{\cE_j}} \vert_{\dot{H}^s(\cE_j)}+\sum_{j=1}^N \vert \overline{f}_{j+1}-\overline{f}_j\vert,
$$
where $\overline{f}_j=\frac{1}{\vert \cE_j\vert}\int_{\cE_j} f$ is $\cE_j$ is bounded, and $\overline{f}_j$ is the average of $f$ over a finite subinterval of $\cE_j$ if this latter is unbounded. See \S \ref{sectdotH12} for details.\\
- We denote by ${\mathcal L}^2(\Gamma^{\rm D})$ (resp. $\dot{\mathcal H}^{1/2}(\Gamma^{\rm D})$, $\dot{\mathcal H}^{1}(\Gamma^{\rm D})$) the subspace of $L^2(\Gamma^{\rm D})$ (resp. $\dot{H}^{1/2}(\Gamma^{\rm D})$, $\dot{H}^{1}(\Gamma^{\rm D})$) obtained by considering only the elements of this space that satisfy an additional zero mass condition, see \S  \ref{sectdotH012}.\\
- We define the space $\widetilde{H}^{1/2}(\Gamma^{\rm D})$ as
$$
\widetilde{H}^{1/2}(\Gamma^{\rm D})=\{ f \in \widetilde{L}^1_{\rm loc}(\Gamma^{\rm D}), \exists \widetilde{f}\in \dot{H}^{1/2}(\RR), \widetilde{f}_{\vert_{\Gamma^{\rm D}}}=f\},
$$
endowed with its canonical semi-norm; as shown in Proposition \ref{propH12char}, the space $\widetilde{H}^{1/2}(\Gamma^{\rm D})$  coincides algebraically and topologically with $\dot{H}^{1/2}(\Gamma^{\rm D})$.

\section{Traces for Beppo-Levi spaces}\label{sectionTraces}

Let $\Omega$ be a curved polygon and $\Gamma^{\rm D}$ a portion of its boundary be defined as in Assumption \ref{assconfig}.
The trace operator  on $\Gamma^{\rm D}$, denoted ${\rm Tr}^{\rm D}$ is well defined on ${H}^1(\Omega)$ even if $\Omega$ is as here a corner domain \cite{Grisvard}. More precisely, the mapping 
\begin{equation}\label{tracemap}
{\rm Tr}^{\rm D}: F\in H^1(\Omega)\mapsto F_{\vert_{\Gamma^{\rm D}}}\in H^{1/2}(\Gamma^{\rm D})
\end{equation}
is continuous and onto since the connected components of $\Gamma^{\rm D}$ are not adjacent (they are separated by a connected component of $\Gamma^{\rm N}$).  

As already mentioned in the introduction, it is relevant for the problem under consideration in this article to work with velocity potentials $\phi$ such that $\nabla\phi$ belongs to $L^2(\Omega)$ but that are not necessarily themselves in $L^2(\Omega)$. In such a context, it is convenient to work with the
  {\it Beppo-Levi space} $\dot{H}^1(\Omega)$ (or homogeneous Sobolev space) studied in a general framework by Deny and Lions \cite{DenyLions} and
 defined in this particular case as
\begin{equation}\label{defdotH1}
\dot{H}^1(\Omega)=\{\phi\in L^1_{\rm loc}(\Omega), \nabla\phi\in L^2(\Omega)^2\}.
\end{equation}
\begin{remark}\label{remL2L1}
We can  replace the condition $\phi\in L^1_{\rm loc}(\Omega)$ in the definition of $\dot{H}^1(\Omega)$ by $\phi\in L^2_{\rm loc}(\Omega)$ or $\phi\in {\mathcal D}'(\Omega)$ since distributions $\phi\in {\mathcal D}'(\Omega)$ whose gradient is in $L^2(\Omega)$ are automatically in all $L^p_{\rm loc}$ for $1\leq p<\infty$ (see Theorem 2.1 in \cite{DenyLions}). 
\end{remark}

From Remark \ref{remL2L1}, it follows that functions in $\dot{H}^1(\Omega)$ also belong to $H^1_{\rm loc}(\Omega)$, so that the mapping  ${\rm Tr}^{\rm D}$ is also well defined on $\dot{H}^1(\Omega)$. In this section, we want to study the range of this trace operator when defined on the Beppo-Levi space $\dot{H}^1(\Omega)$ and to study its continuity properties.

We first recall in \S \ref{sectstrip} some known properties in the case where $\Omega$ is a horizontal strip; we comment there on some striking differences with the inhomogeneous case (trace of $H^1(\Omega)$ functions). We then introduce screened Sobolev spaces in \S \ref{sectscreen} and then define in \S \ref{sectinthalf} the space $\dot{H}^{1/2}(I)$ when $I$ is a finite interval or a half-line. We then define in \S \ref{sectdotH12} the spaces $\dot{H}^{s}(\Gamma^{\rm D})$, with $s=1/2$ and $s\geq 1$, which are algebraically the product spaces of the $\dot{H}^{s}(\cE_j)$ where $\cE_j$ are the connected components of $\Gamma^{\rm D}$ but that are {\it topologically different}. We then establish  in \S \ref{secttracecont} that the trace of functions of $\dot{H}^{1}(\Omega)$ on $\Gamma^{\rm D}$ belong to $\dot{H}^{1/2}(\Gamma^{\rm D})$ and that the trace mapping is continuous; we then show in \S \ref{secttracesurj} that it is onto by constructing a continuous extension mapping. We then prove in \S \ref{sectdensity} that ${\mathcal D}(\Gamma^{\rm D})$ is dense in $\dot{H}^{1/2}(\Gamma^{\rm D})$, and finally we introduce in \S \ref{sectdotH012} the realizations $\dot{\mathcal H}^{s}(\Gamma^{\rm D})$ of  $\dot{H}^{s}(\Gamma^{\rm D})$, with $s=1/2$ and $s\geq 1$; the realization $\dot{\mathcal H}^{1/2}(\Gamma^{\rm D})$ is a Banach space of distributions and will play an important role in this article.

\subsection{The case of a strip}\label{sectstrip}

The extension of the trace mapping ${\rm Tr}^{\rm D}$ to $\dot{H}^1(\Omega)$,  was shown in \cite{AL,Lannes_book} when $\Omega$ is a strip with $\Gamma^{\rm D}$ being the upper boundary and $\Gamma^{\rm N}$ the lower one. Considering for the sake of clarity a flat strip of height $h_0>0$, that is, $\Omega=\RR\times (-h_0,0)$, we can canonically identify functions on $\Gamma^{\rm D}$ with functions on $\RR$; the trace space $\dot{H}^{1/2}(\Gamma^{\rm D})=\dot{H}^{1/2}(\RR)$, is then also a Beppo-Levi space in the sense of \cite{DenyLions}, and defined as
\begin{equation}\label{defdotH12}
\dot{H}^{1/2}(\RR):=\{f\in L^1_{\rm loc}(\RR), \partial_x f\in H^{-1/2}(\RR)\},
\end{equation}
with associated semi-norm $\vert{f}\vert_{\dot{H}^{1/2}(\RR)}=\vert{\partial_x f}\vert_{H^{-1/2}(\RR)}$.

The mapping ${\rm Tr}^{\rm D}$ was also considered by Strichartz, still in the case of a flat strip $\Omega=\RR\times (-h_0,0)$, but this time $\Gamma^{\rm N}=\emptyset$ and $\Gamma^{\rm D}$ having two connected components $\Gamma^{\rm D}=\{z=0\}\cup\{z=-h_0\}$. For all $F\in C(\overline{\Omega})$, the trace $F_{\vert_{{\Gamma}^{\rm D}}}$ can therefore be identified with an element $(F(\cdot,0),F(\cdot,-h_0))$ of $C(\RR)\times C(\RR)$. He showed that this mapping could be extended as an onto mapping $\dot{H}^1(\Omega)\to \dot{H}^{1/2}(\Gamma^{\rm D})$, where $\dot{H}^{1/2}(\Gamma^{\rm D})$ is a subspace of $\dot{H}^{1/2}(\RR)\times \dot{H}^{1/2}(\RR)$, namely,
\begin{equation}\label{H12strip}
\dot{H}^{1/2}(\Gamma^{\rm D})=\{ (f_+,f_-) \in \dot{H}^{1/2}(\RR)\times \dot{H}^{1/2}(\RR), f_+-f_- \in L^2(\RR)\}.
\end{equation}
\begin{remark}
In the standard inhomogeneous case the trace mapping maps $H^1(\Omega)$ {\it onto} $H^{1/2}(\RR)\times H^{1/2}(\RR)$. It is a striking difference that the mapping ${\rm Tr}^{\rm D}: \dot{H}^1(\Omega)\to \dot{H}^{1/2}(\RR)\times \dot{H}^{1/2}(\RR)$ is not onto and that an additional condition must be imposed to describe its range. This additional condition, namely, that $f_+-f_- \in L^2(\RR)$, is non-local inasmuch as it involves the traces of $f$ on the different connected components of $\Gamma^{\rm D}$.
\end{remark}

\subsection{Screened homogeneous Sobolev spaces}\label{sectscreen}

Let us first recall the standard notation $m(D)$ for Fourier multipliers defined for $f\in {\mathcal S}'(\RR)$ as
$$
 m(D)f={\mathcal F}^{-1}(\xi\mapsto m(\xi)\widehat{f}(\xi)),
$$
where $\widehat{f}$ is the Fourier transform of $f$, ${\mathcal F}^{-1}$ the inverse Fourier transform, and with appropriate assumptions on $m$ and $f$ ensuring that $\xi\mapsto m(\xi)\widehat{f}(\xi)$ is a tempered distribution.

Using this notation, for functions $f$ that belong to the space $\dot{H}^{1/2}(\RR)$ defined in \eqref{defdotH12} 
and that are in the Schwartz class ${\mathcal S}'(\RR)$, we have the following equivalence of semi-norms
\begin{equation}\label{equivtempered}
\vert f \vert_{\dot{H}^{1/2}(\RR)}\sim  \vert {\mathfrak P} f\vert_{L^{2}(\RR)},
\end{equation}
where the operator ${\mathfrak P}$ is the Fourier multiplier ${\mathfrak P}=\frac{\vert D\vert}{(1+\vert D\vert)^{1/2}}$; this equivalent formulation of the $\dot{H}^{1/2}(\RR)$ semi-norm was used in \cite{AL,Lannes_book}; note that the operator ${\mathfrak P}$ could equivalently be replaced by
 the Fourier multiplier $\min\{\vert D\vert, \vert D\vert^{1/2}\}$ as in \cite{Strichartz,LeoniTice}.  
 
 In particular, the Beppo-Levi space ${\dot{H}^{1/2}(\RR)}$ does not coincide with the {\it homogeneous} Sobolev space $H^{1/2}_{\rm hom}(\RR)$ defined as
$$
H^{1/2}_{\rm hom}(\RR)=\{ f\in {\mathcal S}'(\RR), |D|^{1/2} f \in L^2(\RR)\},
$$
with associated semi-norm $\vert f\vert_{H^{1/2}_{\rm hom}(\RR)}= \vert |D|^{1/2}f \vert_{L^2(\RR)}$.

We refer to \cite{LeoniTice} for counter examples showing that $\dot{H}^{1/2}(\RR)$ and $H^{1/2}_{\rm hom}(\RR)$ are topologically and algebraically distinct. This difference can also be observed using yet another equivalent formulation for the semi-norms. Indeed, one has
\begin{equation}\label{defH12hom}
\vert f \vert_{H^{1/2}_{\rm hom}(\RR)}\sim  \Big(\int_{\RR}\int_{\RR} \frac{\vert f(x)-f(y)\vert^2}{\vert x-y\vert^2}{\rm d}x{\rm d}y\Big)^{1/2},
\end{equation}
while it was shown by Strichartz \cite{Strichartz} that
\begin{equation}\label{defH12screen}
\vert f \vert_{\dot{H}^{1/2}(\RR)} \sim \Big(\int\int_{\vert x-y\vert\leq 1} \frac{\vert f(x)-f(y)\vert^2}{\vert x-y\vert^2}{\rm d}x{\rm d}y\Big)^{1/2};
\end{equation}
the observation that in this latter expression, the range of the difference quotient is screened led Leoni and Tice \cite{LeoniTice} to introduce {\it screened homogeneous Sobolev spaces.}
\begin{remark}\label{remH12screen}
Screened homogeneous Sobolev spaces are defined in a general framework in \cite{LeoniTice}, where their properties are studied with great care. It is in particular shown that different screening functions may define the same space (algebraically and topologically); for instance, one has (Theorem 3.12 of \cite{LeoniTice})
$$
\vert f \vert_{\dot{H}^{1/2}(\RR)} \sim \Big(\int\int_{\vert x-y\vert\leq \rho_0} \frac{\vert f(x)-f(y)\vert^2}{\vert x-y\vert^2}{\rm d}x{\rm d}y\Big)^{1/2}
$$
for any $\rho_0>0$.
\end{remark}

\subsection{The space $\dot{H}^{1/2}(I)$ when $I$ is a finite interval or a half-line}\label{sectinthalf}

Since $\underline{\Gamma}^{\rm D}$ is a union of finite intervals and half-lines, we pay specific attention to these two cases. 

If $I$ is a finite interval or a half line, we first define the set $\widetilde{L}^1_{\rm loc}(I)$ of $L^1_{\rm loc}(I)$ functions which are the restriction of $L^1(\RR)$ functions (in particular,  $\widetilde{L}^1_{\rm loc}(I)=L^1(I)$ if $I$ is finite),
$$
\widetilde{L}^1_{\rm loc}(I):=\{ f\in L^1_{\rm loc}(I), \exists \widetilde{f}\in L^1_{\rm loc}(\RR), \widetilde{f}_{\vert_I}=f\},
$$
and we introduce the following generalization of the homogeneous fractional Sobolev space of order $1/2$ introduced in \eqref{defH12hom} on the full line, 
$$
{H}^{1/2}_{\rm hom}(I)=\{ f\in \widetilde{L}^1_{\rm loc}(I),\quad \vert f \vert_{{{H}^{1/2}_{\rm hom}(I)}} <\infty\},
$$
with
\begin{equation}\label{defH12homI}
\vert f \vert_{H^{1/2}_{\rm hom}(I)} = \Big(\int_{I}\int_{I} \frac{\vert f(x)-f(y)\vert^2}{\vert x-y\vert^2}{\rm d}x{\rm d}y\Big)^{1/2}.
\end{equation}

If $I$ is bounded,  then we define $\dot{H}^{1/2}(I)$ as ${H}^{1/2}_{\rm hom}(I)$, but if $I$ is unbounded, we have to add a screening effect, consistently with \eqref{defH12screen}. Note that our definition differs slightly from the one used in \cite{LeoniTice}, where $L^1_{\rm loc}$ is used rather than $\widetilde{L}^1_{\rm loc}$; the reason for this change is that we later need to define the average $\overline{f}=\frac{1}{\vert I \vert}\int_I f$ when $I$ is a finite interval.
\begin{definition}\label{defH12I}
If $I$ is a finite interval or a half-line, then we define
$$
\dot{H}^{1/2}(I)=\{f\in \widetilde{L}_{\rm loc}^1(I),\vert f\vert_{\dot{H}^{1/2}(I)}<\infty\},
$$
where:\\
- if $I$ is a finite interval, one has $\vert f\vert_{\dot{H}^{1/2}(I)}=\vert f\vert_{{{H}_{\rm hom}^{1/2}(I)}}$ as defined in \eqref{defH12homI};\\
- If $I$ is half-line, one takes
$$
\vert f\vert_{\dot{H}^{1/2}(I)}=\Big( \int_I \int_{I\cap B(x,1)} \frac{(f(y)-f(x))^2}{(y-x)^2}{\rm d}y {\rm d}x\Big)^{1/2}.
$$
\end{definition}
As in Remark \ref{remH12screen}, it is possible to replace the constant $1$ used to impose the screening by any constant $\rho_0>0$. This is a consequence of the following proposition.
\begin{proposition}\label{propequivnorms}
Let $I$ be a non-empty finite interval or a half-line. Let $\rho_0>0$ and denote, for all $f\in \widetilde{L}^1_{\rm loc}(I)$, 
$$
\vert f\vert_{\dot{H}_{\rho_0}^{1/2}(I)}=\Big( \int_I \int_{I\cap B(x,\rho_0)} \frac{(f(y)-f(x))^2}{(y-x)^2}{\rm d}y {\rm d}x\Big)^{1/2}.
$$
There exists a constant $C$ (depending on $\rho_0$) such that for all $f\in \widetilde{L}^1_{\rm loc}(I)$, one has
$$
\frac{1}{C} \vert f\vert_{\dot{H}^{1/2}(I)} \leq \vert f\vert_{\dot{H}_{\rho_0}^{1/2}(I)} \leq C \vert f\vert_{\dot{H}^{1/2}(I)}.
$$
\end{proposition}
\begin{proof}
Clearly, if $\rho_0'>\rho_0>0$, then $\vert f\vert_{\dot{H}_{\rho_0}^{1/2}(I)}\leq \vert f\vert_{\dot{H}_{\rho'_0}^{1/2}(I)}$. Therefore, proceeding as in \cite{LeoniTice}, it is sufficient to prove that for all $\rho_0>0$, there exists a constant $C>0$ such that for all $f\in L^1_{\rm loc}(I)$, one has
$\vert f\vert_{\dot{H}_{2\rho_0}^{1/2}(I)}\leq C \vert f\vert_{\dot{H}_{\rho_0}^{1/2}(I)}$. 

Remarking that
\begin{align*}
\vert f\vert_{\dot{H}_{2\rho_0}^{1/2}(I)}^2&=\vert f\vert_{\dot{H}_{\rho_0}^{1/2}(I)}^2+ \int_I \int_{I\cap \big( B(x,2\rho_0)\backslash B(x,\rho_0)\big)} \frac{(f(y)-f(x))^2}{(y-x)^2}{\rm d}y {\rm d}x\\
&=\vert f\vert_{\dot{H}_{\rho_0}^{1/2}(I)}^2+ J,
\end{align*}
we just need to bound $J$ from above by $\vert f\vert_{\dot{H}_{\rho_0}^{1/2}(I)}^2$, up to a multiplicative constant of no importance. 

The difference with \cite{LeoniTice} is that we work on a domain with boundaries (and in particular not translation invariant). Writing $I=(a,b)$, the set $K_{\rho_0}$ of all $x\in I$ and $h\in \RR$ such that $y=x+h\in I\cap\big(B(x,2\rho_0)\backslash B(x,\rho_0)\big)$ can be written as
$$
K_{\rho_0}=\{ x\in I, h\in \RR, \rho_0< h  < \min\{   2\rho_0,b-x\}\mbox{ or }   \rho_0< -h  < \min\{   2\rho_0,x-a\}    \}.
$$
In particular, one has $K_{\rho_0}=\emptyset$ if $b-a\leq \rho_0$, in which case there is nothing to prove. We thus consider the case $\rho_0<b-a$.  Let us introduce also the set $\widetilde{K}_{\frac{\rho_0}{2}}$ by $(x,h)\in \widetilde{K}_{\frac{\rho_0}{2}}$ if and only if $(x,2h)\in K_{\rho_0}$.

We can therefore write
\begin{align*}
J
&= \int\int_{{K}_{\rho_0}}  \frac{(f(x+h)-f(x))^2}{h^2}{\rm d}h {\rm d}x\\
&= \frac{1}{2}\int\int_{\widetilde{K}_{\frac{\rho_0}{2}}}  \frac{(f(x+2h)-f(x))^2}{h^2}{\rm d}h  {\rm d}x,
\end{align*}
from which we infer that
\begin{align*}
J\leq &\int\int_{\widetilde{K}_{\frac{\rho_0}{2}}}  \frac{(f(x+2h)-f(x+h))^2}{h^2}{\rm d}h {\rm d}x
&+
\int\int_{\widetilde{K}_{\frac{\rho_0}{2}}}  \frac{(f(x+h)-f(x))^2}{h^2}{\rm d}h {\rm d}x  \\
&=:J_1+J_2.
\end{align*}

Let us first control $J_1$. One has
\begin{align*}
J_1=&\int_{-\rho_0}^{-\frac{\rho_0}{2}}\int_{\min\{a-2h,b\}}^b    \frac{(f(x+2h)-f(x+h))^2}{h^2}{\rm d}x {\rm d}h\\
&+\int_{\frac{\rho_0}{2}}^{ \rho_0}\int_a^{\max\{b-2h,a\}}    \frac{(f(x+2h)-f(x+h))^2}{h^2}{\rm d}x {\rm d}h.
\end{align*}
Remarking further that for all $h\in (-\rho_0,-\frac{1}{2}\rho_0 )$ and $x\in (\min\{a-2h,b\},b)$, one has $a<x+h<b$, and that the same conclusion holds similarly for all $h\in (\frac{1}{2}\rho_0,\rho_0)$ and $x\in (a,\max\{b-2h,a\})$, one deduces after performing the change of variable $x'=x+h$ that
$$
J_1\leq \int_I \int_{\vert h\vert <\rho_0,  x'+h\in I}\frac{(f(x'+h)-f(x'))^2}{h^2}{\rm d}h {\rm d}x'=\vert f \vert_{\dot{H}^{1/2}_{\rho_0}(I)}^2,
$$ 
which is the desired control. 

For $J_2$, we just need to remark that if $(x,h)\in \widetilde{K}_{\frac{\rho_0}{2}}$ then $x\in I$ and $y=x+h$ is in $I\cap B(x,\rho_0)$, so that the same upper bound holds. This concludes the proof.
\end{proof}

We provide here a Poincar\'e inequality that will prove useful. We do not provide the proof since it can be deduced from  more general properties for screened spaces (Theorem 3.5 in \cite{LeoniTice}). 
\begin{proposition}\label{propPoincare}
Let $I\subset \RR$ be an open non-empty finite interval, and let $\psi \in \dot{H}^{1/2}(I)$. Then one has $\psi-\frac{1}{\vert I\vert}\int_I {\psi} \in L^2(I)$, where $\vert I\vert$ denotes the length of $I$,  and
$$
\big\vert \psi-\frac{1}{\vert I\vert}\int_I {\psi}\big\vert_{L^2(I)}\leq C \vert \psi \vert_{\dot{H}^{1/2}(I)},
$$
for some constant $C>0$ independent of $\psi$.
\end{proposition}
\begin{remark}\label{remHdot12H12}
It follows easily from this proposition that if $I$ is an open finite non-empty interval then $\dot{H}^{1/2}(I)=H^{1/2}(I)$; this identity is purely algebraical in the sense that the canonical semi-norm of $\dot{H}^{1/2}(I)$ is obviously not equivalent to the canonical norm of $H^{1/2}(I)$.
\end{remark}

\subsection{The spaces $\dot{H}^{1/2}(\Gamma^{\rm D})$ and $\dot{H}^{s}(\Gamma^{\rm D})$ ($s\geq 1$)}\label{sectdotH12}

We are now ready to define the topological space $\dot{H}^{1/2}(\Gamma^{\rm D})$. 

We recall that according to Assumption \ref{assconfig}, one can write $\Gamma^{\rm D}=\underline{\Gamma}^{\rm D}\times \{0\}$, with $\underline{\Gamma}^{\rm D}=\bigcup_{j=1}^{N+1}{\mathcal E}_j$ where the ${\mathcal E}_j$ are either finite intervals or half-lines. It seems natural at first sight to define $\dot{H}^{1/2}(\Gamma^{\rm D})$ as the set $\prod_{j=1}^{N+1}\dot{H}^{1/2}(\cE_j)$ of functions whose restriction to the $j$-th component of $\Gamma^{\rm D}$ belongs to $\dot{H}^{1/2}({\mathcal E}_j)$ for all $1\leq j\leq N+1$ (with $\dot{H}^{1/2}({\mathcal E}_j)$ as defined in Definition \ref{defH12I}). 

However, the example a horizontal strip with Dirichlet data on both boundaries  suggests that there might be an additional condition relating the different connected components of $\Gamma^{\rm D}$, see \eqref{H12strip}. 

This turns out to be the case; this condition involves averages of $\psi$ on the various connected components of $\Gamma^{\rm D}$, as introduced in the definition below. In the statement, we denote by $\vert I\vert$ the length of a finite interval $I$.

Generalizing the notation used on a finite interval or a half-line, we also set
$$
\widetilde{L}^1_{\rm loc}(\Gamma^{\rm D}):=\{ f\in L^1_{\rm loc}(\Gamma^{\rm D}), \exists \widetilde{f}\in L^1_{\rm loc}(\RR), \widetilde{f}_{\vert_{\Gamma^{\rm D}}}=f\}.
$$

\begin{definition}\label{propaverage}
Let $\Gamma^{\rm D}=\underline{\Gamma}^{\rm D}\times \{0\}$, with $\underline{\Gamma}^{\rm D}=\bigcup_{j=1}^{N+1}{\mathcal E}_j$ be as in Assumption \ref{assconfig}. For all $\psi \in \widetilde{L}^1_{\rm loc}({\Gamma}^{\rm D})$, we write $\psi=(\psi_1,\dots,\psi_{N+1})$ where for all $1\leq j\leq N+1$, $\psi_j=\psi_{\vert_{{\mathcal E}_j\times\{0\}}}$ is identified with a function on $ {\mathcal E}_j$. \\
{\bf i.} If $\cE_j$ is finite, then we define $\overline{\psi}_j:=\frac{1}{\vert \cE_j\vert}\int_{{\mathcal E}_j} \psi_j$;\\
{\bf ii.} If  ${\mathcal E}_1=(-\infty,x_1^{\rm l})$,  then we define $\overline{\psi}_1:=\frac{1}{\vert I_1\vert}  \int_{I_1}\psi_1$ where $I_1=(a_1,b_1)$, with $ -\infty < a_1 < b_1  \leq  x_1^{\rm l}$ ;\\
{\bf iii.} If  ${\mathcal E}_{N+1}= (x_N^{\rm r},\infty)$, then we define $\overline{\psi}_{N+1}:= \frac{1}{\vert I_{N+1}\vert} \int_{I_{N+1}}\psi_{N+1}$, with $x_N^{\rm r}\leq a_{N+1}< b_{N+1}<\infty$. 
\end{definition} 
We can now give the definition of the space $\dot{H}^{1/2}(\Gamma^{\rm D})$.
\begin{definition}\label{defH12Gamma}
Let $\Gamma^{\rm D}=\underline{\Gamma}^{\rm D}\times \{0\}$, with $\underline{\Gamma}^{\rm D}=\bigcup_{j=1}^{N+1}{\mathcal E}_j$ be as in Assumption \ref{assconfig}.
For all $\psi \in \widetilde{L}^1_{\rm loc}({\Gamma}^{\rm D})$, we write $\psi=(\psi_1,\dots,\psi_{N+1})$ where $\psi_j=\psi_{\vert_{{\mathcal E}_j\times\{0\}}}$ is identified with a function on $ {\mathcal E}_j$, and we also denote $\overline{\psi}_j$ the constants introduced in Definition \ref{propaverage}.
We then define $\dot{H}^{1/2}(\Gamma^{\rm D})$ as
$$
\dot{H}^{1/2}(\Gamma^{\rm D})=\{ \psi \in \widetilde{L}^1_{\rm loc}(\Gamma^{\rm D}), \forall 1\leq j\leq N+1, \psi_j\in \dot{H}^{1/2}({\mathcal E}_j)\},
$$
endowed with the semi-norm
$$
\vert \psi\vert_{\dot{H}^{1/2}(\Gamma^{\rm D})}=\sum_{j=1}^{N+1} \vert \psi_j\vert_{\dot{H}^{1/2}({\mathcal E}_j)}+\sum_{j=1}^N \vert \overline{\psi}_{j+1}-\overline{\psi}_j\vert.
$$
\end{definition}
\begin{remark}\label{remH12screen3}
Recalling that $\cE_j=(x_{j-1}^{\rm r},x_j^{\rm l})$ and introducing the interval $H_j(x)=(\max\{-\rho_0,x_{j-1}^{\rm r}-x\}, \min\{\rho_0,x_j^{\rm l}-x\})$ for some $\rho_0>0$,
we deduce from Proposition \ref{propequivnorms} that we can use the equivalent semi-norm
$$
\vert \psi\vert_{\dot{H}^{1/2}(\Gamma^{\rm D})}=\sum_{j=1}^{N+1} \Big(\int_{{\mathcal E}_j} \int_{H_j(x)} \frac{(\psi_j(x+h)-\psi_j(x))^2}{h^2}{\rm d}h {\rm d}x\Big)^{1/2}+\sum_{j=1}^N \vert \overline{\psi}_{j+1}-\overline{\psi}_j\vert.
$$
\end{remark}
\begin{remark}
Note that the adherence of zero in  $\prod_{j=1}^{N+1}\dot{H}^{1/2}({\mathcal E}_j)$ endowed with its canonical semi-norm is given by those functions $\psi$ which are constant in each ${\mathcal E}_{j}$ and can therefore be identified with $\RR^{N+1}$. Now, the adherence of zero for the $\dot{H}^{1/2}(\Gamma^{\rm D})$ semi-norm consists of those functions $\psi$ which are equal to some constant which has to be the same on each ${\mathcal E}_{j}$; it can therefore be identified with $\RR$. This shows that these two spaces are  topologically not the same.
\end{remark}


It looks like the definition of the topological space $\dot{H}^{1/2}(\Gamma^{\rm D})$ depends on the choice of the interval $I$ on which one takes the average on unbounded connected components of $\Gamma^{\rm D}$ in Definition \ref{propaverage}; the following proposition shows that this is not the case.
\begin{proposition}\label{propconstequiv}
Let $\Gamma^{\rm D}=\underline{\Gamma}^{\rm D}\times \{0\}$, with $\underline{\Gamma}^{\rm D}=\bigcup_{j=1}^{N+1}{\mathcal E}_j$ be as in Assumption \ref{assconfig}. \\
If $\cE_1$ (resp. $\cE_{N+1}$) is unbounded, then the topological space $\dot{H}^{1/2}(\Gamma^{\rm D})$ introduced in Definition \ref{defH12Gamma} is independent of the choice of the finite interval $I$ chosen to define $\overline{\psi}_1$ (resp.  $\overline{\psi}_{N+1}$) in Definition \ref{propaverage}.\end{proposition}
\begin{proof}
Let $j=1$ or $N+1$ and let $I_1=(a_1,b_1)$ and $I_2=(a_2,b_2)$ with $-\infty <a_k <b_k<x_1^{\rm l}$ if $j=1$ and $x_N^{\rm r}<a_k <b_k<\infty$ if $j=N+1$ ($k=1,2$).

The proposition is a direct consequence of the fact that
$$
\big\vert \frac{1}{\vert I_1\vert}\int_{I_1} \psi - \frac{1}{\vert I_2\vert}\int_{I_2}\psi \big\vert \leq C \vert \psi_{\vert_{\cE_{j}}}\vert_{\dot{H}^{1/2}(\cE_{j})},
$$
with $C>0$ independent of $\psi$. 
We therefore prove that this inequality holds. 

Since 
$$
\frac{1}{\vert I_1\vert} \int_{I_1} \psi -\frac{1}{\vert I_2\vert} \int_{I_2}\psi = \frac{1}{\vert I_1\vert \, \vert I_2\vert}\Big( \int_{I_2}(\int_{I_1} \psi(x) {\rm d}x){\rm d}y
 - \int_{I_1}(\int_{I_2} \psi(y) {\rm d}y){\rm d}x \Big),
$$
it follows that
\begin{align*}
\big\vert \frac{1}{\vert I_1\vert} \int_{I_1} \psi &-\frac{1}{\vert I_2\vert} \int_{I_2}\psi \big\vert \leq \frac{1}{| I_1| \, | I_2|} \int_{I_1}\int_{I_2} \vert \psi(y)-\psi(x)\vert {\rm d}x {\rm d}y \\
&\leq \frac{1}{| I_1| | I_2|} \big( \int_{I_1}\int_{I_2} (y- x )^2  {\rm d}x {\rm d}y\big)^{1/2} \Big( \int_{I_1}\int_{I_2}  \frac{(\psi_j(y)-\psi_j(x))^2}{(y-x)^2}{\rm d}x {\rm d}y \Big)^{1/2},
\end{align*}
and therefore
$$
\big\vert \frac{1}{\vert I_1\vert} \int_{I_1} \psi -\frac{1}{\vert I_2\vert} \int_{I_2}\psi \big\vert  \lesssim \Big( \int_{\cE_j}\int_{\vert y-x \vert \leq \rho_0}  \frac{(\psi_j(y)-\psi_j(x))^2}{(y-x)^2}{\rm d}x {\rm d}y \Big)^{1/2},
$$
with $\rho_0:=\sup_{x\in I_1, y\in I_2}{\vert y-x\vert }<\infty$; together with Remark \ref{remH12screen3}, this proves the desired result.  
\end{proof}
We now provide a proposition that shows that the standard Sobolev space $H^{1/2}(\Gamma^{\rm D})$ coincides algebraically and topologically with the space $L^2(\Gamma^{\rm D})\cap \dot{H}^{1/2}(\Gamma^{\rm D})$ endowed with its canonical norm.
\begin{proposition}\label{propequivter}
One has $L^2(\Gamma^{\rm D})\cap \dot{H}^{1/2}(\Gamma^{\rm D})=H^{1/2}(\Gamma^{\rm D})$, and moreover there exists a constant $C>0$ such that for all function $\psi$ in $H^{1/2}(\Gamma^{\rm D})$, one has
$$
\frac{1}{C}\vert \psi \vert_{H^{1/2}(\Gamma^{\rm D})}\leq \vert \psi\vert_{L^2(\Gamma^{\rm D})}+\vert \psi \vert_{\dot{H}^{1/2}(\Gamma^{\rm D})}
\leq C\vert \psi\vert_{H^{1/2}(\Gamma^{\rm D})}.
$$
\end{proposition}
\begin{proof}
For the left side of the inequality, one has by the definition of $H^{1/2}(\Gamma^{\rm D})$ that
\[
\begin{split}
\vert \psi \vert^2_{H^{1/2}(\Gamma^{\rm D})}=&\sum_{j=1}^{N+1}\int_{\mathcal E_j}\vert \psi_j\vert^2dx+\sum_{j,k=1}^{N+1}\int_{\mathcal E_j}\int_{\mathcal E_k}\frac{\vert\psi_j(x)-\psi_k(y)\vert^2}{\vert x-y\vert^2}dxdy\\
=&\sum_{j=1}^{N+1}\int_{\mathcal E_j}\vert \psi_j\vert^2dx+\sum_{ j\neq k, \,j,k=1}^{N+1}\int_{\mathcal E_j}\int_{\mathcal E_k}\frac{\vert\psi_j(x)-\psi_k(y)\vert^2}{\vert x-y\vert^2}dxdy\\
&+\sum_{j=1}^{N+1}\int_{\mathcal E_j}\int_{\mathcal E_j\cap B(x,1)^c}\frac{\vert\psi_j(x+h)-\psi_j(x)\vert^2}{\vert h\vert^2}dhdx\\
&+\sum_{j=1}^{N+1}\int_{\mathcal E_j}\int_{\mathcal E_j\cap B(x,1)}\frac{\vert\psi_j(x+h)-\psi_j(x)\vert^2}{\vert h\vert^2}dhdx.
\end{split}
\] 
As a result, the first three terms above can be controlled by $\vert \psi\vert^2_{L^2(\Gamma^{\rm D})}$, and the last one can be controlled directly by $\vert \psi \vert^2_{\dot{H}^{1/2}(\Gamma^{\rm D})}$. 

On the other hand, checking directly Definition \ref{defH12Gamma}, one knows immediately that $\vert \psi\vert_{\dot H^{1/2}(\Gamma^{\rm D})}$ can be controlled by $\vert \psi\vert_{H^{1/2}(\Gamma^{\rm D})}$, so the right side of the desired inequality holds. 
\end{proof}

We finally introduce, for $s\geq 1$, the homogeneous space $\dot{H}^s(\Gamma^{\rm D})$, which cannot be identified topologically with $\prod_{j=1}^{N+1} \dot{H}^s(\cE_j)$, with  $\dot{H}^{s}(\cE_j):=\{\psi \in \widetilde{L}^1_{\rm loc}(\cE_j), \partial_x \psi \in H^{s-1}(\cE_j)\} $ because of the additional terms involving the $\overline{\psi}_j$ in the definition of the semi-norm.

\begin{definition}\label{defdotH1Gamma}
Let $s\geq 1$. With the assumption and notations of Definition \ref{defH12Gamma} we define $\dot{H}^{s}(\Gamma^{\rm D})$ as
$$
\dot{H}^{s}(\Gamma^{\rm D})=\{ \psi \in \widetilde{L}^1_{\rm loc}(\Gamma^{\rm D}), \forall 1\leq j\leq N+1, \psi_j\in \dot{H}^{s}({\mathcal E}_j)\}, 
$$
endowed with the semi-norm
$$
\vert \psi\vert_{\dot{H}^{s}(\Gamma^{\rm D})}=\vert \dx \psi\vert_{H^{s-1}(\Gamma^{\rm D})}+\sum_{j=1}^N \vert \overline{\psi}_{j+1}-\overline{\psi}_j\vert.
$$
\end{definition} 

We will use the following proposition which shows that $\dot{H}^s(\Gamma^{\rm D})$ can be identified algebraically and topologically with $\dot{H}^{1/2}(\Gamma^{\rm D})\cap \big( \prod_{j=1}^{N+1} \dot{H}^s(\cE_j)\big)$. 
\begin{proposition}\label{characdotH1}
Let $s\geq 1$ and $\Gamma^{\rm D}=\underline{\Gamma}^{\rm D}\times \{0\}$, with $\underline{\Gamma}^{\rm D}=\bigcup_{j=1}^{N+1}{\mathcal E}_j$ be as in Assumption \ref{assconfig}. Then one has  $\dot{H}^s(\Gamma^{\rm D})=\dot{H}^{1/2}(\Gamma^{\rm D})\cap \big( \prod_{j=1}^{N+1} \dot{H}^s(\cE_j)\big)$ and moreover, there exists a constant $C>0$ such that for all $\psi \in \dot{H}^1(\Gamma^{\rm D})$, one has
$$
\frac{1}{C} \vert \psi\vert_{\dot{H}^s(\Gamma^{\rm D})}\leq \vert \psi\vert_{\dot{H}^{1/2}(\Gamma^{\rm D})}+\vert \dx \psi\vert_{H^{s-1}(\Gamma^{\rm D})}\leq C  \vert \psi\vert_{\dot{H}^s(\Gamma^{\rm D})}.
$$
\end{proposition} 
\begin{proof}
The only thing we need to prove is that for $\rho_0$ small enough and all $1\leq j\leq N+1$, one has
$$
\Big(\iint_{x,y \in \cE_j, \vert y-x\vert \leq \rho_0} \frac{\big(\psi(y)-\psi(x)\big)^2}{(y-x)^2}{\rm d}y{\rm d}x \Big)^{1/2}\lesssim \vert \dx \psi \vert_{L^2(\cE_j)},
$$
which is a classical consequence of Hardy's inequality.
\end{proof}

\subsection{The trace theorem on $\dot{H}^1(\Omega)$}\label{secttracecont}

We prove here the following theorem which shows that the trace on $\Gamma^{\rm D}$ of functions in $\dot{H}^1(\Omega)$ belongs to $\dot{H}^{1/2}(\Gamma^{\rm D})$, and the mapping is continuous for the corresponding semi-norms.
\begin{theorem}\label{theortrace}
Let $\Omega$ and $\Gamma^{\rm D}$ be as in Assumption \ref{assconfig}. The trace mapping ${\rm Tr}^{\rm D}$ defined in \eqref{tracemap} takes its values in $\dot{H}^{1/2}(\Gamma^{\rm D})$ and moreover ${\rm Tr}^{\rm D}: \dot{H}^1(\Omega) \to \dot{H}^{1/2}(\Gamma^{\rm D})$ is continuous. 
\end{theorem}

\begin{proof}
Let us consider $\phi\in \dot{H}^{1}(\Omega)$. We want to show that its trace on $\Gamma^{\rm D}$, which is well defined in $L^2_{\rm loc}(\Gamma^{\rm D})$, is in $\dot{H}^{1/2}(\Gamma^{\rm D})$ and that there is a constant $C>0$ independent of $\phi$ such that
$$
\vert \phi_{\vert_{\Gamma^{\rm D}}}\vert^2_{\dot{H}^{1/2}(\Gamma^{\rm D})}\leq C \Vert \nabla \phi\Vert_{L^2(\Omega)}^2.
$$

From Remark \ref{remH12screen3} and with $H_j(x)$ as defined there, it is equivalent to prove that for all $1\leq j\leq N+1$
\begin{equation}\label{reformnorm}
\int_{\cE_j}\int_{H_j(x)} \frac{(\phi(x+h,0)-\phi(x,0))^2}{\vert h\vert^2}{\rm d}h{\rm d}x \leq C \Vert \nabla \phi\Vert_{L^2(\Omega)}^2,
\end{equation} 
and that furthermore
\begin{equation}\label{reformnormadd}
\sum_{j=1}^N \vert \overline{\phi}_{j+1}-\overline{\phi}_{j}\vert \leq C \Vert \nabla \phi\Vert_{L^2(\Omega)};
\end{equation}
in addition, we need to prove that 
 $ \phi_{\vert_{\Gamma^{\rm D}}} \in \widetilde{L}^1_{\rm loc}(\Gamma^{\rm D})$ (and not only in $  {L}^1_{\rm loc}(\Gamma^{\rm D})$).

We start by proving \eqref{reformnorm}. It is sufficient to prove that
$$
\int_{\cE_j}\int_{H_j^\pm(x)} \frac{(\phi(x+h,0)-\phi(x,0))^2}{\vert h\vert^2}{\rm d}h{\rm d}x \leq C \Vert \nabla \phi\Vert_{L^2(\Omega)}^2,
$$
with $H_j(x)^\pm=H_j(x)\cap \RR^\pm$. We only consider the case of $H_j(x)^+$ here, since the case of $H_j^-(x)$ is treated similarly.

Let us choose $c_0>0$ small enough such that for all $x\in \cE_j$ and $h\in H_j(x)^+$, the square triangle with vertices $(x,0)$, $(x+h,0)$ and $(x+\frac{c_0^2}{1+c_0^2}h,-\frac{c_0}{1+c_0^2}h)$ is contained in ${\Omega}$. 

As in \cite{LeoniTice}, the key point is to be able to write the difference of the values of $\phi$ as an integral of its derivative. 
 It is possible \cite{Morrey,DenyLions} to find a function $\dot{\phi}$ which is almost everywhere equal to $\phi$ and that is absolutely continuous on almost every line parallel to the orthogonal axis $(O{\bf e})$ and $(O{\bf e}^\perp)$, with ${\bf e}=\frac{c_0}{\sqrt{1+c_0^2}}(1,-1/c_0)^{\rm T}$, and whose usual derivative coincides almost everywhere with its derivative in the sense of distributions. Identifying $\phi$ with $\dot{\phi}$ for the sake of clarity, it follows that for almost all $x\in \cE_j$ and $h\in H_j^+(x)$, one has
\begin{align*}
\phi(x+h,0)-\phi(x,0)=&\int_0^{ c_1 h} (\dx-\frac{1}{c_0}\dz)\phi(x+h',-\frac{1}{c_0}h'){\rm d}h'\\
&+\int_0^{(1-c_1)h} (\dx+c_0\dz)\phi(x+ c_1h+h',-\frac{c_1}{c_0}h+c_0 h'){\rm d}h',
\end{align*}
with $c_1=  \frac{c_0^2}{1+c_0^2}h $. Therefore one obtains
\begin{align}
\nonumber
\frac{\big(\phi(x+h,0)-\phi(x,0)\big)^2}{h^2}\lesssim  \Big(\frac{1}{h}\int_0^{ c_1 h} {\bf e}\cdot \nabla\phi(x+h',-\frac{1}{c_0}h'){\rm d}h'\Big)^2\\
\label{eqnorm1}
+\Big(\frac{1}{h}    \int_0^{(1-c_1)h} {\bf e}^\perp \cdot \nabla \phi(x+ c_1h+h',-\frac{c_1}{c_0}h+c_0 h'){\rm d}h'   \Big)^2.
\end{align}

We now need the following lemma.
\begin{lemma}
Let $u\in L^2(\Omega)$ and $c_0$, $c_1$, $\cE_j$ and $H_j^+(x)$ as above. Then the following upper bounds hold
$$
\Big[\int_{\cE_j}\int_{H_j^+(x)} \Big(\frac{1}{h}\int_0^{ c_1 h} u(x+h',-\frac{1}{c_0}h'){\rm d}h'\Big)^2{\rm d} h {\rm d}x\Big]^{1/2}\leq C \Vert u \Vert_{L^2(\Omega)},
$$
and
$$
\Big[\int_{\cE_j}\int_{H_j^+(x)}\Big(\frac{1}{h}\int_0^{(1-c_1)h} u(x+c_1 h+h',-\frac{c_1}{c_0}h+c_0 h'){\rm d}h'\Big)^2{\rm d} h {\rm d}x\Big]^{1/2}\leq C \Vert u \Vert_{L^2(\Omega)},
$$
for some constant $C>0$ independent of $u$.
\end{lemma} 
\begin{proof}
For the first point, let us first rewrite
$$
\int_0^{c_1 h} u(x+h',-\frac{1}{c_0}h'){\rm d}h'=c_1 \int_0^{h} u(x+c_1 h',-\frac{c_1 }{c_0}h'){\rm d}h';
$$
we then integrate over $h\in H_j^+(x)$ and use the fact that the mapping $v\mapsto \frac{1}{h}\int_0^h v $ is bounded on $L^2(H_j^+(x))$ (Hardy's inequality) to obtain that the left-hand side of the first inequality of the lemma is bounded from above by
$$
\Big[\int_{\cE_j}\int_{H_j^+(x)}  u(x+c_1 h,-\frac{c_1}{c_0}h)^2{\rm d} h {\rm d}x\Big]^{1/2}
$$
(up to a multiplicative constant of no importance),
which is itself bounded from above by $\Vert u \Vert_{L^2(\Omega)}$ since $(x+c_1 h,-\frac{c_1}{c_0}h)$ belongs to $\Omega$ for all $x\in \cE_j$ and $h\in H_j^+(x)$. This completes the proof of the first point of the lemma.

For the second point of the lemma, the same proof does not work because of the dependence on $h$ in the arguments of $u$. Let us remark that 
\[
\begin{split}
&\frac{1}{h}\int_0^{(1-c_1)h} u(x+c_1 h+h',-\frac{c_1}{c_0}h+c_0h'){\rm d}h'\\
&=(1-c_1)\int_0^1 u\big(x+ h (c_1+(1-c_1)s),h \frac{c_1}{c_0}(s-1) \big){\rm d}s.    
\end{split}
\]

Using Minkowski's inequality, we obtain therefore the following upper bound for the left-hand side of the second inequality of the lemma,
\begin{equation}\label{upperb}
\int_0^1 \Big[ \int_{\cE_j}\int_{H_j^+(x)}  u\big(x+ h (c_1+(1-c_1)s),h \frac{c_1}{c_0}(s-1) \big)^2 {\rm d}h{\rm d}x \Big]^{1/2}{\rm d}s
\end{equation} 
(up to a multiplicative constant of no importance). 

If we denote by ${\mathcal O}_j$ the open subset of $\RR^2$ defined as
$$
{\mathcal O}_j:=\{(x,h)\in \RR^2,\quad  x\in \cE_j,  h\in H_j^+(x)\},
$$
one can check that the mapping 
$$
\Phi_s:\begin{array}{lcl}
{\mathcal O}_j& \to & \Omega\\
(x,h)&\mapsto &\big(  x+ h (c_1+(1-c_1)s),h \frac{c_1}{c_0}(s-1)  \big)
\end{array}
$$
is well defined (in the sense that it takes its values in $\Omega$) for all $0<s<1$. Moreover, the Jacobian determinant of $\Phi_s$ is $J(\Phi_s)=\frac{c_1}{c_0}(s-1)$. 

By a change of variables, we can therefore rewrite \eqref{upperb} under the form
$$
\int_0^1 \frac{\sqrt{c_0/c_1}}{\sqrt{1-s}}\Big[  \int_{\Phi_s({\mathcal O}_j)} u(x',z')^2{\rm d}x'{\rm d}z'\Big]^{1/2}{\rm d}s.
$$
Since $\Phi_{s}({\mathcal O}_j)$ is a subset of $\Omega$, one can bound this term from above by
$$
\big(\int_0^1 \frac{\sqrt{c_0/c_1}}{\sqrt{1-s}}{\rm d}s\big) \Vert u\Vert_{L^2(\Omega)},
$$
which proves the lemma because the integral in $s$ converges.
\end{proof}

It is now a direct consequence of \eqref{eqnorm1} and of the lemma that
$$
\int_{\cE_j} \int_{H_j^+(x)} \frac{\big(\phi(x+h,0)-\phi(x,0)\big)^2}{h^2} {\rm d}h{\rm d}x \leq C \Vert \nabla \phi\Vert_{L^2(\Omega)},
$$
which completes the proof of \eqref{reformnorm}.

\medskip
We now turn to prove \eqref{reformnormadd}. For all $1\leq j\leq N$, let $\xi_j\in \cE_j$, $\xi_{j+1}\in \cE_{j+1}$ and $\ell >0$ be such that:
\begin{itemize}
\item There exists an open broken line $L_j\subset  \Omega $ with a finite number of vertices, and with endpoints $(\xi_j,0)$ and $(\xi_{j+1},0)$ (see fig \ref{fig:brokenline});
\item The horizontal translation by $(0,l)$ of $L_j$ is contained in $\Omega$ for all $0\leq l\leq \ell$.
\item The segments of $L_j$ are either parallel to $(1,c)$ or $(-c,1)$ for some $c\in \RR$ (two adjacent segments are therefore orthogonal).
\end{itemize}
\begin{figure}[htbp]
                \centering
                \includegraphics[width=0.55\textwidth]{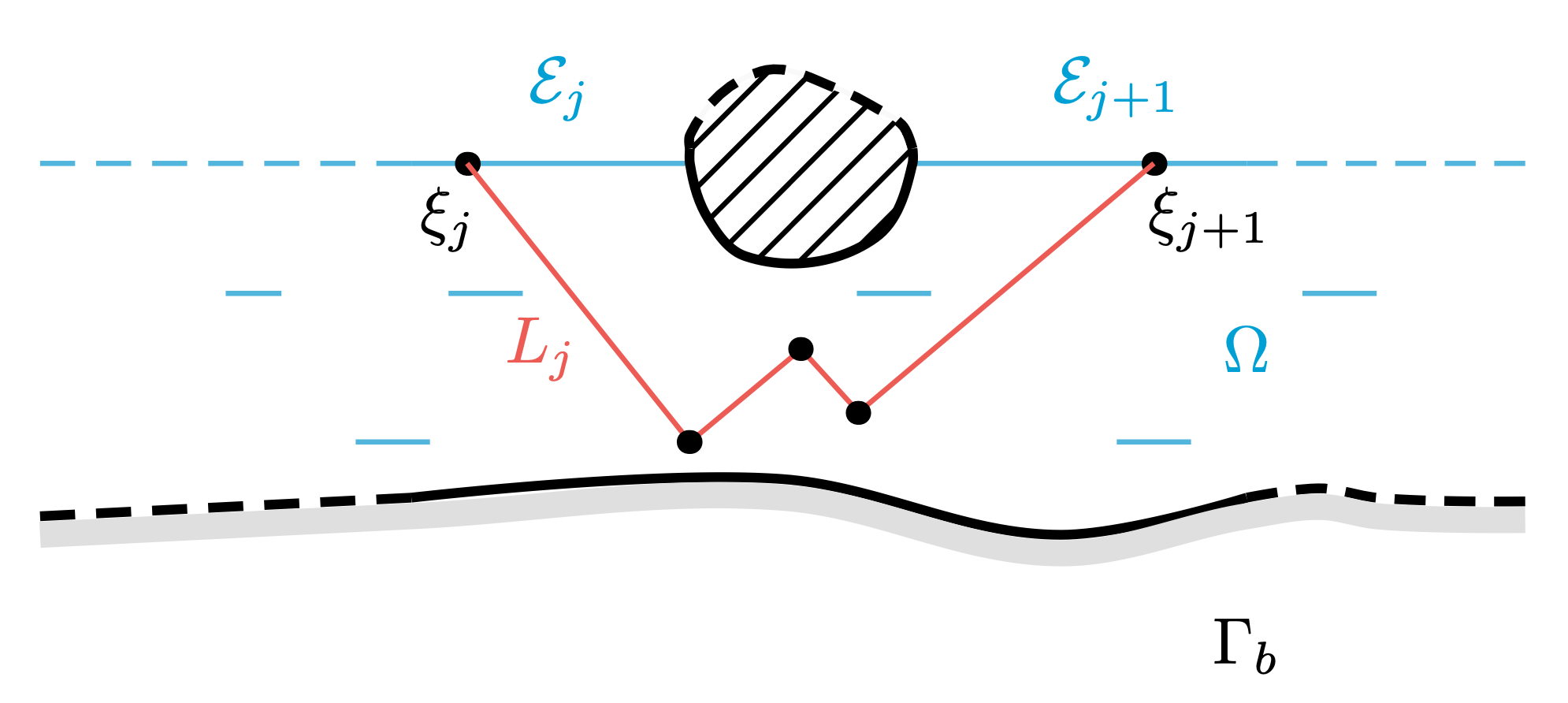}
                \caption{The broken line $L_j$}
        \label{fig:brokenline}
\end{figure}

Owing to  Proposition \ref{propconstequiv}, in order to prove \eqref{reformnormadd}, it is enough to prove that for all $1\leq j\leq N$, one has
\begin{equation}\label{reformnormaddbis}
\vert \bar{\bar{\phi}}_{j+1}-\bar{\bar{\phi}}_{j}\vert \leq C \Vert \nabla \phi\Vert_{L^2(\Omega)}.
\end{equation}
where we define
$$
\bar{\bar{\phi}}_{j}=\frac{1}{\ell} \int_{\xi_j}^{\xi_{j}+\ell}\phi(x,0){\rm d}x,\qquad
\bar{\bar{\phi}}_{j+1}=\frac{1}{\ell} \int_{\xi_{j+1}}^{\xi_{j+1}+\ell}\phi(x,0){\rm d}x.
$$

For the sake of simplicity, we assume that for each $1\leq j\leq N$, the broken line $L_j$ has only three vertices, namely, $(\xi_j,0)$, $(\xi_{j+1/2},-h_{j+1/2})$ and $(\xi_{j+1/2},0)$, where $\xi_{j+1/2}$ and $h_{j+1/2}$ are such that the triangle formed by these three points is rectangle at $(\xi_{j+1/2},-h_{j+1/2})$ (handling the case of more vertices does not raise any other difficulty).

We can therefore write
\begin{align*}
 \bar{\bar{\phi}}_{j+1}-\bar{\bar{\phi}}_j  =&
\frac{1}{\ell}  \int_0^\ell [\phi(\xi_j+x,0)- \phi(\xi_{j+1}+x,0) ] {\rm d} x\\
=&\frac{1}{\ell} \int_0^\ell [\phi(\xi_j+x,0)-\phi(\xi_{j+1/2}+x,-h_{j+1/2})]\\
&+\frac{1}{\ell} \int_0^\ell [\phi(\xi_{j+1/2}+x,-h_{j+1/2})- \phi(\xi_{j+1}+x,0)].
\end{align*}

Using as before the fact that $\phi$ is almost everywhere equal to a function that is absolutely continuous on almost all parallel lines to the two orthogonal segments of  $L_j$, we can rewrite this expression as for \eqref{eqnorm1} above. Since the difficulty raised by the division by $h^2$ in \eqref{eqnorm1} is not present here, it is straightforward to obtain that 
$$
\big\vert  \bar{\bar{\phi}}_{j+1}-\bar{\bar{\phi}}_j \big\vert \leq C \Vert \nabla \phi \Vert_{L^2(\Omega)}^2,
$$
which proves  \eqref{reformnormaddbis}. 

The last thing to prove is therefore that $ \phi_{\vert_{\Gamma^{\rm D}}} \in \widetilde{L}^1_{\rm loc}(\Gamma^{\rm D})$; it is enough to prove that if $I$ is a finite interval with at least one endpoint coinciding with a corner, then one has $ \phi_{\vert_{\Gamma^{\rm D}}} \in L^1(I)$. It is possible to find a triangle $T\subset \Omega$ which has one of its sides coinciding with $I$. In virtue of Theorem 6.2 in \cite{DenyLions}, $T$ is a Nikodym set, that is, every function in $\dot{H}^1(T)$ is in $H^1(T)$. It follows that the restriction of $\phi$ to $T$ is in $H^1(T)$ and that its trace on $I$ is in $H^{1/2}(I)$ by the standard trace theorem. It is therefore in $L^2(I)\subset L^1(I)$, so that the proof of the theorem is complete.
\end{proof}

\subsection{Surjectivity of the trace mapping ${\rm Tr}^{\rm D}$}\label{secttracesurj}

We have shown in the previous section that the mapping ${\rm Tr}^{\rm D}: \dot{H}^1(\Omega) \to \dot{H}^{1/2}(\Gamma^{\rm D})$ is well defined and continuous.
We show in this section that it is also  onto and admits a continuous right inverse. This relies on a characterization of $\dot{H}^{1/2}(\Gamma^{\rm D})$ of independent interest. We need to define first the space $\widetilde{H}^{1/2}(\Gamma^{\rm D})$.
\begin{definition}\label{deftildeH12}
Let $\Gamma^{\rm D}$ be as in Assumption \ref{assconfig}. We define the space $\widetilde{H}^{1/2}(\Gamma^{\rm D})$ as
$$
\widetilde{H}^{1/2}(\Gamma^{\rm D})= \{ \psi \in \widetilde{L}^1_{\rm loc}(\Gamma^{\rm D}), \exists \widetilde{\psi}\in \dot{H}^{1/2}(\RR), \widetilde{\psi}_{\vert_{\Gamma^{\rm d}}}=\psi\},
$$
endowed with the semi-norm 
$$
\vert \psi \vert_{\widetilde{H}^{1/2}(\Gamma^{\rm D})}=\inf_{\widetilde{\psi}\in \dot{H}^{1/2}(\RR), \widetilde{\psi}_{\vert_{\Gamma^{\rm d}}}=\psi} \vert \widetilde{\psi}\vert_{\dot{H}^{1/2}(\RR)}.
$$
\end{definition}
The following proposition shows that $\widetilde{H}^{1/2}(\Gamma^{\rm D})$  coincides algebraically and topologically with $\dot{H}^{1/2}(\Gamma^{\rm D})$.
\begin{proposition}\label{propH12char}
Let $\Gamma^{\rm D}$ be as in Assumption \ref{assconfig}, and let $\dot{H}^{1/2}(\Gamma^{\rm D})$ and $\widetilde{H}^{1/2}(\Gamma^{\rm D})$ be as in Definitions \ref{defH12Gamma} and \ref{deftildeH12} respectively. Then $\dot{H}^{1/2}(\Gamma^{\rm D})=\widetilde{H}^{1/2}(\Gamma^{\rm D})$, and there is $C>0$ such that for all $\psi \in \dot{H}^{1/2}(\Gamma^{\rm D})$, one has
$$
\frac{1}{C}\vert \psi \vert_{\widetilde{H}^{1/2}(\Gamma^{\rm D})} \leq \vert \psi\vert_{\dot{H}^{1/2}(\Gamma^{\rm D})} \leq C \vert \psi \vert_{\widetilde{H}^{1/2}(\Gamma^{\rm D})}.
$$
\end{proposition} 
\begin{proof}
Let us first prove that $\widetilde{H}^{1/2}(\Gamma^{\rm D})$ is continuously embedded in $\dot{H}^{1/2}(\Gamma^{\rm D})$. Let therefore $\psi \in \widetilde{H}^{1/2}(\Gamma^{\rm D})$. We need to prove that there is a constant $C_0$ independent of $\psi$ such that $\vert \psi \vert_{\dot{H}^{1/2}(\Gamma^{\rm D})}\leq C \vert \widetilde{\psi}\vert_{\widetilde{H}^{1/2}(\RR)}$, for all  $\widetilde{\psi}\in \dot{H}^{1/2}(\RR)$ such that $ \widetilde{\psi}_{\vert_{\Gamma^{\rm d}}}=\psi$. 

From the definition of the semi-norm of $\dot{H}^{1/2}(\Gamma^{\rm D})$, the only non trivial thing to prove is that for all $1\leq j\leq N$, one has $\vert \overline{\psi}_{j+1}-\overline{\psi}_j\vert \leq C  \vert \widetilde{\psi}\vert_{\widetilde{H}^{1/2}(\RR)}$. Since $\overline{\psi}_k=\frac{1}{\vert I \vert } \int_I \widetilde{\psi}_k$  with $I=\cE_k$ if $\cE_k$ is bounded, and $I$ is a finite subinterval of $\cE_k$ if $\cE_k$ is unbounded, the result follows from a straightforward adaptation of Proposition \ref{propconstequiv}. 

The proof that $\dot{H}^{1/2}(\Gamma^{\rm D})$ is continuously embedded in $\widetilde{H}^{1/2}(\Gamma^{\rm D})$ follows from the following lemma. 
\begin{lemma}\label{lemmaext}
Let $\Gamma^{\rm D}=\underline{\Gamma}^{\rm D}\times \{0\}$, with $\underline{\Gamma}^{\rm D}=\bigcup_{j=1}^{N+1}{\mathcal E}_j$ be as in Assumption \ref{assconfig}. With the notations of Definition \ref{defH12Gamma}, there exists a constant $C>0$ such that for all $\psi=(\psi_1,\dots,\psi_{N+1}) \in \dot{H}^{1/2}(\Gamma^{\rm D})$, there exists $\widetilde{\psi}\in \dot{H}^{1/2}(\RR)$ such that for all $1\leq j\leq N+1$ one has  $\widetilde{\psi}_{\vert_{{\mathcal E}_j}}=\psi_j$ and moreover
$$
\vert \widetilde{\psi}\vert_{\dot{H}^{1/2}(\RR)}\leq C \vert \psi\vert_{\dot{H}^{1/2}(\Gamma^{\rm D})}.
$$
\end{lemma}
\begin{proof}[Proof of the lemma]
For the sake of clarity, we consider here the configuration of Figure \ref{fig:image3} for which $N=1$ and ${\mathcal E}_1=(x^{\rm r}_0,x_1^{\rm l})$ and ${\mathcal E}_2=(x_1^{\rm r},+\infty)$ with $x_0^{\rm r}>-\infty$. This configuration contains all the difficulties (emerging bottom and unbounded domain). The general case does not raise additional difficulty. The main idea of the proof here is to use extension of $H^{1/2}$ norm on finite intervals to extend the function to be defined on $\mathbb R$, while the part on the infinite interval of $\mathcal E_2$ is already controlled by $\vert \psi\vert_{\dot{H}^{1/2}(\Gamma^{\rm D})}$. 

\noindent{\bf Step 1.} We show here that there exists a function ${\psi}^{\rm ext}_1\in \dot{H}^{1/2}(\RR)$ such that ${\psi}^{\rm ext}_1=\psi_1$ on ${\mathcal E}_1$ and such that there exists a constant $C>0$ such that
$ \vert \psi_1^{\rm ext}\vert_{\dot{H}^{1/2}(\RR)}\leq C \vert \psi_1\vert_{\dot{H}^{1/2}(\cE_1)}$. 

From Proposition \ref{propPoincare}, we know that the function $\psi_1-\overline{\psi}_1$ belongs to $H^{1/2}({\mathcal E}_1)$; there exists therefore an extension of this function, denoted $(\psi_1-\overline{\psi}_1)^{\rm ext}$, that belongs to $H^{1/2}(\RR)$ and such that $\vert (\psi_1-\overline{\psi}_1)^{\rm ext}\vert_{H^{1/2}(\RR)}\leq C \vert \psi_1-\overline{\psi}_1\vert_{H^{1/2}(\mathcal E_1)}$, for some constant $C>0$. 

The function $\psi_1^{\rm ext}:=(\psi_1-\overline{\psi}_1)^{\rm ext}+\overline{\psi}_1$ is therefore an extension of $\psi_1$ that belongs to $\dot{H}^{1/2}(\RR)$. Moreover, one has
\begin{align*}
\vert \psi_1^{\rm ext}\vert_{\dot{H}^{1/2}(\RR)}=&\vert (\psi_1-\overline{\psi}_1)^{\rm ext}\vert_{\dot{H}^{1/2}(\RR)}\\
\leq & \vert (\psi_1-\overline{\psi}_1)^{\rm ext}\vert_{{H}^{1/2}(\RR)}\\
\leq & C  \vert \psi_1-\overline{\psi}_1\vert_{H^{1/2}(\mathcal E_1)}.
\end{align*}
The result therefore follows again from Propositions \ref{propPoincare} and \ref{propequivter}.

\noindent{\bf Step 2.} We show here that there exists a function ${\psi}^{\rm ext}_2\in \dot{H}^{1/2}(\RR)$ such that  ${\psi}^{\rm ext}_2=\psi_2$ on ${\mathcal E}_2$ and such that there exists a constant $C>0$ such that
$ \vert \psi_2^{\rm ext}\vert_{\dot{H}^{1/2}(\RR)}\leq C \vert \psi_2\vert_{\dot{H}^{1/2}(\cE_2)}$. The proof slightly differs from the previous step because ${\mathcal E}_2=(x_1^{\rm r},\infty)$ is unbounded. 

Let us denote by  $I=(x_1^{\rm r},x_1^{\rm r}+2)$. The restriction $(\psi_2)_{\vert_{I}}$ belongs to $\dot{H}^{1/2}(I)$. As in Step 1, and with the same notations, we can construct  an extension  $((\psi_2)_{\vert_I}-\overline{(\psi_2)_{\vert_I}})^{\rm ext}$ in ${H}^{1/2}(\RR)$ such that 
\begin{align*}
\vert ((\psi_2)_{\vert_I}-\overline{(\psi_2)_{\vert_I}})^{\rm ext}\vert_{{H}^{1/2}(\RR)} 
&\lesssim \vert (\psi_2)_{\vert_I}\vert_{\dot{H}^{1/2}(I)}\\
&\lesssim  \vert \psi_2\vert_{\dot{H}^{1/2}({\mathcal E}_2)}.
\end{align*}

Let us then construct $\psi_2^{\rm ext}$ as follows,
$$
\psi_2^{\rm ext}:= \chi\big[ ((\psi_2)_{\vert_I}-\overline{(\psi_2)_{\vert_I}})^{\rm ext}+\overline{(\psi_2)_{\vert_I}} \big]
+(1-\chi)\psi_2,
$$
for some smooth decreasing positive function $\chi$ such that $\chi(x)=1$ for $x<x_1^{\rm r}$ and $\chi(x)=0$ for $x>x_1^{\rm r}+1$.
In particular, one can check that $\psi_2^{\rm ext}=\psi_2$ on $\cE_2$.

One then has
\begin{align*}
\vert \psi_2^{\rm ext}\vert_{\dot{H}^{1/2}(\mathbb R)}^2=&\int_{-\infty}^{x^{\rm r}_1+1} \int_{\vert y-x\vert \leq 1} \frac{\vert\psi_2^{\rm ext}(y)-\psi_2^{\rm ext}(x)\vert^2}{(y-x)^2}{\rm d}y{\rm d}x\\
&+
\int_{x^{\rm r}_1+1}^\infty \int_{\vert y-x\vert \leq 1} \frac{\vert\psi_2^{\rm ext}(y)-\psi_2^{\rm ext}(x)\vert^2}{(y-x)^2}{\rm d}y {\rm d}x.
\end{align*}

Remarking that in the first integral, one can replace $\psi_2^{\rm ext}$ by $((\psi_2)_{\vert_I}-\overline{(\psi_2)_{\vert_I}})^{\rm ext}$, and by $\psi_2$ in the second one, we deduce that
\begin{align*}
\vert \psi_2^{\rm ext}\vert_{\dot{H}^{1/2}(\mathbb R)} &\lesssim \vert ((\psi_2)_{\vert_I}-\overline{(\psi_2)_{\vert_I}})^{\rm ext}\vert_{{H}^{1/2}(\RR)}  + \vert \psi_2\vert_{\dot{H}^{1/2}(\cE_2)}\\
&\lesssim \vert \psi_2\vert_{\dot{H}^{1/2}(\cE_2)},
\end{align*}
which is the desired result.

\noindent{\bf Step 3.} Conclusion. Let $\theta$ be a smooth decreasing compactly supported function on $\RR$ such that $\theta\equiv 1$ on ${\mathcal E}_1$ and $\theta\equiv 0$ on ${\mathcal E}_2$, and define
\[
\widetilde{\psi}:=\theta \psi_1^{\rm ext}+(1-\theta)\psi_2^{\rm ext};
\]
in particular, one has $\widetilde{\psi}=\psi_j$ on ${\mathcal E}_j$, $j=1,2$.

Moreover, we know directly that 
\[
\widetilde{\psi}=
\psi_2^{\rm ext}+\theta \big(  (\psi_1-\overline{\psi}_1)^{\rm ext}-((\psi_2)_{\vert_I}-\overline{(\psi_2)_{\vert_I}})^{\rm ext}\big) +\theta (\overline{\psi}_1-\overline{(\psi_2)_{\vert_I}}), \qquad\hbox{when}\quad x\le x^{\rm r}_1,  
\]
and 
\[
\widetilde{\psi}=\psi_2,\qquad\hbox{when}\quad x\ge x^{\rm r}_1.
\]
 From the previous steps and the fact that  $\vert \theta \vert_{\dot{H}^{1/2}(\RR)}<\infty$, one obtains that
\begin{align*}
 \vert \widetilde{\psi}\vert_{\dot{H}^{1/2}(\RR)}\lesssim & \Big(\int_{-\infty}^{x^{\rm r}_1+1} \int_{\vert y-x\vert \leq 1} \frac{\vert\widetilde{\psi}(y)-\widetilde{\psi}(x)\vert^2}{(y-x)^2}{\rm d}y{\rm d}x\Big)^{1/2}\\
&+
\Big(\int_{x^{\rm r}_1+1}^\infty \int_{\vert y-x\vert \leq 1} \frac{\vert\widetilde{\psi}(y)-\widetilde{\psi}(x)\vert^2}{(y-x)^2}{\rm d}y{\rm d}x\Big)^{1/2}\\
\lesssim & \vert \psi_2^{\rm ext}\vert_{\dot H^{1/2}(\mathbb R)}+\vert\theta  (\psi_1-\overline{\psi}_1)^{\rm ext}\vert_{\dot H^{1/2}(\mathbb R)}+\vert\theta((\psi_2)_{\vert_I}-\overline{(\psi_2)_{\vert_I}})^{\rm ext}\vert_{\dot H^{1/2}(\mathbb R)}\\
&+\vert\theta (\overline{\psi}_1-\overline{(\psi_2)_{\vert_I}})\vert_{\dot H^{1/2}(\mathbb R)}+\vert \psi_2\vert_{\dot H^{1/2}(\mathcal E_2)}\\
\lesssim & \vert \psi_1\vert_{\dot{H}^{1/2}({\mathcal E}_1)}+\vert \psi_2\vert_{\dot{H}^{1/2}({\mathcal E}_2)} +\vert \overline{\psi}_1-\overline{(\psi_2)_{\vert_I}}\vert  
\end{align*}
by the definition of $\vert \cdot \vert_{\dot{H}^{1/2}(\Gamma^{\rm D})}$, this concludes the proof of the lemma thanks to Proposition \ref{propconstequiv}.
\end{proof}
We have therefore proved the proposition.
\end{proof}

It is now easy to prove the surjectivity of the trace mapping and that it admits a continuous right-inverse.
\begin{theorem}\label{theoremsurj}
Let $\Omega$ and $\Gamma^{\rm D}$ be as in Assumption \ref{assconfig}. 
Then there exists a continuous mapping ${\bf E}: \dot{H}^{1/2}(\Gamma^{\rm D})\to \dot{H}^1(\Omega)$ such that ${\rm Tr}^{\rm D}\circ {\bf E}={\rm Id}_{\dot{H}^{1/2}(\Gamma^{\rm D})\to \dot{H}^{1/2}(\Gamma^{\rm D})}$.
\end{theorem} 
\begin{proof}
By Proposition \ref{propH12char}, functions in  $\dot{H}^{1/2}(\Gamma^{\rm D})$ can be continuously extended to functions in $\dot{H}^{1/2}(\RR)$.
Let $h_0>0$ be such that $\Omega \subset \RR\times (-h_0,0)$. The result now simply follows from the fact that the extension theorem is known in the case of a strip \cite{Lannes_book,Strichartz,LeoniTice}; more precisely,  for all $\psi\in \dot{H}^{1/2}(\RR)$, there exists $\phi\in \dot{H}^1(\RR\times (-h_0,0))$ such that $\Vert \nabla \phi \Vert_{L^2(\RR\times (-h_0,0))}\leq C \vert \psi\vert_{\dot{H}^{1/2}(\RR)}$, with a constant $C>0$ independent of $\psi$. Thanks to the lemma, it is clear that the restriction of $\phi$ to $\Omega$ furnishes the desired extension.
\end{proof}

Since $\dot{H}^1(\Omega)/\RR$ is a Banach space (see \cite{DenyLions}), we directly get the following corollary from Theorems \ref{theortrace} and \ref{theoremsurj}.
\begin{corollary}\label{coroBanach}
Let $\Omega$ and $\Gamma^{\rm D}$ be as in Assumption \ref{assconfig}. The space $\dot{H}^{1/2}(\Gamma^{\rm D})/\RR$ is a Banach space.
\end{corollary}

\subsection{A density result}\label{sectdensity}

We prove here the following density result.
\begin{proposition}\label{propdense2}
Let $\Omega$ and $\Gamma^{\rm D}$ be as in Assumption \ref{assconfig}. Then ${\mathcal D}(\Gamma^{\rm D})$ is dense in $\dot{H}^{1/2}(\Gamma^{\rm D})$ for the norm  $\vert\cdot\vert_{\dot{H}^{1/2}(\Gamma^{\rm D})}$. 
\end{proposition} 
\begin{proof}
Let us denote by $\widetilde{\psi}\in \dot{H}^{1/2}(\RR)$ the extension of $\psi$ furnished by Lemma \ref{lemmaext}. This extension is a tempered distribution in virtue of the following lemma. 
\begin{lemma}\label{lemmatempered}
If $f\in \dot{H}^{1/2}(\RR)$ then $f\in {\mathcal S}'(\RR)$.
\end{lemma}
\begin{proof}[Proof of the lemma]
Let $f\in \dot{H}^{1/2}(\RR)$. If we can prove that $F(x):=\int_0^x f$ is in ${\mathcal S}'(\RR)$, then we can conclude that $f$ is also in ${\mathcal S}'(\RR)$ since one has $F'=f$ almost everywhere. 

In order to prove that $ F\in {\mathcal S}'(\RR)$, we show that it has at most polynomial growth as $|x|\to \infty$. We consider the case $x\to \infty$, the case of negative values of $x$ being treated similarly.

Let $n=[x]$. Then we have $\vert F(x)\vert \leq \sum_{j=0}^n c_j$, with $c_j=\int_j^{j+1} \vert f \vert$. It follows that  
\begin{align*}
\vert F(x)\vert &\leq (n+1) c_0 +\sum_{j=1}^{n} (n-j+1) | c_{j}-c_{j-1} |.
\end{align*}
But we know from Proposition \ref{propconstequiv} that $| c_{j}-c_{j-1} | \leq C  \vert f\vert_{\dot{H}^{1/2}(\RR)} $ with a constant $C$ that does not depend on $j$ (this fact can easily be checked in the proof of Proposition \ref{propconstequiv}). Therefore
$$
\vert F(x)\vert \leq (x+1) c_0 + C (1+x)^2  \vert f\vert_{\dot{H}^{1/2}(\RR)};
$$
the function $F$ has therefore at most a polynomial growth at infinity and therefore belongs to ${\mathcal S}'(\RR)$, which concludes the proof.
\end{proof} 
Using Lemma \ref{lemmatempered}, we can prove the density of ${\mathcal D}(\RR)$ in $\dot{H}^{1/2}(\RR)$. Note that this result can be found in \cite{StevensonTice} where it is proved in a more general framework using involved interpolation arguments; for  the sake of completeness we give here an elementary proof.
\begin{lemma}\label{lemmadense1}
The space  ${\mathcal D}(\RR)$ is dense in $\dot{H}^{1/2}(\RR)$ for the $\dot{H}^{1/2}(\RR)$ semi-norm.
\end{lemma}
\begin{proof}[Proof of the lemma]
We want to prove that for all $f\in \dot{H}^{1/2}(\RR)$ and all $\epsilon>0$, one can find $\varphi\in {\mathcal D}(\RR)$ such that $\vert f-\varphi\vert_{\dot{H}^{1/2}(\RR)}\leq \epsilon$.

From Lemma \ref{lemmatempered}, functions in $\dot{H}^{1/2}(\RR)$ are tempered distributions; we can therefore use the Fourier characterization \eqref{equivtempered} of $\dot{H}^{1/2}(\RR)$, namely that the semi-norms $\vert {\mathfrak P} \cdot  \vert_{L^2}$ (with ${\mathfrak P}=\frac{\vert{D}\vert}{(1+\vert D\vert)^{1/2}}$)and $\vert \cdot \vert_{\dot{H}^{1/2}(\RR)}$ are equivalent. 

Let us decompose $f=f_{\rm low}+f_{\rm high}$ with
$$
f_{\rm low}={\bf 1}_{\vert D\vert\leq 1} f \quad\mbox{ and }\quad f_{\rm high}={\bf 1}_{\vert D\vert\geq 1} f,
$$
where ${\bf 1}_{\vert D\vert\leq 1}$ and ${\bf 1}_{\vert D\vert\geq 1}$ are Fourier multipliers associated to the projection on frequencies smaller and greater than $1$ respectively. One has $f_{\rm low} \in \dot{H}^1(\RR)$ and $f_{\rm high}\in H^{1/2}(\RR)$. Indeed, 
$$
\vert \vert D\vert f_{\rm low} \vert_{L^2(\RR)}
+ \vert (1+\vert D\vert^2)^{1/4}f_{\rm high}\vert_{L^2(\RR)}
\lesssim \vert \frac{\vert{D}\vert}{(1+\vert D\vert)^{1/2}} f\vert_{L^2(\RR)}.
$$

 It is also known that ${\mathcal D}(\RR)$ is dense in $\dot{H}^1(\RR)$ (see \cite{Sobolev}) and in $H^{1/2}(\RR)$, so that we can find $\varphi_1$ and $\varphi_2$ in ${\mathcal D}(\RR)$ such that $\vert f_{\rm low}-\varphi_1\vert_{\dot{H}^1(R)}<\epsilon/2$ and $\vert f_{\rm high}-\varphi_2\vert_{{H}^{1/2}(R)}<\epsilon/2$. Setting $\varphi=\varphi_1+\varphi_2$, we have
 \begin{align*}
 \vert f-\varphi\vert_{\dot{H}^{1/2}(\RR)}&\leq \vert f_{\rm low}-\varphi_1\vert_{\dot{H}^{1/2}(\RR)} +\vert f_{\rm high}-\varphi_2\vert_{\dot{H}^{1/2}(\RR)} \\
 &\leq  \vert f_{\rm low}-\varphi_1\vert_{\dot{H}^{1}(\RR)} +\vert f_{\rm high}-\varphi_2\vert_{{H}^{1/2}(\RR)}.
 \end{align*}
 Therefore, $ \vert f-\varphi\vert_{\dot{H}^{1/2}(\RR)}\leq \epsilon$, and since $\varphi\in {\mathcal D}(\RR)$, this proves the result.
\end{proof}
The lemma below shows the density of ${\mathcal D}(I)$ in  $\dot{H}^{1/2}(I)$ when $I$ is finite.
\begin{lemma}\label{lemmadense2}
Let $I$ be a finite non empty interval. Then ${\mathcal D}(I)$ is dense in $\dot{H}^{1/2}(I)$ for the $\dot{H}^{1/2}(I)$ semi-norm.
 \end{lemma}
 \begin{proof}[Proof of the lemma]
 We want to prove that for all $f\in \dot{H}^{1/2}(I)$, there is a function $\varphi\in {\mathcal D}(I)$ such that $\vert f-\varphi\vert_{\dot{H}^{1/2}(I)}\leq \epsilon$.
 Functions that are in $\dot{H}^{1/2}(I)$ also belong to ${H}^{1/2}(I)$ (see Remark \ref{remHdot12H12}). Since smooth functions compactly supported in $I$ are dense in $H^{1/2}(I)$, one can find  $\varphi\in {\mathcal D}(I)$ satisfying $\vert f-\varphi\vert_{{H}^{1/2}(I)}<\epsilon$, and therefore $\vert f-\varphi \vert_{\dot{H}^{1/2}(I)}<\epsilon$, which proves the result.
 \end{proof}
 \begin{remark}\label{remconav}
 With the notations used in the proof, it also follows that
 $$
 \frac{1}{\vert I\vert}\int_I \vert f-\varphi \vert \leq \frac{1}{\vert I\vert^{1/2}}\vert f-\varphi\vert_{L^2(I)} < \frac{1}{\vert I\vert^{1/2}}\epsilon,
 $$
 which shows that one can approximate functions of  $\dot{H}^{1/2}(I)$ at arbitrary precision   by functions in ${\mathcal D}(I)$ whose average over $I$ is arbitrarily close to the average of $f$.
 \end{remark} 
With this lemma we can conclude to the density of ${\mathcal D}(\Gamma^{\rm D})$ in $\dot{H}^{1/2}(\Gamma^{\rm D})$. 
Let $f\in \dot{H}^{1/2}(\Gamma^{\rm D})$. By Proposition \ref{propH12char} and Lemma \ref{lemmadense1}, for all $\epsilon>0$, one can find $\varphi\in {\mathcal D}(\RR)$ such that 
\begin{align*}
    \vert \varphi_{\vert_{\Gamma^{\rm D}}}-f\vert_{\dot{H}^{1/2}(\Gamma^{\rm D})} <\epsilon. 
\end{align*}
Now, by Lemma \ref{lemmadense2} and Remark \ref{remconav}, one can find $\widetilde{\varphi} \in {\mathcal D}(\Gamma^{\rm D})$ such that $\vert \varphi_{\vert_{\Gamma^{\rm D}}}-\widetilde{\varphi}\vert_{\dot{H}^{1/2}(\Gamma^{\rm D})}<\epsilon$. This shows that $f$ can be approximated at any precision in $\dot{H}^{1/2}(\Gamma^{\rm D})$ by some function in ${\mathcal D}(\Gamma^{\rm D})$, which is the desired result.
\end{proof}

\subsection{The spaces $\dot{\mathcal H}^{1/2}(\Gamma^{\rm D})$ and $\dot{\mathcal H}^{1}(\Gamma^{\rm D})$} \label{sectdotH012}

Defining $\overline{\psi}_1, \dots, \overline{\psi}_{N+1}$ as in Definition \ref{propaverage}, we can therefore define the subspace $\dot{\mathcal H}^{s}(\Gamma^{\rm D})$ of $\dot{H}^{s}(\Gamma^{\rm D})$, with $s=1/2$ or $1$, as follows
\begin{equation}\label{defHI120}
\dot{\mathcal H}^{s}(\Gamma^{\rm D})=\{ \psi \in \dot{H}^{s}(\Gamma^{\rm D}), \sum_{j=1}^{N+1} \vert I_j\vert \overline{\psi}_{j}=0\} \qquad (s=1/2,1),
\end{equation} 
where $I_j=\cE_j$ if $\cE_{j}$ is bounded and, as in Definition \ref{propaverage}, $I_j$ a non-empty subinterval of $\cE_{j}$ if it is unbounded.
\begin{remark}
If $\Gamma^{\rm D}$ is bounded, then $\psi\in \dot{H}^{s}(\Gamma^{\rm D})$ is in $\dot{\mathcal H}^{s}(\Gamma^{\rm D})$ if and only if $\int_{\Gamma^{\rm D}}\psi =0$.
\end{remark} 
The following proposition shows that $\dot{\mathcal H}^{1/2}(\Gamma^{\rm D})$ is a realization of $\dot{H}^{1/2}(\Gamma^{\rm D})$ in the sense of \cite{Bourdaud} (that is, a convenient choice of a particular representative in each class of $\dot{H}^{1/2}(\Gamma^{\rm D})$).
\begin{proposition}\label{propBanach}
Let $\Omega$ and $\Gamma^{\rm D}$ be as in Assumption \ref{assconfig}. Then the space $\dot{\mathcal H}^{1/2}(\Gamma^{\rm D})$ endowed with $\vert\cdot\vert_{\dot{H}^{1/2}(\Gamma^{\rm D})}$ is a Banach space of distributions. 
\end{proposition}
\begin{proof}
The fact that $(\dot{\mathcal H}^{1/2}(\Gamma^{\rm D}), \vert\cdot\vert_{\dot{H}^{1/2}(\Gamma^{\rm D})})$ is a distribution space is a consequence of the following lemma.
\begin{lemma}
If a sequence $(\psi^{n})_{n\in {\mathbb N}}$ is such that for all $n\in {\mathbb N}$, $\psi^n\in \dot{\mathcal H}^{1/2}(\Gamma^{\rm D})$ and $\vert \psi^n\vert_{\dot{H}^{1/2}(\Gamma^{\rm D})}\to 0$ as $n\to \infty$ then, for all $\varphi \in {\mathcal D}(\Gamma^{\rm D})$, $\int_{\Gamma^{\rm D}} \psi^n \varphi \to_{n\to \infty} 0$.
\end{lemma}
\begin{proof}[Proof of the lemma]
We can decompose  
$$
\int_{\Gamma^{\rm D}} \psi^n \varphi =\sum_{j=1}^{N+1} \big(\int_{\cE_j}  \big(\psi^n_j-\overline{\psi^n_j})\varphi +\overline{\psi^n_j} \int_{\cE_j} \varphi \big).
$$

 If $\cE_j$ is bounded, then we can write
\begin{align*}
\big\vert  \int_{\cE_j}  \big(\psi^n_j-\overline{\psi^n_j})\varphi \big\vert & \leq  \vert \psi^n_j-\overline{\psi^n_j}\vert_{L^2(\cE_j)} \vert \varphi\vert_{L^2(\cE_j)} \\
&\lesssim  \vert \psi_j^n\vert_{\dot{H}^{1/2}(\cE_j)} \vert \varphi\vert_{L^2(\cE_j)} ,
\end{align*}
the last line being a consequence of the Poincar\'e inequality of Proposition \ref{propPoincare}.

If $\cE_j$ is unbounded, then let $K\subset \cE_j$ be a compact interval that contains the intersection of the support of $\varphi$ with $\cE_j$. Denoting by $\overline{\overline{\psi^n}}_j=\frac{1}{\vert K\vert}\int_K \psi^n_j$, we have
\begin{align*}
\int_{\cE_j}  \big(\psi^n_j-\overline{\psi^n_j})\varphi &=\int_{K}  \big(\psi^n_j-\overline{\psi^n_j})\varphi \\
&= \int_{K}  \big(\psi^n_j-\overline{\overline{\psi^n_j}})\varphi + (  \overline{\overline{\psi^n_j}}-\overline{\psi^n_j} ) \int_K \varphi;
\end{align*}
the first term in the right-hand side can be handled as for the case when $\cE_j$ is bounded, while we can use Proposition \ref{propconstequiv} to control the second one. 

In the end, we get
$$
\big\vert \int_{\Gamma^{\rm D}} \psi^n \varphi \big\vert  \leq C(\varphi)  \sum_{j=1}^{N+1} (\vert \psi^n_j\vert_{\dot{H}^{1/2}(\cE_j)} +\vert \overline{\psi^n_j} \vert ),
$$
where $C(\varphi)$ is some finite constant depending on $\varphi$.

We can moreover write
$$
\overline{\psi^n_j}=\big( \sum_{k=1}^{N+1} \vert I_k \vert \big)^{-1} \big( \sum_{k=1}^{N+1} \vert I_k \vert \overline{\psi^n_k}+ \sum_{k=1}^{N+1} \vert I_k \vert (\overline{\psi^n_j} - \overline{\psi^n_k}) \big),
$$
and use the fact that $ \sum_{k=1}^{N+1} \vert I_k\vert \overline{\psi^n_k}=0$ by assumption to deduce the upper bound $\vert \overline{\psi^n_j}\vert \lesssim  \sum_{1}^{N}  \vert \overline{\psi_j^n}- \overline{\psi^n_{j+1}}\vert $. These observations, together with the definition of the semi-norm $\vert \cdot \vert_{\dot{H}^{1/2}(\Gamma^{\rm D})}$ imply that
$$
\big\vert \int_{\Gamma^{\rm D}} \psi^n \varphi \big\vert  \leq C(\varphi) \vert \psi^n\vert_{\dot{H}^{1/2}(\Gamma^{\rm D})},
$$
where we still denote by $C(\varphi)$ a finite constant depending on $\varphi$. The result follows immediately.
\end{proof}

We have seen that $\dot{H}^{1/2}(\Gamma^{\rm D})/\RR$ is a Banach space. 
For all $\psi^\bullet\in \dot{H}^{1/2}(\Gamma^{\rm D})/\RR$, it follows easily from Proposition \ref{propaverage} that there exists a unique representative $\psi^{(0)} \in \psi^\bullet$ such that 
$\sum_{j=1}^{N+1} \vert I_j\vert \overline{\psi^{(0)}}_{j}=0$, that is, such that ${\psi}^{(0)}\in \dot{\mathcal H}^{1/2}(\Gamma^{\rm D})$. Denoting by $\sigma(\psi^\bullet)=\psi^{(0)}$, the mapping $\sigma: \dot{H}^{1/2}(\Gamma^{\rm D})/\RR \to \dot{\mathcal H}^{1/2}(\Gamma^{\rm D})$ is bijective so that $ \dot{\mathcal H}^{1/2}(\Gamma^{\rm D})$ inherits the structure of Banach space of  $ \dot{H}^{1/2}(\Gamma^{\rm D})/\RR$.
\end{proof}

%
%

\section{The Laplace equation and the Dirichlet-Neumann operator in corner domains}\label{sectLaplace}

It is well known that for all Dirichlet data $\psi$ in the inhomogeneous Sobolev space $H^{1/2}(\Gamma^{\rm D})$ one can construct a unique velocity potential which is harmonic in $\Omega$, whose trace on $\Gamma^{\rm D}$ coincides with $\psi$ and with homogeneous Neumann boundary condition on $\Gamma^{\rm N}$. We shall denote by $\psi^\mfh$ this velocity potential, and will refer to it as the {\it harmonic extension} of $\psi$. Moreover, the mapping $\psi \in H^{1/2}(\Gamma^{\rm D}) \mapsto \psi^\mfh \in H^1(\Omega)$ is continuous.

We prove in \S \ref{sectharmext} that this classical result can be extended to the homogeneous case, that is, that the mapping $\psi \in \dot{H}^{1/2}(\Gamma^{\rm D}) \mapsto \psi^\mfh \in \dot{H}^1(\Omega)$ is well defined and continuous. This result, of independent interest, is optimal inasmuch as the trace on $\Gamma^{\rm D}$ of any velocity potential of finite kinetic energy belongs to $\dot{H}^{1/2}(\Gamma^{\rm D})$.\\
Once this is done, we rigorously construct in \S \ref{sectDN} the Dirichlet-Neumann operator $G_0: \psi \mapsto (\partial_z\psi^\mfh)_{\vert_{\Gamma^{\rm D}}}$ as a mapping from $\dot{H}^{1/2}(\Gamma^{\rm D})$ with values in its dual. Using Rellich type identities we also prove that it is a continuous elliptic operator from $\dot{H}^1(\Gamma^{\rm D})$ to $L^2(\Gamma^{\rm D})$. Other properties (self-adjointness and elliptic regularity) are also investigated.

\subsection{The harmonic extension on $\dot{H}^{1/2}(\Gamma^{\rm D})$}\label{sectharmext}

It is well known that if $\psi \in H^{1/2}(\Gamma^{\rm D})$, there is a unique harmonic function in  ${H}^1(\Omega)$, denoted by $\psi^\mfh$,  such that $\psi^\mfh_{\vert_{\Gamma^{\rm D}}}=\psi$  and with  homogeneous Neumman boundary conditions on $\Gamma^{\rm N}$. We show here that this result can be extended to functions in $\dot{H}^{1/2}(\Gamma^{\rm D})$.

Let us first note that if $u\in \dot{H}^1(\Omega)$ is a harmonic function, its outward normal derivative $(\partial_{\rm n} u)_{\vert_{\Gamma^{\rm N}}}$ on $\Gamma^{\rm N}$ can be defined in weak sense as 
\begin{equation}\label{defweakdn}
\forall v\in {H}^1_{\rm D}(\Omega), \qquad \langle (\partial_{\rm n} u)_{\vert_{{\Gamma^{\rm N}}}}, v_{\vert_{{\Gamma}^{\rm N}}} \rangle=\int_\Omega \nabla u\cdot \nabla v,
\end{equation}
where we recall that ${H}^1_{\rm D}(\Omega)$ denotes the subspace of $H^1(\Omega)$ of functions that vanish on $\Gamma^{\rm D}$, and where the brackets in \eqref{defweakdn} therefore stand for the $H_{00}^{1/2}(\Gamma^{\rm N})'-H_{00}^{1/2}(\Gamma^{\rm N})$ duality bracket, with  $H_{00}^{1/2}(\Gamma^{\rm N})=H^{1/2}(\Gamma^{\rm N})\cap {\rho^{1/2}}L^2(\Gamma^{\rm N})$  see the remark below.
\begin{remark}
Since the distance of every point of $\Omega$ to $\Gamma^{\rm D}$ is uniformly bounded in $\Omega$ under Assumption \ref{assconfig}, the Poincar\'e inequality implies that $H_{\rm D}^1(\Omega)$ and $\ker {\rm Tr}^{\rm D}$ are algebraically and topologically the same. Therefore, one could replace the condition $v\in H^1_{\rm D}(\Omega)$ in \eqref{defweakdn} by $v\in \dot{H}^1(\Omega)$ such that ${\rm Tr}^{\rm D}v=0$. From the standard trace theorem on corner domains \cite{Grisvard}, we know that the trace mapping ${\rm Tr}^{\rm N}: u \in H^1_{\rm D}(\Omega)\mapsto u_{\vert_{\Gamma^{\rm N}}}$ takes its values in $H^{1/2}_{00}(\Gamma^{\rm N})$ and is onto, this is why \eqref{defweakdn} defines $ (\partial_{\rm n} u)_{\vert_{\Gamma^{\rm N}}}$ in $H_{00}^{1/2}(\Gamma^{\rm N})'$.
\end{remark}
\begin{proposition}\label{propharmonic}
Let $\Omega$, $\Gamma^{\rm D}$ and $\Gamma^{\rm N}$ be as in Assumption \ref{assconfig}. Then for all $\psi\in \dot{H}^{1/2}(\Gamma^{\rm D})$ there exists a unique function in $\dot{H}^1(\Omega)$, denoted by $\psi^\mfh$, that solves the boundary value problem
$$
\begin{cases}
\Delta \psi^\mfh=0 &\quad\mbox{ in }\Omega,\\
\psi^\mfh=\psi& \quad \mbox{ on }\Gamma^{\rm D},\\
\partial_{\rm n}\psi^\mfh=0 & \quad \mbox{ on }\Gamma^{\rm N};
\end{cases}
$$
moreover, $\psi^\mfh$ minimizes the Dirichlet energy,
$$
\Vert \nabla\psi^\mfh \Vert_{L^2(\Omega)}^2= \min_{\phi \in \dot{H}^1(\Omega), {\rm Tr}^{\rm D}\phi=\psi}\int_\Omega \vert \nabla \phi\vert^2.
$$
\end{proposition}

\begin{remark}
As shown in Proposition \ref{propequivnorms2} below, one has $\Vert \nabla\psi^{\mfh}\Vert_{L^2(\Omega)}\sim \vert \psi \vert_{\dot{H}^{1/2}(\Gamma^{\rm D})}$, which shows the optimality of Proposition \ref{propharmonic} (this equivalence is of course wrong if we work with $\psi\in H^{1/2}(\Gamma^{\rm D})$ rather than $ \dot{H}^{1/2}(\Gamma^{\rm D})$).
\end{remark}
\begin{proof}
We have now the tools to adapt the standard proof \cite{DenyLions}. Let us denote by $H^1_{\rm D}(\Omega)^{\bullet}$ the image of $H^{\rm D}(\Omega)$ in $\dot{H}^1(\Omega)/\RR$ by the canonical mapping $\phi \in \dot{H}^1(\Omega)\mapsto \phi^\bullet \in \dot{H}^1(\Omega)/\RR$, where $u^\bullet$ denotes the class of $u$ in $ \dot{H}^1(\Omega)/\RR$. 

Then $\dot{H}^1(\Omega)/\RR$ can be written as the direct sum of two orthogonal (for the $\dot{H}^1(\Omega)/\RR$-scalar product) spaces,
$$
\dot{H}^1(\Omega)/\RR=H^1_{\rm D}(\Omega)^{\bullet}\oplus {\mathcal H}^\bullet,
$$
where ${\mathcal H}^\bullet$ is the class of all distributions $\phi\in \dot{H}^1(\Omega)$ such that $\Delta \phi=0$ in $\Omega$ and $(\partial_{\rm n}\phi)_{\vert_{\Gamma^{\rm N}}}=0$. 

Indeed, if $\phi^\bullet\in \dot{H}^1(\Omega)/\RR$ is orthogonal to $H^1_{\rm D}(\Omega)^\bullet$ then for all $\phi \in \phi^\bullet$ and $v\in H^1_{\rm D}(\Omega)$, one has
$$
\int_\Omega \nabla \phi\cdot \nabla v=0.
$$

By considering $v$ compactly supported in $\Omega$, this implies that $\Delta \phi=0$, and we can then use \eqref{defweakdn} to deduce that $(\partial_{\rm n}\phi)_{\vert_{\Gamma^{\rm N}}}=0$.

It follows that any $\phi \in \dot{H}^1(\Omega)$ can be decomposed in a unique way under the form
$$
\phi=\phi_1+\phi_2,
$$
where $\phi_1\in H^1_{\rm D}(\Omega)$ and $\phi_2\in \dot{H}^1(\Omega)$ is such that $\Delta \phi_2=0$ with $(\partial_{\rm n}\phi_2)_{\vert_{\Gamma^{\rm N}}}=0$, and moreover,
$$
\Vert \phi\Vert_{\dot{H}^1(\Omega)}=\Vert \phi_1\Vert_{\dot{H}^1(\Omega)}+\Vert \phi_2\Vert_{\dot{H}^1(\Omega)}.
$$

For any $\psi\in \dot{H}^{1/2}(\Gamma^{\rm D})$, we can take $\phi$ to be the extension constructed in Theorem \ref{theoremsurj} and apply the above decomposition to obtain the result, with $\psi^\mfh=\phi_2$.
\end{proof}

\subsection{The Dirichlet-Neumann operator}\label{sectDN}

For all $\psi\in \dot{H}^{1/2}(\Gamma^{\rm D})$, the previous section provides an harmonic extension $\psi^\mfh\in \dot{H}^1(\Omega)$; its outwards normal derivative on $\Gamma^{\rm D}$ is then given in the weak sense by
\begin{equation}\label{defdnpsih}
\forall \psi'\in \dot{H}^{1/2}(\Gamma^{\rm D}),\qquad \langle (\partial_{\rm n}\psi^\mfh)_{\Gamma^{\rm D}}, \psi'\rangle=\int_\Omega \nabla \psi^\mfh\cdot \nabla (\psi')^\mfh,
\end{equation}
where the brackets here stand for the $\dot{H}^{1/2}(\Gamma^{\rm D})'-\dot{H}^{1/2}(\Gamma^{\rm D})$ duality brackets. 

We can now provide the following definition of the Dirichlet-Neumann operator on corner domains.  
\begin{definition}\label{propDN}
Let $\Omega$, $\Gamma^{\rm D}$ and $\Gamma^{\rm N}$ be as in Assumption \ref{assconfig}. The Dirichlet-Neumann operator $G_0$ is defined as
$$
G_0: \begin{array}{lcl}
\dot{H}^{1/2}(\Gamma^{\rm D})& \to & \dot{H}^{1/2}(\Gamma^{\rm D})' \\
\psi &\mapsto & (\partial_{\rm n}\psi^\mfh)_{\vert_{\Gamma^{\rm D}}}
\end{array},
$$
where $\psi^\mfh$ is given by Proposition \ref{propharmonic} and the normal derivative $(\partial_{\rm n}\psi^\mfh)_{\vert_{\Gamma^{\rm D}}}$ is defined in the sense of \eqref{defdnpsih}. In particular,
$$
 \forall \psi,\psi'\in \dot{H}^{1/2}(\Gamma^{\rm D}), \qquad \langle G_0\psi,\psi'\rangle= \int_\Omega \nabla \psi^\mfh\cdot \nabla (\psi')^\mfh.
$$
\end{definition}
\begin{remark}\label{remDN}
According to the proposition, the Dirichlet energy $\Vert \nabla \psi^\mfh\Vert_{L^2(\Omega)}$ that often appears in the computations can be written in terms of $G_0$ and $\psi$ only, namely,
\begin{equation}\label{G0kin}
\Vert \nabla\psi^\mfh\Vert_{L^2(\Omega)}^2=\langle G_0\psi,\psi\rangle;
\end{equation}
if $\psi$ and $G_0\psi$ are in $ L^2(\Gamma^{\rm D})$, then $\langle G_0\psi,\psi\rangle=(G_0\psi,\psi)$, where $(\cdot,\cdot)$ stands for the standard $L^2(\Gamma^{\rm D})$-scalar product. By abuse of notations, we often write $(G_0\psi,\psi)$ instead of $\langle G_0\psi,\psi\rangle$.
\end{remark}
We can now prove that $\langle G_0\psi,\psi \rangle^{1/2}$ defines on $\dot{H}^{1/2}(\Gamma^{\rm D})$ a semi-norm that is equivalent to $\vert \cdot \vert_{\dot{H}^{1/2}(\Gamma^{\rm d})}$. 
\begin{proposition}\label{propequivnorms2}
Let $\Omega$, $\Gamma^{\rm D}$ and $\Gamma^{\rm N}$ be as in Assumption \ref{assconfig}. There exists a constant $C>0$ such that for all  $\psi\in \dot{H}^{1/2}(\Gamma^{\rm D})$, one has
$$
\frac{1}{C}\vert \psi\vert^2_{\dot{H}^{1/2}(\Gamma^{\rm D})} \leq \langle G_0\psi,\psi\rangle \leq C\vert \psi\vert^2_{\dot{H}^{1/2}(\Gamma^{\rm D})}.
$$ 
\end{proposition}
\begin{proof}
The function $\psi^\mfh$ is in $\dot{H}^1(\Omega)$ and has trace $\psi$ on $\Gamma^{\rm D}$, so that we can use the continuity of the trace mapping proved in Theorem \ref{theortrace} to get the first inequality. Remark now that there  exists an extension of $\psi$ with finite Dirichlet energy bounded from above by  $C \vert \psi\vert_{\dot{H}^{1/2}(\Gamma^{\rm D})}$ (Theorem \ref{theoremsurj}). The second inequality follows from the minimization property stated in Proposition \ref{propharmonic}. 
\end{proof}

Since the angles between $\Gamma^{\rm N}$ and $\Gamma^{\rm D}$ at each corner are strictly less than $\pi$, it is possible to show that for $\epsilon>0$ small enough, one has $\psi^\mfh\in H^{3/2+\epsilon}(\Omega)$ if $\psi\in H^{1+\epsilon}(\Gamma^{\rm D})$ \cite{Dauge,Poyferre}. By the trace theorem, one deduces that $G_0\psi=(\dz \psi^\mfh)_{\vert_{\Gamma^{\rm D}}}$ belongs to $H^\epsilon(\Gamma^{\rm D})$. The Dirichlet-Neumann operator is therefore well defined as a continuous operator $G_0: H^{1+\epsilon}(\Gamma^{\rm D}) \to H^\epsilon(\Gamma^{\rm D})$. 

The following proposition shows that this remains true in the endpoint case $\epsilon=0$ (and with a homogeneous Sobolev space), which is the one we shall need in this paper. It is based on a Rellich identity, in the spirit of \cite{AgrawalAlazard} where the same identity is used to obtain refined estimates on the Dirichlet-Neumann operator for the Dirichlet problem in a half-domain delimited by the graph of a $C^1$-function. 

In the statement below, $\partial_{\rm tan} \psi^\mfh$ stands for the tangential derivative of $\psi^\mfh$ on the boundary, $\partial_{\rm tan} \psi^\mfh=\widetilde{\bf n}^\perp \cdot \nabla \psi^\mfh$, where $\widetilde{\bf n}$ denotes the outward unit normal vector on each point of $\Gamma^*$ (note that as the normal derivative, the tangential derivative can be defined in a weak sense).  We also recall that the semi-norm $\vert \psi \vert_{\dot{H}^1(\Gamma^{\rm D})}$, defined in Definition \ref{defdotH1Gamma} is larger and not equivalent to the semi-norm $\vert \dx \psi\vert_{L^2(\Gamma^{\rm D})}$. 
\begin{proposition}\label{propDNcont}
Let $\Omega$, $\Gamma^{\rm D}$ and $\Gamma^{\rm N}$ be as in Assumption \ref{assconfig}. 
Then for all $\psi\in \dot{H}^1(\Gamma^{\rm D})$, one has $G_0\psi\in L^2(\Gamma^{\rm D})$ and $(\partial_{\rm tan} \psi^\mfh)_{\vert_{\Gamma^{\rm N}}}\in L^2(\Gamma^{\rm N})$, and moreover
$$
\vert G_0\psi \vert_{L^2(\Gamma^{\rm D})}+\vert (\partial_{\rm tan}\psi^\mfh)_{\vert_{\Gamma^{\rm N}}} \vert_{L^2(\Gamma^{\rm N})}\leq C \vert \psi\vert_{\dot{H}^1(\Gamma^{\rm D})} , 
$$
for some constant $C>0$ that depends only on the geometry of $\Omega$.
\end{proposition}
\begin{proof}
From the assumption on the angles at the boundaries, it is possible to find a smooth vector field ${\boldsymbol{\alpha}}=(\alpha_1,\alpha_2)^{\rm T}$ and a constant $\delta>0$ such that
$$
\begin{cases}
{\boldsymbol{\alpha}}\cdot \widetilde{\bf n} \leq -\delta & \mbox{on }\Gamma^{\rm D}\\
{\boldsymbol{\alpha}} \cdot \widetilde{\bf n} \geq \delta & \mbox{on }  \Gamma^{\rm N},
\end{cases}
$$
where $\widetilde{\bf n}$  is the outwards unit normal vector on  $\Gamma^{\rm D}\cup \Gamma^{\rm N}$. 

We now use the fact that it is possible to define the non-tangential trace at the boundary of solutions of mixed problems of the type 
\begin{equation}\label{systrellich}
\begin{cases}
\Delta u= f \mbox{ in }\Omega,\\
u_{\vert_{\Gamma^{\rm D}}}=g, \qquad (\partial_{\rm n} u)_{\vert_{\Gamma^{\rm N}}}=h,
\end{cases}
\end{equation}
provided that $f\in L^2(\Omega)$, $g\in \dot{H}^1(\Gamma^{\rm D})$ and $h\in L^2(\Gamma^{\rm N})$. 

We refer to \cite{Brown} where this is proved for a class of bounded domains that includes curvilinear polygons with angles smaller than $\pi$, and to \cite{LCB} where it is proved for unbounded domains that are the epigraph of a class of Lipschitz functions that includes curvilinear broken lines with angles smaller than $\pi$; these two references (which are generalizations to mixed boundary conditions of classical results for the purely Dirichlet or purely Neumann cases \cite{Dahlberg,JK,Kenig}) allow us to cover the configurations of Assumption \ref{assconfig}. 

It is also proved in \cite{Brown,LCB} that under this regularity, it is possible to derive a Rellich identity which reads
\begin{equation}\label{rellich1}
\int_{\Gamma} \vert \nabla u \vert^2 {\boldsymbol{\alpha}}\cdot \widetilde{\bf n}-2\int_{\Gamma } (\widetilde{\bf n}\cdot \nabla u)({\boldsymbol{\alpha}}\cdot \nabla u)=F,
\end{equation}
with
$$
F=\int_\Omega (\nabla\cdot \alpha)\vert{\nabla u}\vert^2-2\int_\Omega\big( (\nabla\alpha_1)\dx u +(\nabla\alpha_2)\dz u\big)\cdot \nabla u
- 2\int_\Omega ( {\boldsymbol{\alpha}}\cdot\nabla u) f.
$$
(this stems easily from Gauss' divergence theorem). 

Applying this identity to $\psi^\mfh$ (that is, taking $f=0$, $g=\psi$ and $h=0$, we obtain 
$$
\int_{\Gamma} \vert \nabla  \psi^\mfh \vert^2 {\boldsymbol{\alpha}}\cdot \widetilde{\bf n}-2\int_{\Gamma^{\rm D}}\partial_{z} \psi^\mfh({\boldsymbol{\alpha}}\cdot \nabla\psi^\mfh)=F,
$$
while $F$ takes the form
$$
F=\int_\Omega (\nabla\cdot \alpha)\vert{\nabla \psi^\mfh}\vert^2-2\int_\Omega\big( (\nabla\alpha_1)\dx\psi^\mfh+(\nabla\alpha_2)\dz\psi^\mfh\big)\cdot \nabla\psi^\mfh.
$$

Decomposing  $\vert{\nabla\psi^\mfh}\vert^2=(\partial_{\rm tan} \psi^\mfh)^2+(\partial_{\rm n}\psi^\mfh)^2$ on $\Gamma^{\rm D}\cup \Gamma^{\rm N}$, this can be rewritten as
$$
\int_{\Gamma^{\rm D}} (\partial_z \psi^\mfh)^2 \alpha_2
+\int_{\Gamma^{\rm N}} (\partial_{\rm tan} \psi^\mfh)^2 {\boldsymbol{\alpha}}\cdot \widetilde{\bf n}
-2\int_{\Gamma^{\rm D}}\partial_z \psi^\mfh({\boldsymbol{\alpha}}\cdot \nabla\psi^\mfh)=F
-\int_{\Gamma^{\rm D}} ( \dx \psi)^2 \alpha_2.
$$

Decomposing ${\boldsymbol{\alpha}}\cdot \nabla\psi^\mfh=\alpha_1\dx\psi+\alpha_2\partial_z\psi^\mfh$ in the third boundary integral, we get
\begin{equation}\label{Rellichbis}
-\int_{\Gamma^{\rm D}} (\partial_z \psi^\mfh)^2 \alpha_2
+\int_{\Gamma^{\rm N}} (\partial_{\rm tan} \psi^\mfh)^2 {\boldsymbol{\alpha}}\cdot \widetilde{\bf n}
=F
-\int_{\Gamma^{\rm D}} ( \partial_x \psi)^2 \alpha_2+2\int_{\Gamma^{\rm D}} \alpha_1\partial_z \psi^\mfh\partial_x\psi ,
\end{equation}
from which, using the properties of $\alpha$ and Young's inequality to control the last term, we obtain that there is a constant $C>0$ such that
$$
\int_{\Gamma^{\rm D}} (\partial_z \psi^\mfh)^2 +\int_{\Gamma^{\rm N}} (\partial_{\rm tan} \psi^\mfh)^2
\leq C\big( \int_{\Gamma^{\rm D}} ( \partial_x\psi)^2 +\int_\Omega \vert \nabla \psi^\mfh \vert^2 \big).
$$
The result then follows from the estimate of Propositions \ref{propequivnorms2} and \ref{characdotH1}.
\end{proof}

Using similar tools, it is also possible to prove an ellipticity result on $G_0$. More precisely, the proposition below shows that if $\psi\in \dot{H}^{1/2}(\Gamma^{\rm D})$ and $G_0\psi\in L^2(\Gamma^{\rm D})$ (and not only in $\dot{H}^{1/2}(\Gamma^{\rm D})'$), then $\psi\in \dot{H}^1(\Gamma^{\rm D})$.
\begin{proposition}\label{propDNell}
Let $\Omega$, $\Gamma^{\rm D}$ and $\Gamma^{\rm N}$ be as in Assumption \ref{assconfig}. 
Then if $\psi\in \dot{H}^{1/2}(\Gamma^{\rm D})$ and  $G_0\psi \in L^2(\Gamma^{\rm D})$, one has $\psi\in \dot{H}^1(\Gamma^{\rm D})$, and moreover
$$
\vert \psi \vert_{\dot{H}^1(\Gamma^{\rm D})} \leq C \big( \vert \psi\vert_{\dot{H}^{1/2}(\Gamma^{\rm D})}+ \vert G_0\psi\vert_{L^2(\Gamma^{\rm D})}\big),
$$
for some constant $C>0$ that depends only on the geometry of $\Omega$.
\end{proposition}
\begin{proof}
We proceed as in the proof of the previous proposition but construct the vector field ${\boldsymbol{\alpha}}$ in such a way that ${\boldsymbol{\alpha}}\cdot \widetilde{\bf n}=0$ on $\Gamma^{\rm N}$. The identity \eqref{Rellichbis} then becomes
$$
-\int_{\Gamma^{\rm D}} ( \partial_x \psi)^2 \alpha_2= F-\int_{\Gamma^{\rm D}} (\partial_z \psi^\mfh)^2 \alpha_2+2\int_{\Gamma^{\rm D}} \alpha_1\partial_z \psi^\mfh\partial_x\psi.
$$

As  in the proof of the previous proposition, we can use Young's inequality to obtain that there is a constant $C>0$ such that
\begin{align*}
\vert \dx \psi \vert_{L^2(\Gamma^{\rm D})}^2 &\leq C \big( \Vert \nabla \psi^\mfh\Vert_{L^2(\Omega)}^2 + \vert G_0\psi\vert_{L^2(\Gamma^{\rm D})}^2 \big)\\
&\leq C \big( \vert \psi\vert_{\dot{H}^{1/2}(\Gamma^{\rm D})} + \vert G_0\psi\vert_{L^2(\Gamma^{\rm D})}^2 \big),
\end{align*}
where we used Proposition \ref{propequivnorms} to derive the second inequality. The result then follows directly from Proposition \ref{characdotH1}.
\end{proof}


As a corollary, we can prove the fact that the operator $G_0$ is self-adjoint. Such a result is well known when $\Omega$ is bounded and when Dirichlet boundary conditions are imposed on the whole boundary of $\Omega$ (i.e. $\Gamma^{\rm N}=\emptyset$), even if the boundary of $\Omega$ is rough (see for instance \cite{ArendtElst}); for unbounded domains it is proved in \cite{RoussetTzvetkov,Lannes_book} in the context of water waves (the Dirichlet data is given on the surface of a strip of fluid, which is assumed to be a smooth graph). 

Proposition \ref{propDNell} allows us to simplify this proof and to generalize it to the present case of corner domains with mixed boundary conditions. We will not use this result in the rest of this article, but we state it for its independent interest. It is stated in the framework of standard non homogeneous Sobolev spaces.
\begin{corollary}
Let $\Omega$, $\Gamma^{\rm D}$ and $\Gamma^{\rm N}$ be as in Assumption \ref{assconfig}. 
The Dirichlet-Neumann operator $G_0$ defined in Definition \ref{propDN} admits a self-adjoint realization on $L^2(\Gamma^{\rm D})$ with domain $H^1(\Gamma^{\rm D})$.
\end{corollary}
\begin{proof}
The operator $G_0: {D}(G_0) \to L^2(\Gamma^{\rm D})$ is obviously symmetric since for all $\psi,\psi'\in {D}(G_0)$, one has
$$
 (G_0\psi ,\psi')_{L^2(\Gamma^{\rm D})}=\int_\Omega \nabla\psi^\mfh \cdot \nabla (\psi')^\mfh=(\psi,G_0\psi')_{L^2(\Gamma^{\rm D})};
$$
moreover, the domain $ {D}(G_0)$ is dense in $L^2(\Gamma^{\rm D})$ since by Proposition \ref{propDNcont} it contains $H^1(\Gamma^{\rm D})$. 

The operator $G_0$ is therefore closable; by abuse of notation, we still denote by $G_0$ its closure, which is also symmetric, so that ${D}(G_0)\subset {D}(G_0^*)$, where the domain of the adjoint operator  ${D}(G_0^*)$ is by definition given by
$$
{ D}(G_0^*)=\{ \psi \in L^2(\Gamma^{\rm D}), \quad \exists C>0, \forall \psi'\in {D}(G_0), \vert (\psi,G_0\psi')\vert \leq C \vert \psi' \vert_{L^2(\Gamma^{\rm D})} \}.
$$
For  $\psi\in L^2(\Gamma^{\rm D})$,  one can define $G_0\psi \in H^{-1}(\Gamma^{\rm D})$ by $\langle  G_0\psi,\psi' \rangle_{H^{-1}-H_0^1}=(\psi, G_0\psi')$ for all $\psi'\in H_0^1(\Gamma^{\rm D})$. 

Therefore, if $\psi\in {D}(G_0^*)$ one actually has $G_0\psi\in L^{2}(\Gamma^{\rm D})$; by Proposition \ref{propDNell}, this implies that $\psi\in H^1(\Gamma^{\rm D})$. To summarize, we have proved the inclusions $H^1(\Gamma^{\rm D})\subset {D}(G_0)\subset {D}(G_0^*)\subset H^1(\Gamma^{\rm D})$. It follows that all these inclusions are equalities, which proves the corollary.
\end{proof}

We finally prove a higher order ellipticity result stating basically that if $G_0\psi$ is in $H^{1/2}(\Gamma^{\rm D})$ it remains true that $\psi$ is one order more regular. Note that in the case where $\Omega$ is unbounded, we must consider $\psi \in {H}^{1/2}(\Gamma^{\rm D})$ rather than its homogeneous version $\dot{H}^{1/2}(\Gamma^{\rm D})$.
\begin{proposition}\label{propDNHO}
Let $\Omega$, $\Gamma^{\rm D}$ and $\Gamma^{\rm N}$ be as in Assumption \ref{assconfig}. \\ 
- If $\Omega$ is bounded then if $\psi\in \dot{H}^{1/2}(\Gamma^{\rm D})$ and  $G_0\psi \in H^{1/2}(\Gamma^{\rm D})$, one has $\dx \psi\in {H}^{1/2}(\Gamma^{\rm D})$, and moreover
$$
\vert \partial_x \psi \vert_{H^{1/2}(\Gamma^{\rm D})}\leq C\big( \vert \psi \vert_{\dot{H}^{1/2}(\Gamma^{\rm D})}+\vert G_0\psi\vert_{H^{1/2}(\Gamma^{\rm D})}\big),
$$
for some constant $C$ independent of $\psi$.\\
- If $\Omega$ is unbounded  then if $\psi\in {H}^{1/2}(\Gamma^{\rm D})$ and  $G_0\psi \in H^{1/2}(\Gamma^{\rm D})$, one has $\psi\in {H}^{3/2}(\Gamma^{\rm D})$, and moreover
$$
\vert  \psi \vert_{H^{3/2}(\Gamma^{\rm D})}\leq C\big( \vert \psi \vert_{{H}^{1/2}(\Gamma^{\rm D})}+\vert G_0\psi\vert_{H^{1/2}(\Gamma^{\rm D})}\big),
$$
for some constant $C$ independent of $\psi$.
\end{proposition}
\begin{proof}
We first prove the result in the case where $\Omega$ is bounded. By Proposition \ref{propDNell}, we have 
$$
\vert \partial_x \psi \vert_{L^2(\Gamma^{\rm D})}\leq C\big( \vert \psi \vert_{\dot{H}^{1/2}(\Gamma^{\rm D})}+\vert G_0\psi\vert_{L^{2}(\Gamma^{\rm D})}\big),
$$
so that, by Proposition \ref{propequivter}, it is enough to prove that
$$
\vert \partial_x \psi \vert_{\dot{H}^{1/2}(\Gamma^{\rm D})}\leq C\big( \vert \psi \vert_{\dot{H}^{1/2}(\Gamma^{\rm D})}+\vert G_0\psi\vert_{H^{1/2}(\Gamma^{\rm D})}\big)
$$
to get the result. 

An important step is to get estimate for the Neumann problem. This is done in the following lemma, where we use the following functional spaces,
$$
\dot{H}^2(\Omega):=\{u\in \dot{H}^{1}(\Omega), \quad\Vert u \Vert_{\dot{H}^2(\Omega)}:= \Vert \nabla u \Vert_{{H}^1(\Omega)} <\infty \}
$$
and, in the case where $\Omega$ is bounded and identifying functions $g$ on $\Gamma^{\rm D}$ as an $(N+1)$-uplet of functions defined on the intervals $\cE_j$ associated with the connected components of $\Gamma^{\rm D}$ (see Assumption \ref{assconfig} for the  notations),  
$$
{\mathbb H}:=\{f\in L^2(\Omega), g=(g_1,\dots,g_{N+1}) \in \prod_{j=1}^{N+1} {H}^{1/2}(\cE_j), \quad \int_\Omega f=\int_{\Gamma^{\rm D}} g \},
$$
endowed with the canonical norm of $L^2(\Omega)\times  \prod_{j=1}^{N+1} {H}^{1/2}(\cE_j)$.
\begin{lemma}\label{lemmaBVP}
Let $\Omega$ be as in Assumption \ref{assconfig}, and assume moreover that $\Omega$ is bounded. For all $(f,g)\in {\mathbb H}$, 
 there exists a  solution $u\in \dot{H}^2(\Omega)$, unique up to a constant, to the boundary value problem
 $$
\begin{cases}
\Delta u=f \quad\mbox{ in }\Omega,\\
\partial_{\rm n} u =0 \quad \mbox{ on }\Gamma^{\rm N},\\
\partial_{\rm n} u=g  \quad \mbox{ on }\Gamma^{\rm D},
\end{cases}
$$
and there exists a constant $C>0$ independent of $f$ and $g$ such that
$$
\Vert u \Vert_{\dot{H}^2(\Omega)}\leq C  \Vert (f,g)\Vert_{\mathbb H}.
$$
\end{lemma}
\begin{proof}[Proof of the lemma]
The existence and uniqueness up to a constant of a variational solution $u\in \dot{H}^1(\Omega)$ to the boundary value problem stated in the lemma is classical.

Since $\Omega$ is bounded and because the angles at the corners of $\Omega$ are assumed to be smaller than $\pi$, we know that $u \in H^2(\Omega)$ (see for instance Theorem 14.6 in \cite{Dauge}, Example 6.6.2 in \cite{KMR} or Section 3.4 in \cite{Poyferre}). 

The mapping
\begin{equation}\label{mappingelliptic}
 \begin{array}{lcl}
\dot{H}^2(\Omega)\backslash \RR & \to &  L^2(\Omega)\times  \prod_{j=1}^{N+1} {H}^{1/2}(\cE_j) \\
u & \mapsto &(\Delta u,(\partial_{\rm n} u)_{\vert_{\cE_1\times \{0\}}}, \dots, (\partial_{\rm n} u)_{\vert_{ \cE_{N+1}\times \{0\}}})
\end{array}
\end{equation}
is therefore an isomorphism of Banach spaces. 

Since it is continuous, the bounded inverse theorem implies that its inverse is also continuous, so that  there exists a constant $C>0$ such that for all $(f,g)\in {\mathbb H}$, the corresponding solution $u\in \dot{H}^2(\Omega)$ to the boundary value problem (defined up to a constant), satisfies 
$$
\Vert u \Vert_{\dot{H}^2(\Omega)}\leq C  \Vert (f,g)\Vert_{\mathbb H}.
$$
\end{proof}

We can now prove the proposition when $\Omega$ is bounded. Indeed, the harmonic extension $\psi^\mfh$ solves a boundary value problem belonging to the class considered in the lemma, namely, with $f=0$ and $g=G_0\psi$. We therefore have under the assumptions of the proposition that $\Vert \psi^\mfh \Vert_{\dot{H}^2(\Omega)}\lesssim \vert G_0\psi \vert_{H^{1/2}}$. In particular, $\Vert \dx \psi^\mfh \Vert_{\dot{H}^1(\Omega)}\lesssim \vert G_0\psi \vert_{H^{1/2}}$. By the trace Theorem \ref{theortrace}, and because the trace of $\dx\psi^\mfh $ on $\Gamma^{\rm D}$ is $\dx \psi$, this implies that $\vert \dx \psi \vert_{\dot{H}^{1/2}(\Gamma^{\rm D})}\lesssim\vert G_0\psi \vert_{H^{1/2}}$ which, as explained above, is enough to get the result.

Let us now consider the case where $\Omega$ is unbounded. As in the proof of Lemma \ref{lemmaext}, we consider for the sake of clarity
 the reference configuration of  Figure \ref{fig:image3} for which $N=1$ and ${\mathcal E}_1=(x^{\rm r}_0,x_1^{\rm l})$ and ${\mathcal E}_2=(x_1^{\rm r},+\infty)$ with $x_0^{\rm r}>-\infty$. The general case does not raise additional difficulty.

We denote as usual by $\psi^\mfh$ the harmonic extension of $\psi$;
since $\psi$ belongs to $H^{1/2}(\Gamma^{\rm D})$, we have $\psi^\mfh \in H^1(\Omega)$ and moreover $\Vert \psi^\mfh \Vert_{H^1(\Omega)}\lesssim \vert \psi \vert_{H^{1/2}(\Gamma^{\rm D})}$. We also know from Proposition \ref{propDNell} that $\psi\in H^1(\Gamma^{\rm D})$ and $\vert \psi \vert_{{H}^1(\Gamma^{\rm D})} \lesssim C \big( \vert \psi\vert_{{H}^{1/2}(\Gamma^{\rm D})}+ \vert G_0\psi\vert_{L^2(\Gamma^{\rm D})}\big)$ (we use the fact proved in Proposition \ref{propequivter} that $L^2(\Gamma^{\rm D})\cap \dot{H}^{1/2}(\Gamma^{\rm D})$ can be identified with ${H}^{1/2}(\Gamma^{\rm D})$).

Introduce $M>x^{\rm r}_1$  and $\chi_M$  a smooth positive cutoff function defined on $\overline{\Omega}$, such that $\chi_M\equiv 0$ on $\overline{\Omega}\cap \{x <M \}$ and $\chi_M\equiv 1$ on $\overline{\Omega}\cap \{x>M+1\}$, and satisfying $\partial_{\rm n} \chi_M=0$ on $\Gamma^{\rm D}\cup \Gamma^{\rm N}$. 

We further decompose
 $$
 \psi^\mfh=\phi_{\rm l}+\phi_{\rm r}
 \quad\mbox{ with }\quad \phi_{\rm l}=(1-\chi_M)\psi^\mfh \quad \mbox{ and }\quad \phi_{\rm r}=\chi_M\psi^\mfh.
 $$
 Since by definition $G_0\psi=(\dz \psi^\mfh)_{\vert_{\Gamma^{\rm D}}}$, we get that $(\partial_{\rm n} \phi_{\rm l})_{\vert_{\Gamma^{\rm D}}}=(1-\chi_M)G_0\psi$ and $(\partial_{\rm n} \phi_{\rm r})_{\vert_{\Gamma^{\rm D}}}=\chi_MG_0\psi$. Let us first examine $\phi_{\rm l}$ and then $\phi_{\rm r}$.

Let $\widetilde{\Gamma}^{{\rm b}}$ be a smooth curve parametrized by a function $\widetilde{b}$ defined on $(x_0^{\rm r}, M+2)$, coinciding with the bottom parametrization $b$ on $(x_0^{\rm r}, M+1)$, increasing on $(M+1,M+2)$ and such that $\lim_{x\to M+2}\widetilde{b}(x)=0$. 

The restriction of $\Gamma^{(\rm top)}$ and the adherence of $\widetilde{\Gamma}^{{\rm b}}$ enclose a bounded domain $\Omega_{\rm l}$ which satisfies Assumption \ref{assconfig}.
Moreover, the function $\phi_{\rm l}$ solves a boundary value problem that belongs to the class considered in Lemma \ref{lemmaBVP}, with
$f=-(\Delta \chi_M)\psi^\mfh- 2\nabla \chi_M\cdot \nabla \psi^\mfh$ and $g=(1-\chi_M)G_0\psi$. 

Since $\Vert f\Vert_{L^2(\Omega_{\rm l})}\lesssim \Vert \psi^\mfh \Vert_{H^1(\Omega)}$, we easily deduce from the lemma and the trace theorem that $(\phi_{\rm l})_{\vert_{\Gamma^{\rm D}\cap \{x<M+2\}}}\in H^{3/2}(\Gamma^{\rm D}\cap \{x<M+2\})$ and that its norm is bounded, up to a multiplicative constant, by $\vert G_0\psi\vert_{H^{1/2}(\Gamma^{\rm D})}+\vert \psi \vert_{H^{1/2}(\Gamma^{\rm D})}$.

By definition of $\phi_{\rm l}$ and because $(1-\chi_M)$ vanishes for $x\geq M+1$, this implies that
\begin{equation}\label{contleft}
\vert (1-\chi_M)\psi \vert_{H^{3/2}(\Gamma^{\rm D})}\lesssim \vert G_0\psi\vert_{H^{1/2}(\Gamma^{\rm D})}+\vert \psi \vert_{H^{1/2}(\Gamma^{\rm D})}.
\end{equation}

Let us turn now to investigate $\phi_{\rm r}$. We consider here a smooth extension $\widetilde{b}$ of $b$, which is defined on $\RR$ and such that $\widetilde{b}=b$ on $\{x>M\}$ and  $\inf_{\RR} \widetilde b>-\infty$, $\sup_{\RR} \widetilde b <0$.  We remark that the extension of ${\phi}_{\rm r}$ by $0$ on the strip ${\mathcal S}_{\widetilde{b}}:=\{(x,z)\in \RR^2,\widetilde b(x) <z<0 \}$, denoted by $\widetilde{\phi_{\rm r}}$ solves
$$
\begin{cases}
\Delta \widetilde{\phi_{\rm r}}=-\widetilde{f} &\mbox{ in } {\mathcal S}_{\widetilde{b}},\\
\partial_{\rm n} \widetilde{\phi_{\rm r}}= \chi_M\widetilde{G_0\psi} &\mbox{ on } \{z=0\},\\
\partial_{\rm n} \widetilde{\phi_{\rm r}}= 0 &\mbox{ on } \{z=\widetilde{b}\},
 \end{cases}
$$
where $\widetilde{f}$ denotes the extension of $f$ by zero on ${\mathcal S}_{\widetilde{b}}$ and $\widetilde{G_0\psi}$ the extension by zero of $G_0\psi$ to $\RR$.

The classical variational estimate yields $\Vert \nabla \widetilde{\phi_{\rm r}}\Vert_{L^2({\mathcal S}_{\widetilde{b}})}\lesssim \Vert \psi^\mfh \Vert_{H^1(\Omega)}$. If $\widetilde{b}$ is constant, then applying $\partial_x$ to the equation (we actually should use a mollified version of $\dx$, or differential quotients and then pass to the limit  in order to manipulate only meaningful quantities, but we omit these classical technicalities for the sake of clarity and refer for instance to the proof of Lemma 2.38 in \cite{Lannes_book} for the details), we obtain
$$
\begin{cases}
\Delta \dx\widetilde{\phi_{\rm r}}=\dx\widetilde{f}&\mbox{ in } {\mathcal S}_{\widetilde{b}},\\
\partial_{\rm n} \dx\widetilde{\phi_{\rm r}}= \dx (\chi_M\widetilde{G_0\psi}) &\mbox{ on } \{z=0\},\\
\partial_{\rm n} \dx \widetilde{\phi_{\rm r}}= 0 &\mbox{ on } \{z=\widetilde{b}\}.
 \end{cases}
$$
We therefore get the variational estimate
\begin{align*}
\Vert \nabla \dx \widetilde{\phi_{\rm r}}\Vert_{L^2({\mathcal S}_{\widetilde{b}})}^2 &\leq \big\vert \int_{\{z=0\}}  \dx (\chi_M\widetilde{G_0\psi}) \dx\widetilde{\phi_{\rm r}} \big\vert
+\big\vert \int_{{\mathcal S}_{\widetilde{b}} }\dx\widetilde{f} \dx  \widetilde{\phi_{\rm r}}\big\vert \\
&\lesssim \big( \vert G_0\psi \vert_{H^{1/2}(\Gamma^{\rm D})}  + \Vert \widetilde{f}\Vert_{L^2({\mathcal S}_{\widetilde{b}})} \big)\Vert  \Vert  \dx \widetilde{\phi_{\rm r}}\Vert_{H^1({\mathcal S}_{\widetilde{b}})},
\end{align*}
where we used the trace theorem to control the first term of the right-hand side in the first inequality, and integrate by parts to control the second one.

We can therefore conclude that
\begin{align*}
\Vert \dx  \widetilde{\phi_{\rm r}}\Vert_{ H^1({\mathcal S}_{\widetilde{b}})} &\lesssim (\Vert \psi^\mfh \Vert_{H^1(\Omega)}+ \vert G_0\psi \vert_{H^{1/2}(\Gamma^{\rm D})} )  \\
&  \lesssim (\vert \psi \vert_{H^{1/2}(\Omega)}+ \vert G_0\psi \vert_{H^{1/2}(\Gamma^{\rm D})} ).
\end{align*}
We then deduce from the trace theorem and the definition of $ \widetilde{\phi_{\rm r}}$ that
\begin{equation}\label{contright}
\vert \chi_M\psi \vert_{H^{3/2}(\Gamma^{\rm D})}\lesssim \vert G_0\psi\vert_{H^{1/2}(\Gamma^{\rm D})}+\vert \psi \vert_{H^{1/2}(\Gamma^{\rm D})}.
\end{equation}

When $\widetilde{b}$ is not constant, then one can go back to the case of a flat strip using a diffeomorphism as in Chapter 2 of \cite{Lannes_book} and show without difficulty that this estimates still holds.

The result is then a consequence of \eqref{contleft} and \eqref{contright}. 
 \end{proof}

\section{Well-posedness in the energy space and consequences}
\label{sectWP}

We show in this section that the evolution equation
\begin{equation}\label{CP}
\begin{cases}
\dt \zeta - G_0 \psi=f,\\
\dt \psi+\gr \zeta=g
\end{cases}
\quad\mbox{ on }\quad \RR^+\times \Gamma^{\rm D}
\end{equation}
with initial condition
\begin{equation}\label{CP0}
(\zeta,\psi)_{\vert_{t=0}}=(\zeta^{\rm in},\psi^{\rm in})
\end{equation}
is well posed for data in the energy space, and regular in time under additional assumptions. 

We will often write \eqref{CP} in the abstract form
\begin{equation}\label{CPcomp}
\partial_t U +{\bf A}U=F,
\end{equation}
with $U=(\zeta,\psi)^{\rm T}$, $F=(f,g)^{\rm T}$ and 
$$
{\bf A}=\begin{pmatrix} 0 & -G_0  \\  \gr  & 0\end{pmatrix}.
$$

A difficulty when dealing with the Cauchy problem \eqref{CP}-\eqref{CP0} is that the natural energy space $L^2(\Gamma^{\rm D})\times \dot{H}^{1/2}(\Gamma^{\rm D})$ associated with the equations is only a semi-normed space. 

When $\Gamma^{\rm D}$ is bounded, it is possible to remove this difficulty by working with $\psi$ in the realization $\dot{\mathcal H}^{1/2}(\Gamma^{\rm D})$ of  $\dot{H}^{1/2}(\Gamma^{\rm D})$ introduced in \S \ref{sectdotH012}; indeed, we know by Proposition \ref{propBanach} that $\dot{\mathcal H}^{1/2}(\Gamma^{\rm D})$ is a Banach space. Imposing in addition a zero mass condition on $\zeta$ (i.e. $\int_{\Gamma^{\rm D}}\zeta=0$), we can show in \S \ref{subsectskew} that the operator ${\bf A}$ is skew-adjoint, from which the well-posedness of \eqref{CP}-\eqref{CP0} is deduced  in \S \ref{subsectWPbnded} using semi-group theory. 

When $\Gamma^{\rm D}$ is unbounded, this approach  no longer works since the zero mass condition $\int_{\Gamma^{\rm D}}\zeta=0$ does not make sense anymore. We therefore construct a solution by a duality method in a semi-normed functional space of function of space \emph{and} time. Instead of working simply with a realization of  $\dot{H}^{1/2}(\Gamma^{\rm D})$ we need to introduce a realization of this functional space that involves also the time variable. We show that this realization can be physically interpreted as a convenient choice of the Bernoulli constant. 
A well posedness result is established in this framework in \S \ref{subsectWPunbnded}. 

Under additional assumptions on the data, we show in \S \ref{sectHOTR} that it is possible to construct solutions that are more regular in time. We then investigate in \S \ref{sectlimreg} whether it is possible to deduce space regularity from this time regularity; we show in particular that proceeding like this, one can reach the $H^1(\Gamma^{\rm D})\times \dot{H}^{3/2}(\Gamma^{\rm D})$ regularity but  going above this threshold requires smallness assumptions on the angles at the corners of the fluid domain.

\subsection{Skew-adjointness of the evolution operator when $\Gamma^{\rm D}$ is bounded}\label{subsectskew}

When $\Gamma^{\rm D}$ is bounded, then we recall that $\dot{\mathcal H}^{1/2}(\Gamma^{\rm D})$ consists of all the $f \in \dot{H}^{1/2}(\Gamma^{\rm D})$ such that $\int_{\Gamma^{\rm D}}f=0$. We define similarly
\begin{equation}\label{calL2}
\begin{cases}
{\mathcal L}^2(\Gamma^{\rm D})&=\{ f\in L^2(\Gamma^{\rm D}), \int_{\Gamma^{\rm D}} f=0\}, \\
{\mathcal H}^s(\Gamma^{\rm D})&=\{ f\in H^s(\Gamma^{\rm D}), \int_{\Gamma^{\rm D}} f=0\}, \quad (s=0,1/2).
\end{cases}
\end{equation}

Let us define also
$$
{\mathbb X}={\mathcal L}^2(\Gamma^{\rm D})\times \dot{\mathcal H}^{1/2}(\Gamma^{\rm D});
$$
 it forms a Hilbert space for the scalar product
$$
\langle (\zeta_1,\psi_1),(\zeta_2,\psi_2)\rangle_{\mathbb X}=\gr (\zeta_1,\zeta_2)_{L^2(\Gamma^{\rm D})}+\langle G_0\psi_1,\psi_2\rangle,
$$
where we recall that by definition $\langle G_0\psi_1,\psi_2\rangle=\int_\Omega \nabla\psi_1^\mfh\cdot \nabla\psi_2^\mfh$; indeed, by Proposition \ref{propequivnorms2}, $\langle G_0\psi,\psi \rangle$ defined a semi-norm equivalent to $\vert \cdot \vert_{\dot{H}^{1/2}(\Gamma^{\rm D})}$ and, by Proposition \ref{propBanach} is therefore a norm on $\dot{\mathcal H}^{1/2}(\Gamma^{\rm D})$, and $\dot{\mathcal H}^{1/2}(\Gamma^{\rm D})$ is a Hilbert space for the scalar product $\langle G_0\psi_1,\psi_2\rangle$. 

The following proposition shows that ${\bf A}$ is skew-adjoint for this scalar product.
\begin{proposition}\label{propSA}
Let $\Omega$, $\Gamma^{\rm D}$ and $\Gamma^{\rm N}$ be as in Assumption \ref{assconfig}. 
If $\Gamma^{\rm D}$ is bounded then the operator ${\bf A}$ admits a skew-adjoint realization on ${\mathbb X}$ with domain ${\mathcal H}^{1/2}(\Gamma^{\rm D})\times \dot{\mathcal H}^1(\Gamma^{\rm D})$. \end{proposition}
\begin{proof}
The operator ${\bf A}: D({\bf A})\to {\mathbb X}$ is obviously skew-symmetric since for all $U_1=(\zeta_1,\psi_1)$, $U_2=(\zeta_2,\psi_2)$ in $D({\bf A})$, one has
$$
\langle U_1, {\bf A}U_2\rangle_{{\mathbb X}}=-\gr (\zeta_1, G_0\psi_2)+\gr (\psi_1,G_0\zeta_2).
$$
This shows that $\langle U_1, {\bf A}U_2\rangle_{{\mathbb X}}=-\langle U_2, {\bf A}U_1\rangle_{{\mathbb X}}$ and therefore that ${\bf A}$ is skew-symmetric.

Moreover, the domain $D({\bf A})$ is given by
$$
D({\bf A})=\big({\mathcal L}^2(\Gamma^{\rm D})\cap \dot{\mathcal H}^{1/2}(\Gamma^{\rm D})\big)\times \{ \psi \in \dot{\mathcal H}^{1/2}(\Gamma^{\rm D}), G_0\psi \in {\mathcal L}^2(\Gamma^{\rm D}) \};
$$
since $\int_{\Gamma^{\rm D}} G_0\psi=0$ for all $\psi \in \dot{H}^{1/2}(\Gamma^{\rm D})$ by Green's identity, and using Proposition \ref{propequivter}, we can rewrite
$$
D({\bf A})={\mathcal H}^{1/2}(\Gamma^{\rm D})\big)\times \{ \psi \in \dot{\mathcal H}^{1/2}(\Gamma^{\rm D}), G_0\psi \in L^2(\Gamma^{\rm D}) \}.
$$

By Propositions  \ref{propDNcont} and \ref{propDNell} we get that $\{ \psi \in \dot{\mathcal H}^{1/2}(\Gamma^{\rm D}), G_0\psi \in L^2(\Gamma^{\rm D}) \}= \dot{\mathcal H}^{1/2}(\Gamma^{\rm D})\cap \dot{H}^1(\Gamma^{\rm D})$, and therefore that $D({\bf A})={\mathcal H}^{1/2}(\Gamma^{\rm D})\big)\times \dot{\mathcal H}^1(\Gamma^{\rm D})$. Since $D({\bf A})$ contains ${\mathcal D}(\Gamma^{\rm D})\times {\mathcal D}(\Gamma^{\rm D})$, it is dense in ${\mathbb X}$ by Proposition \ref{propdense2}. 
 The operator ${\bf A}$ is therefore closable; by abuse of notation, we still denote by ${\bf A}$ its closure, which is also symmetric, so that $D({\bf A})\subset D({\bf A}^*)$, where the domain of the adjoint operator $D({\bf A}^*)$ is by definition given by
$$
D({\bf A}^*)=\{ U \in {\mathbb X}, \quad \exists C>0, \forall U'\in {D}({\bf A}), \vert \langle U,{\bf A} U'\rangle_{\mathbb X}\vert \leq C \vert U' \vert_{{\mathbb X}} \}.
$$

In particular, if $U\in D({\bf A}^*)$ then $\zeta \in {\mathcal L}^2(\Gamma^{\rm D})$ and $\psi' \in \dot{\mathcal H}^1(\Gamma^{\rm D})\mapsto (\zeta, G_0 \psi')$ can be extended as a continuous linear form on $\dot{\mathcal H}^{1/2}(\Gamma^{\rm D})$; by Riesz theorem, there exists a unique $ g \in \dot{\mathcal H}^{1/2}(\Gamma^{\rm D})$ such that $ (\zeta, G_0 \psi')=\langle g, G_0 \psi' \rangle$, for all $\psi'\in \dot{\mathcal H}^{1/2}(\Gamma^{\rm D})$. It follows that $\zeta=g$ and therefore $\zeta\in {\mathcal H}^{1/2}(\Gamma^{\rm D})$.

One also has that $\psi \in \dot{\mathcal H}^{1/2}(\Gamma^{\rm D})$ and that
$$
\forall \zeta'\in {\mathcal H}^{1/2}(\Gamma^{\rm D}), \qquad \langle G_0\psi,\zeta'\rangle \leq C \vert \zeta'\vert_{L^2(\Gamma^{\rm D})}.
$$
If $\zeta' \in H^{1/2}(\Gamma^{\rm D})$, then denoting $\langle \zeta' \rangle := \frac{1}{\vert \Gamma^{\rm D}\vert}\int_{\Gamma^{\rm D}}\zeta'$, we have $\zeta'-\langle\zeta'\rangle \in {\mathcal H}^{1/2}(\Gamma^{\rm D})$, and therefore
$ \langle G_0\psi,(\zeta'-\langle \zeta'\rangle)\rangle \leq C \vert \zeta'-\langle \zeta'\rangle\vert_{L^2(\Gamma^{\rm D})}$. 

Since moreover $\langle G_0\psi, \langle \zeta'\rangle \rangle=0$ and because $\vert \zeta' -\langle \zeta'\rangle \vert^2_{L^2(\Gamma^{\rm D})} \leq 2\vert \zeta' \vert^2_{L^2(\Gamma^{\rm D})}$, we deduce that
$$
\forall \zeta'\in {H}^{1/2}(\Gamma^{\rm D}), \qquad \langle G_0\psi,\zeta'\rangle \leq 2 C \vert \zeta'\vert_{L^2(\Gamma^{\rm D})}.
$$
This shows that $G_0\psi$, which is well defined in $\dot{\mathcal H}^{1/2}(\Gamma^{\rm D})'$, is actually in $ L^2(\Gamma^{\rm D})$. By Proposition \ref{propDNell}, this implies that $\psi\in \dot{\mathcal H}^1(\Gamma^{\rm D})$. 

We have therefore proved the inclusions ${\mathcal H}^{1/2}(\Gamma^{\rm D}) \times \dot{\mathcal H}^1(\Gamma^{\rm D})\subset D({\bf A})\subset D({\bf A}^*)\subset {\mathcal H}^{1/2}(\Gamma^{\rm D}) \times \dot{\mathcal H}^1(\Gamma^{\rm D})$. It follows that all these inclusions are equalities, which proves the proposition. 
\end{proof}

\subsection{Well-posedness theorem when $\Gamma^{\rm D}$ is bounded} \label{subsectWPbnded}

We can now prove that the Cauchy problem \eqref{CP}-\eqref{CP0} is well posed, in the weak sense in the energy space ${\mathcal L}^2(\Gamma^{\rm D})\times \dot{\mathcal H}^{1/2}(\Gamma^{\rm D})$ and classically in ${\mathcal H}^{1/2}(\Gamma^{\rm D})\times \dot{\mathcal H}^{1}(\Gamma^{\rm D})$. 

Let us first recall the definition of the Hilbert space ${\mathbb X}$ and introduce the space  ${\mathbb X}_1$,
\begin{equation}\label{defXbounded}
{\mathbb X}={\mathcal L}^2(\Gamma^{\rm D})\times \dot{\mathcal H}^{1/2}(\Gamma^{\rm D})
\quad\mbox{ and }\quad
{\mathbb X}^1={\mathcal H}^{1/2}(\Gamma^{\rm D})\times \dot{\mathcal H}^{1}(\Gamma^{\rm D}),
\end{equation}
where we recall that for $U=(\zeta,\psi)^{\rm T}$, one has 
\begin{align*}
\vert  U\vert_{\mathbb X}^2&=\langle U,U\rangle_{\mathbb X}\\
&={\mathtt g}\vert \zeta \vert_{L^2({\Gamma^{\rm D}})}^2+\langle G_0 \psi, \psi \rangle;
\end{align*}
we also recall  the fact that $\vert \cdot \vert_{\mathbb X}$ is equivalent to the canonical norm of ${\mathcal L}^2(\Gamma^{\rm D})\times \dot{\mathcal H}^{1/2}(\Gamma^{\rm D})$ stemming from Proposition \ref{propequivnorms2}. We also endow ${\mathbb X}^1$  with its canonical norm. 
\begin{theorem}\label{theoWPbounded}
Let $\Omega$, $\Gamma^{\rm D}$ and $\Gamma^{\rm N}$ be as in Assumption \ref{assconfig}, and ${\mathbb X}$ and ${\mathbb X}^1$ be as defined in \eqref{defXbounded}, and assume moreover that $\Gamma^{\rm D}$ is bounded. 
Let $F=(f,g)^{\rm T}\in C(\RR^+;{\mathbb X})$. For all $U=(\zeta^{\rm in},\psi^{\rm in})\in {\mathbb X}$, there is a unique weak solution $U\in C(\RR^+;  {\mathbb X})$ to \eqref{CP}-\eqref{CP0}. Moreover, one has
$$
\forall t\geq 0, \qquad \vert U(t) \vert_{\mathbb X}\leq \vert U^{\rm in} \vert_{\mathbb X} +\int_0^t \vert F(t') \vert_{\mathbb X}{\rm d}t',
$$
with equality if $F=0$.

If in addition $(\zeta^{\rm in},\psi^{\rm in})\in {\mathbb X}^1$ and $F \in  C^1(\RR^+;{\mathbb X})$ then the solution $U$ belongs to $C(\RR^+;  {\mathbb X}^1)\cap C^1(\RR^+; {\mathbb X}) $.
\end{theorem} 
\begin{remark}\label{remzeromasszeta1}
The zero mass assumption is not restrictive. Indeed, if $c_1=\int_{\Gamma^{\rm D}}\zeta^{\rm in}$ and $c_2=\int_{\Gamma^{\rm D}}\psi^{\rm in}$ are such that $\vert c_1 \vert+\vert c_2\vert>0$, one can look for $(\zeta,\psi)$ under the form $\zeta=c_1+\tilde{\zeta}$ and $\psi=c_2-\gr c_1 t +\tilde{\psi}$; then $(\tilde{\zeta},\tilde{\psi})$ solves the same equation, but with initial data that satisfy the zero mass assumption. 

\end{remark} 
\begin{proof}
Since ${\bf A}$ is skew-adjoint on ${\mathbb X}$, it generates a unitary group on ${\mathbb X}$; this proves the first part of the theorem when $F=0$ since the conservation of energy is a consequence of the fact group generated by ${\bf A}$ is unitary. Moreover, if $U_0\in D({\bf A})$, then there is a unique solution $U\in C( \RR^+; {\mathbb X}^1)\cap C^1(\RR^+;{\mathbb X})$ to the equation $\dot{U}+{\bf A}U=0$ such that $U(t=0)=U_0$, and where ${\mathbb X}^1$ denotes $D({\bf A})$ equipped with the graph norm. As seen in the proof of Proposition \ref{propSA}, ${\mathbb X}^1$ can be identified with ${\mathcal H}^{1/2}\times \dot{\mathcal H}^1(\Gamma^{\rm D})$.
This completes the proof in the case $F=0$. The general case classically follows from Duhamel's formula.
\end{proof}

\subsection{Well-posedness theorem when $\Gamma^{\rm D}$ is possibly unbounded}\label{subsectWPunbnded}

The proof of Theorem \ref{theoWPbounded} relied on the fact that it was possible to work with the realization $\dot{\mathcal H}^{1/2}(\Gamma^{\rm D})$ of $\dot{H}^{1/2}(\Gamma^{\rm D})$ provided that we imposed the zero mass condition $\int_{\Gamma^{\rm D}}\zeta=0$ on $\zeta$. 

In the unbounded case, this is no longer possible because the integral  $\int_{\Gamma^{\rm D}}\zeta$  might not be defined. It is however possible to state the following result, but with the spaces ${\mathbb X}$ and ${\mathbb X}^1$ now defined as
\begin{equation}\label{defXunbounded}
{\mathbb X}={ L}^2(\Gamma^{\rm D})\times \dot{H}^{1/2}(\Gamma^{\rm D})
\quad\mbox{ and }\quad
{\mathbb X}^1={H}^{1/2}(\Gamma^{\rm D})\times \dot{H}^{1}(\Gamma^{\rm D});
\end{equation}
note that contrary to the definition \eqref{defXbounded} of these spaces in the bounded case, ${\mathbb X}$ and ${\mathbb X}_1$ are now only semi-normed spaces.
\begin{remark}
Let us consider the homogeneous problem $\dt U +{\bf A}U=0$ with initial condition $U_{\vert_{t=0}}=0$. The fact that ${\mathbb X}$ is a semi-norm space implies that uniqueness cannot be deduced from energy conservation. Indeed, any couple of the form $(\zeta,\psi)(t,x)=(0,\underline{\psi}(t))$ has zero energy. However among those functions, only those for which $\underline{\psi}(t)\equiv\underline{\psi}$ is time independent satisfy $\dt U +{\bf A}U=0$ almost everywhere on $\Gamma^{\rm D}$, and among those, there is only one, corresponding to $\underline{\psi}=0$ that solves the  initial condition $U_{\vert_{t=0}}=0$. In the theorem below, when we say that there is a unique solution $U\in C(\RR^+;  {\mathbb X})$ to \eqref{CP}-\eqref{CP0} we understand that the equation and the initial condition on $\psi$ are not only satisfied in $\dot{H}^{1/2}(\Gamma^{\rm D})$ but also almost everywhere.
\end{remark}
In the statement below, compared to Theorem \ref{theoWPbounded}, the assumptions that $\Gamma^{\rm D}$ is bounded is removed, and no zero mass assumption is made on the data.
\begin{theorem}\label{theoWPunbounded}
Let $\Omega$, $\Gamma^{\rm D}$ and $\Gamma^{\rm N}$ be as in Assumption \ref{assconfig},  and ${\mathbb X}$ and ${\mathbb X}^1$ be as defined in \eqref{defXunbounded}.   
Let $F=(f,g)^{\rm T}\in C(\RR^+;{\mathbb X})$. For all $U^{\rm in}=(\zeta^{\rm in},\psi^{\rm in})\in {\mathbb X}$, there is a unique solution $U\in C(\RR^+;  {\mathbb X})$ to \eqref{CP}-\eqref{CP0}. Moreover, one has
$$
\forall t\geq 0, \qquad \vert U(t) \vert_{\mathbb X}^2\leq \vert U^{\rm in} \vert_{\mathbb X}^2 +\int_0^t \vert F(t') \vert_{\mathbb X}^2{\rm d}t',
$$
with equality if $F=0$.

If in addition $(\zeta^{\rm in},\psi^{\rm in})\in {\mathbb X}^1$ and $F\in C^1(\RR^+;{\mathbb X})$ then the solution $U$ belongs to $C(\RR^+;  {\mathbb X}^1)\cap C^1(\RR^+; {\mathbb X}) $.
\end{theorem}
\begin{remark}
In the \emph{bounded} case, we defined the space ${\mathbb X}$ as ${\mathbb X}={\mathcal L}^2(\Gamma^{\rm D})\times \dot{\mathcal H}^{1/2}(\Gamma^{\rm D})$ where we recall that $\dot{\mathcal H}^{1/2}(\Gamma^{\rm D})$  is the realization of the semi-normed space $\dot{H}^{1/2}(\Gamma^{\rm D})$ consisting of its elements that have zero mean on $\Gamma^{\rm D}$.  
In order to solve the equations  \eqref{CP}-\eqref{CP0}, we had to impose a similar zero mean condition on $\zeta$; it was crucial in the proof of Theorem \ref{theoWPbounded} that these zero mean conditions were propagated by the equations. 
In the \emph{unbounded} case, one can also define a realization $\dot{\mathcal H}^{1/2}(\Gamma^{\rm D})$ of $\dot{\mathcal H}^{1/2}(\Gamma^{\rm D})$ (see Section \ref{sectdotH012}); unfortunately, this realization is not compatible with the equations in the sense that the constraint used in \eqref{defHI120} to define this realization is not propagated anymore by the equations. As we show below, instead of working with a realization of the space $\dot{H}^{1/2}(\Gamma^{\rm D})$, which is a homogeneous space of functions of the space variable, we need to work with a realization of the space $L^2([0,T];\dot{H}^{1/2}(\Gamma^{\rm D})$, which is a homogeneous space of functions of the space and time variables.
 The idea is the following: we construct a solution by a duality method in $L^2([0,T]; {\mathbb X})$; with ${\mathbb X}$ now defined by \eqref{defXunbounded}; since this space is only semi-normed, the equations \eqref{CP} are only satisfied up to some element of the adherence of $0$ in $L^2([0,T]; {\mathbb X})$. This can be physically interpreted as the fact that the Bernoulli equation (the second equation in \eqref{CP}) is satisfied only up to a constant in time. Setting this constant equal to $0$ (and imposing that the initial condition is satisfied almost everywhere) provides us with the relevant realization of $L^2([0,T]; {\mathbb X})$.  
\end{remark}
\begin{remark}\label{remzeromasszeta2}
When $\Gamma^{\rm D}$ is bounded, one can apply Theorem \ref{theoWPunbounded} if the zero mass assumption does not hold, instead of proceeding as in Remark \ref{remzeromasszeta1}. 
\end{remark}
\begin{proof}
The structure of the proof is the following: we first show a priori estimates for regular enough solutions; by a duality argument, we then construct a weak solution. In order to prove that it is unique, we use a mollification in time of the equations to approximate this weak solution by a convergent sequence of more regular solutions when the initial data is equal to zero. The energy estimate applies to these approximations and passing to the limit, we deduce that the weak solution also satisfies it; it is therefore unique. We then extend the result to the case of non-zero initial data, and prove the higher order regularity stated in the second part of the theorem. The a priori energy estimate is provided by the following lemma. 
\begin{lemma}\label{lemmaNRJest}
Let $T>0$. There exists $C>0$ such that for all  $V\in H^1([0,T];{\mathbb X})\cap L^2([0,T];{\mathbb X}^1)$ one has, for all $0\leq t\leq T$,
$$
\vert V(t) \vert_{\mathbb X}^2 \leq C \big( \vert V(0) \vert_{\mathbb X}^2 + \int_0^t \vert (\dt +{\bf A}) V \vert_{{\mathbb X}}^2 \big).
$$
\end{lemma}
\begin{proof}[Proof of the lemma] 
Let us write $G=  (\dt +{\bf A}) V $. Taking the ${\mathbb X}$ (non-definite) scalar product of this equation with $V$, and since $\langle {\bf A}U,U\rangle_{{\mathbb X}}$ is well defined and equal to $0$ when $U\in {\mathbb X}^1$, we get
$$
\frac{1}{2}\frac{{\rm d}}{{\rm d} t } \vert V \vert^2_{{\mathbb X}}=\langle G, V \rangle_{{\mathbb X}},
$$
from which the result follows easily.
\end{proof}
\begin{remark}\label{remreverse}
Reversing in time, we also get that for $\forall V\in H^1([0,T];{\mathbb X})\cap L^2([0,T];{\mathbb X}^1)$ one has, for all $0\leq t\leq T$,
$$
\vert V(t) \vert_{\mathbb X}^2 \leq C \big( \vert V(T) \vert_{\mathbb X}^2 + \int_t^T \vert (\dt +{\bf A}) V \vert_{{\mathbb X}}^2 \big).
$$
\end{remark}
This a priori estimate can be used to run a duality argument and establish the existence of a weak solution. Special attention must be paid to the fact that we work with a semi-normed space. Taking the quotient of this space by the adherence of zero for the seminorm makes it a Banach space; the proof shows that setting the Bernoulli constant equal to zero in the second equation of \eqref{CP} is equivalent to choosing a particular representative of the solution in this quotient space.

\begin{lemma}\label{lemmaweaksol}
There exists a weak solution $U\in L^2([0,T];{\mathbb X})$ to the initial value problem \eqref{CP}-\eqref{CP0}.
\end{lemma}
\begin{proof}
Let us introduce the space
$$
{\mathcal E}=\{V \in H^1([0,T] ; {\mathbb X}) \cap L^2([0,T];{\mathbb X}^1), \quad V(T)=0\};
$$
from Remark \ref{remreverse}, we get that for all $V \in {\mathcal E}$ and $0\leq t\leq T$, one has
$$
\vert V(t) \vert_{\mathbb X}^2 \leq C  \int_0^T \vert (\dt +{\bf A}) V \vert_{{\mathbb X}}^2 ;
$$
the operator $(-\dt -{\bf A})$ is therefore one-to-one on ${\mathcal E}$ and we can therefore define a linear form $\ell$ on $(-\dt -{\bf A}){\mathcal E}$ by
$$
\ell((-\dt -{\bf A})V)=\int_0^T \langle F(t), V(t) \rangle_{{\mathbb X}}{\rm d}t +\langle U^{\rm in}, V(0) \rangle_{{\mathbb X}}.
$$
Using the above estimate, we easily get that
$$
\ell((-\dt -{\bf A})V)^2 \leq C\big(\vert F \vert_{L^2([0,T];{\mathbb X})}^2 + \vert U(0)\vert_{{\mathbb X}}^2 \big)   \vert (-\dt -{\bf A}) V \vert_{L^2([0,T];{\mathbb X})}^2,
$$
so that $\ell$ is continuous for the semi-norm $\vert \cdot \vert_{L^2([0,T];{\mathbb X})}^2$. \\
By the Hahn-Banach theorem on semi-normed spaces, $\ell$ can be extended as a linear continuous form on $L^2([0,T];{\mathbb X})$. We can canonically associate to $\ell$ a continuous linear form, denoted by $\ell^\bullet$, on the Hilbert space $L^2([0,T];{\mathbb X})/K$, where $K$ is the adherence of $0$ in $L^2([0,T];{\mathbb X})$, which consists in all functions $V\in L^2([0,T];{\mathbb X})$ such that $\int_0^T \vert V(t)\vert^2_{{\mathbb X}}{\rm d}t=0$, that is, $K=\{0\}\times L^2([0,T];\RR)$. \\
Since $L^2([0,T];{\mathbb X})/K$ is a Hilbert space, one can use the Riesz representation theorem to deduce that there exists $U^\bullet \in L^2([0,T];{\mathbb X})/K$ such that
$$
\ell^\bullet\big( ((-\dt -{\bf A})V)^\bullet\big)= \langle U^\bullet , (-\dt -{\bf A})V \rangle_{L^2([0,T];{\mathbb X})},
$$
for all $V\in {\mathcal E}$, and where $((-\dt -{\bf A})V)^\bullet$ denotes the class of $(-\dt -{\bf A})V$ in $L^2([0,T];{\mathbb X})/K$. \\
Since by definition $\ell^\bullet\big( ((-\dt -{\bf A})V)^\bullet\big)=\ell\big( (-\dt -{\bf A})V\big)$, this means by the definition of $\ell$ that
\begin{equation}\label{repres}
\int_0^T \langle F(t), V(t) \rangle_{{\mathbb X}}{\rm d}t +\langle U^{\rm in}, V(0) \rangle_{{\mathbb X}}
=\int_0^T\langle \widetilde{U}(t), (-\dt -{\bf A})V(t) \rangle_{\mathbb X}{\rm d}t,
\end{equation}
where $\widetilde{U}=(\widetilde{\zeta},\widetilde{\psi})^{\rm T} \in L^2([0,T];{\mathbb X})$ is any representative of $U^\bullet$.\\
For all $V \in {\mathcal D}(\Gamma^{\rm D})^2$, we have
$$
\int_0^T \langle F(t), V(t) \rangle_{{\mathbb X}}{\rm d}t 
=\int_0^T\langle  (\dt +{\bf A}) \widetilde{U}(t),V(t) \rangle_{\mathbb X}{\rm d}t.
$$
This implies that $ (\dt +{\bf A}) \widetilde{U}- F$ belongs to $K$; in particular, there exists a function ${\mathfrak c} \in L^2([0,T];\RR)$ such that
$$
\begin{cases}
\dt \widetilde{\zeta}-G_0\widetilde{\psi}=f,\\
\dt \widetilde{\psi} +{\mathtt g} \widetilde{\zeta}={\mathfrak c} +g
\end{cases}
$$
is satisfied almost-everywhere.
It is possible to take ${\mathfrak c}=0$ by taking a different representative of $U^\bullet$, namely, ${U}=(\zeta,\psi)$ with $\zeta=\widetilde{\zeta}$ and  $\psi=\widetilde{\psi}-\int_0^t {\mathfrak c}(t'){\rm d}t'$; this representative is still defined up to a time independent constant. We can then use the equation to get that $\dt U \in L^2([0,T];(\dot{H}^{1/2}(\Gamma^{\rm D})'+L^2(\Gamma^{\rm D}))\times (L^2(\Gamma^{\rm D})+\dot{H}^{1/2}(\Gamma^{\rm D}))$ and therefore $U \in C([0,T];(\dot{H}^{1/2}(\Gamma^{\rm D})'+L^2(\Gamma^{\rm D}))\times (L^2(\Gamma^{\rm D})+\dot{H}^{1/2}(\Gamma^{\rm D}))$, and we get from \eqref{repres} that $U(0)=U^{\rm in}+(0,{\mathfrak c}_0)^{\rm T}$, for some ${\mathfrak c}_0\in \RR$. Since $\psi$ is defined up to a constant, we can choose in such a way that ${\mathfrak c}_0=0$.
\end{proof}
We can now show that the weak solution constructed above can be approximated by a sequence of strong solutions.
\begin{lemma}\label{lemmaunique}
Let $F\in L^2([0,T];{\mathbb X})$. The Cauchy problem \eqref{CP}-\eqref{CP0} with initial data $U^{\rm in}=0$, has a unique weak solution $U$. Moreover, $U \in C([0,T]; {\mathbb X})$ and satisfies the energy estimate of Lemma \ref{lemmaNRJest}, and it is a limit in $C([0, T ]; {\mathbb X})$ of a sequence $U^\epsilon \in H^1([0,T];{\mathbb X})\cap L^2([0,T];{\mathbb X}^1)$ such that $U^\epsilon(0)=0$ and $(\dt +{\bf A})U^\epsilon \to F$ in $L^2([0,T];{\mathbb X})$ as $\epsilon \to 0$.
\end{lemma}

\begin{proof}
In order to construct the sequence $(U^\epsilon)_\epsilon$, we need to introduce mollifiers in time. 

If $u\in L^2((-\infty,T)$, then we define, for all $\epsilon>0$ and $0\leq t\leq T$,
$$
J_\epsilon u(t)=\epsilon^{-1}\int_{-\infty}^t e^{(s-t)/\epsilon}u(s){\rm d}s=\epsilon^{-1}\int_0^\infty e^{-s/\epsilon} u(t-s){\rm d}s.
$$
The operators $J_\epsilon$ and $\epsilon\dt J_\epsilon$ are uniformly bounded on $L^2((-\infty,T)$ and $\vert J_\epsilon u -u \vert_{L^2([0,T])}\to 0$ as $\epsilon \to 0$. 
We also notice that if $u\in H^1(-\infty,T)$, then $\dt J_\epsilon u=J_\epsilon \dt u$.

If $U^{\rm in}=0$, and $U$ an associated weak solution (as provided by Lemma \ref{lemmaweaksol}) then we can extend $U$ by zero for negative times; this extension, still denoted by $U$, is a weak solution to \eqref{CP} on $(-\infty,T)$. 

Denoting $U_\epsilon=J_\epsilon U$ and $F_\epsilon=J_\epsilon F$, we get after applying the mollifier $J_\epsilon$ to \eqref{CP} that $\dt U_\epsilon +{\bf A}U_\epsilon = F_\epsilon$ in the weak sense. Since $\dt U_\epsilon= \dt J_\epsilon U$, we get from the aforementioned properties of $J_\epsilon$ that $\dt U_\epsilon$ and $F_\epsilon$ are in  $L^2((-\infty,T);{\mathbb X})$. 
From the equations, this implies  $G_0\psi^\epsilon \in L^2((-\infty,T);L^2(\Gamma^{\rm D}))$ and $\zeta^\epsilon \in L^2((-\infty,T); \dot{H}^{1/2}(\Gamma^{\rm D}))$. Together with the fact that $U^\epsilon \in L^2((-\infty,T);{\mathbb X})$, this implies that $U^\epsilon \in L^2((-\infty,T);{\mathbb X}^1)$.

We have thus constructed a sequence $U^\epsilon \in H^1([0,T];{\mathbb X}) \cap L^2([0,T];{\mathbb X}_1)$ that converges to $U$ in $ L^2([0,T];{\mathbb X})$ and such that $F^\epsilon=(\dt +{\bf A})U^\epsilon$ converges to $F$ in $ L^2([0,T];{\mathbb X})$  as $\epsilon\to 0$.
Since we have enough regularity to apply Lemma \ref{lemmaweaksol}, we can apply the energy estimate to $U_\epsilon -U_{\epsilon'}$, and get that the sequence $(U_\epsilon)_{\epsilon}$ is a Cauchy sequence in $C([0,T];{\mathbb X})$ and therefore converges in this space as $\epsilon \to 0$. Its limit $U$ therefore belongs to $C([0,T];{\mathbb X})$ and satisfies the energy estimate. Uniqueness follows directly, up to a time depending constant for $\psi$; as above, uniqueness follows from the fact that the equation and the initial data must be satisfied almost everywhere.
\end{proof}

In order to prove the first point of the theorem, we need to consider the case $U^{\rm in}\neq 0$. By density of ${\mathbb X}^1$ in ${\mathbb X}$, it is enough to consider $U^{\rm in}\in {\mathbb X}^1$. 

In this case, $\tilde U:=U-U_0$ solves 
\[
(\dt +{\bf A})\tilde{U}=\tilde{F},
\]
with $\tilde F=F-{\bf A}U^{\rm in} \in L^2([0,T];{\mathbb X})$. 

Since $\tilde{U}(0)=0$, we can apply Lemma \ref{lemmaunique} and denote by $\tilde{U}^\epsilon$ the approximated sequence furnished by the lemma. The sequence $(U^\epsilon)_\epsilon$ defined by $U^\epsilon=U^{\rm in}+\tilde{U}^\epsilon$ is therefore an approximating sequence of $U$ that has enough regularity to apply the energy estimate of Lemma \ref{lemmaNRJest}. As in the proof of Lemma \ref{lemmaunique}, we can deduce that the weak solution to \eqref{CP}-\eqref{CP0} is unique, belongs to $C([0,T];{\mathbb X})$ and satisfies the energy energy estimate.

In order to prove the second point, we now consider $U^{\rm in}\in {\mathbb X}^1$, so that $U^{\rm in}_1:= -{\bf A}U^{\rm in}+F(0)\in {\mathbb X}$. By the first point of the theorem, there exists therefore a unique solution $U_1\in C([0,T];{\mathbb X})$ to \eqref{CP} with initial data $U^{\rm in}_1$  and source term $\dt F$.

We then define $\widetilde U \in C^1(\RR^+;  {\mathbb X})$ by
$$
\widetilde U=U^{\rm in}+\int_0^t U_1;
$$
in particular, we have $\dt \widetilde{U}=U_1$ and therefore
$$
\dt^2 \widetilde{U}+{\bf A}\dt \widetilde{U}=\dt F
$$
and
$$
(\dt \widetilde{U})_{\vert_{t=0}}=U^{\rm in}_1, \qquad \widetilde{U}_{\vert_{t=0}}=U^{\rm in}.
$$

Integrating in time, we therefore get
$$
\dt \widetilde{U}+{\bf A} \widetilde{U}= F,
$$
so that $\widetilde{U}$ and $U$ solve \eqref{CP} with the same initial data $U^{\rm in}$. By uniqueness, we deduce that $U=\widetilde{U}$, so that $U \in C^1([0,T];{\mathbb X})$. Using the equation, this implies that ${\bf A}U \in C([0,T];{\mathbb X})$, and therefore that $U\in  C([0,T];{\mathbb X}^1)$. This concludes the proof of the theorem.
\end{proof}

\subsection{Higher order time regularity}\label{sectHOTR}

We have seen in Theorems \ref{theoWPbounded} and \ref{theoWPunbounded}  that if $U^{\rm in} \in {\mathbb X}^1=\{U\in {\mathbb X}, (-{\bf A})U \in {\mathbb X}\}$ endowed with the graph norm and $F\in C^1({\mathbb R}^+;{\mathbb X})$, then the solution $U$ to \eqref{CPcomp} belongs to $C^1(\RR^+;{\mathbb X})\cap C(\RR^+;{\mathbb X}^1)$, with ${\mathbb X}$ defined by \eqref{defXbounded} when $\Gamma^{\rm D}$ is bounded and by \eqref{defXunbounded} otherwise.
In order to study time regularity or order $n\in {\mathbb N}$, we need to introduce ${\mathbb X}^n$ as
$$
{\mathbb X}^n=\{U\in {\mathbb X}, \, \forall 0\leq j\leq n, \, (-{\bf A})^jU \in {\mathbb X}\},
$$
endowed with the semi-norm $\vert U \vert_{{\mathbb X}^n}:= \sum_{j=0}^n \vert (-{\bf A})^jU \vert_{\mathbb X}$; 
for time depending function, we also define the space ${\mathbb V}^n_T$ for $T>0$ by
\begin{equation}\label{defWT}
{\mathbb V}^n_T:=\left\lbrace U\in C^n([0,T];{\mathbb X}), \, \forall 0\leq j\leq n, \, (-{\bf A})^j U \in C^{n-j}([0,T];{\mathbb X}\right\rbrace;
\end{equation}
for all $t\in [0,T]$ and $U\in {\mathbb V}^n_T$, we also use the notation
$$
	\vertiii{U(t)}_n:= 
	\sum_{j=0}^n \vert \dt^{j} U (t)\vert_{{\mathbb X}^{n-j}}.
$$  
We show here that if $F\in {\mathbb V}^n_T$ and $U^{\rm in}$ are such that $\vertiii{U(0)}_n$ as defined above is finite, then the solution $U$ provided by Theorem \ref{theoWPbounded} or \ref{theoWPunbounded} (depending on whether $\Gamma^{\rm D}$ is bounded or not) belongs to ${\mathbb V}^j_T$.
 \begin{corollary}\label{corohigherT}
Let $\Omega$, $\Gamma^{\rm D}$ and $\Gamma^{\rm N}$ be as in Assumption \ref{assconfig}. 
Let $T>0$ and $n\in {\mathbb N}^*$. If $U^{\rm in}\in {\mathbb X}^n$ and $F\in {\mathbb V}^n_T$ then the solution $U$ to \eqref{CP}-\eqref{CP0} furnished by Theorem \ref{theoWPbounded} if $\Gamma^{\rm D}$ is bounded or Theorem \ref{theoWPunbounded} if $\Gamma^{\rm D}$ is unbounded belongs to ${\mathbb V}^n_T$ and for all $t\in [0,T]$,
$$
	\vertiii{U(t)}_n \leq \vertiii{U(0)}_n+ \int_0^t \vertiii{F(t')}_{n}{\rm d}t'.
$$
\end{corollary}
\begin{remark}\label{exprCI}
We can of course use the equation to write
$$
(\dt^j U)_{\vert_{t=0}}=(-{\bf A})^jU^{\rm in}+\sum_{k=0}^{j-1}(-{\bf A})^k (\dt^{j-1-k}F)_{\vert_{t=0}},
$$
so that $\vertiii{U(0)}_n$ can be bounded from above in terms of $U^{\rm in}$ and $F$, namely,
$$
\vertiii{U(0)}_n\lesssim  \vert U^{\rm in}\vert_{{\mathbb X}^n}+\vertiii{F(0)}_{n-1}.
$$
\end{remark}
\begin{proof}
The proof is the same in the bounded and unbounded case.  One just has to apply the method used in the proof of the second point of Theorem \ref{theoWPunbounded} to get that $U\in C^n([0,T];{\mathbb X})$ and that for
all $0\leq j+k\leq n$, $(-{\bf A})^j \dt^{k}U$ is a strong solution to the initial boundary value problem
$$
\dt [(-{\bf A})^j \dt^{k}U] +{\bf A} [(-{\bf A})^j \dt^{k}U] = (-{\bf A})^j\dt^{k} F \quad \mbox{ on }\RR^+\times \Gamma^{\rm D},
$$ 
with initial condition $[(-{\bf A})^j \dt^{k}U]_{\vert_{t=0}}$, so that we can apply the energy estimate of Theorem \ref{theoWPunbounded} and obtain
$$
\vert(-{\bf A})^j \dt^{k} U (t)\vert_{\mathbb X}\leq \vert(-{\bf A})^j \dt^{k} U (0) \vert_{{\mathbb X}}+\int_0^t \vert (-{\bf A})^j \dt^{k}  F(t')\vert_{\mathbb X}{\rm d}t',
$$
from which the result follows easily.
\end{proof}

\subsection{Consequences on the space regularity}\label{sectlimreg}  
When $n= 2$, we show that it is possible to deduce space regularity from the time regularity of the previous section. For higher values of $n$, this is still possible, but one has to impose smallness conditions on the corners of the fluid domain. Let us start with the case $n=2$. 

\begin{corollary}\label{space regularity} 
Let $\Omega$, $\Gamma^{\rm D}$ and $\Gamma^{\rm N}$ be as in Assumption \ref{assconfig}.  
Let $T>0$ and $U^{\rm in}\in {\mathbb X}^2$, $F\in {\mathbb V}^2_T$. Let  also  $U \in {\mathbb V}_T^2$ be the solution to \eqref{CP}-\eqref{CP0} furnished by Corollary \ref{corohigherT}.\\
- If $\Gamma^{\rm D}$ is bounded  then $U$ also belongs to $C([0,T];H^1(\Gamma^{\rm D})\times \dot{H}^{3/2}(\Gamma^{\rm D}))$ and there exists  some constant $C>0$ independent of $U$ such that for all $0\leq t\leq T$,
\[
\vert U(t)\vert_{H^1(\Gamma^{\rm D})\times \dot{H}^{3/2}(\Gamma^{\rm D})}\le C
\Big(\vert U^{\rm in} \vert_{{\mathbb X}^2}+ \vertiii{F(0)}_{1}+\int_0^t \vertiii{F(t')}_{2}{\rm d}t'\Big).
\]
-  If $\Gamma^{\rm D}$ is unbounded, and if moreover $\psi^{\rm in}\in L^2(\Gamma^{\rm D})$ and $g\in L^1([0,T];L^2(\Gamma^{\rm D}))$ then
$U$ also belongs to $C([0,T];H^1(\Gamma^{\rm D})\times {H}^{3/2}(\Gamma^{\rm D}))$ and there exists  some constant $C>0$ independent of $U$ such that for all $0\leq t\leq T$,
\begin{align*}
\vert U(t)\vert_{H^1(\Gamma^{\rm D})\times {H}^{3/2}(\Gamma^{\rm D})}\le C
\Big(&\vert \psi^{\rm in}\vert_{L^2(\Gamma^{\rm D})}+\vert U^{\rm in} \vert_{{\mathbb X}^2}+ \vertiii{F(0)}_{1}\\
&+\int_0^t ( \vert g(t')\vert_{L^2(\Gamma^{\rm D})}+\vertiii{F(t')}_{2}){\rm d}t'\Big).
\end{align*}
\end{corollary}
\begin{remark}
In the unbounded case, if the initial data $\psi^{\rm in}$ is not in $L^2(\Gamma^{\rm D})$, then one can write $U=U^\sharp+U^\flat$ with  $U^\flat:=(0,(1+G_0)^{-1}\psi^{\rm in})^{\rm T}$ and
$U^\sharp$ solving
$$
\dt U^\sharp +{\bf A}U^\sharp=F^\sharp
$$
with
$$
F^\sharp:=F+\begin{pmatrix} G_0(1+G_0)^{-1}\psi^{\rm in} \\ 0 \end{pmatrix}
$$
and initial condition
$$
U^\sharp_{\vert_{t=0}}=(\zeta^{\rm in}, G_0(1+G_0)^{-1}\psi^{\rm in})^{\rm T}.
$$
Since $G_0(1+G_0)^{-1}\psi^{\rm in}  \in L^2(\Gamma^{\rm D})$, $(1+G_0)^{-1}\psi^{\rm in} \in \dot{H}^{1/2}(\Gamma^{\rm D})$ and $\vert G_0(1+G_0)^{-1}\psi^{\rm in} \vert_{L^2}+\vert (1+G_0)^{-1}\psi^{\rm in} \vert_{\dot{H}^{1/2}}\lesssim \vert \psi^{\rm in}\vert_{\dot{H}^{1/2}}$, we can apply the second point of the corollary to $U^\sharp$. The solution $U$ can therefore be represented as the sum of a time independent function $U^\flat \in {\mathbb X}^2$ and a perturbation $U^\sharp\in C([0,T]; H^1(\Gamma^{\rm D})\times H^{3/2}(\Gamma^{\rm D}))$.
\end{remark} 
\begin{remark}\label{rem20}
As commented in \S \ref{ssectparadox} of the introduction, we are here at the critical regularity threshold  for $\psi$, namely $\dx\psi\in H^{1/2}(\Gamma^{\rm D})$. It is therefore worth considering whether the assumption $U^{\rm in}\in {\mathbb X}^2$ hides some compatibility condition on $\psi$. This assumption implies in particular that $G_0\psi^{\rm in}\in H^{1/2}(\Gamma^{\rm D})$, which in turns implies, as we saw in Proposition \ref{propDNHO} , that $(\psi^{\rm in})^\mfh \in H^2(\Omega)$, and therefore that $(\dx (\psi^{\rm in})^\mfh)_{\vert_{\Gamma}}\in H^{1/2}(\Gamma)$. Since $(\dx (\psi^{\rm in})^\mfh )_{\vert_{\Gamma^{\rm N}}}$ vanishes in the vicinity of the contact point, this implies that $\dx\psi=(\dx (\psi^{\rm in})^\mfh)_{\vert_{\Gamma^{\rm D}}}$ belongs $\rho^{1/2}L^2$ (\cite{Grisvard}, Theorem 1.5.2.3), which can be viewed as weak form of the compatibility condition $\dx\psi=0$ that arises, as we saw in the introduction, when $((\psi^{\rm in})^\mfh)\in H^{2+\epsilon}(\Omega)$, with $\epsilon >0$. Since Corollary \ref{corohigherT} implies that $G_0\psi \in C([0,T]; H^{1/2}(\Gamma^{\rm D})$, this compatibility condition is propagated by the flow.
\end{remark} 
\begin{proof}
 Applying Corollary \ref{corohigherT} with $n=2$ shows directly that the solution $U(t)\in \mathbb V^2_T$ satisfies 
\[
\sum_{j=0}^2\vert (-{\bf A})^jU(t)\vert_{\mathbb X}\le C
\Big( \vert U^{\rm in} \vert_{{\mathbb X}^2}+ \vertiii{F(0)}_{1}+\int_0^t \vertiii{F(t')}_{2}{\rm d}t'\Big),
\]
with some constant $C$.  
Consequently,  using this estimate, and omitting the dependence on time for the sake of clarity, we know that 
\[
\zeta\in H^{1/2}(\Gamma^{\rm D}),\ \psi\in \dot H^1(\Gamma^{\rm D}),\ G_0\zeta\in L^2(\Gamma^{\rm D}), \ G_0\psi\in H^{1/2}(\Gamma^{\rm D}),
\]
and that they can all be controlled by the right side above.

When $\Gamma^{\rm D}$ is bounded, thanks to Proposition \ref{propDNHO}, we know immediately that $\psi\in  \dot{H}^{3/2}(\Gamma^{\rm D})$ with the desired estimate. Moreover, we also know from Proposition \ref{propDNell} that $\zeta\in H^1(\Gamma^{\rm D})$ with the desired estimate.

When $\Gamma^{\rm D}$ is unbounded, the additional assumptions imply that $\psi=\psi^{\rm in}+\int_0^t (g-{\mathtt g}\zeta)$ is in $C([0,T];L^2)$ and that for some $C>0$, 
$$
\forall 0\leq t\leq T, \qquad \vert \psi (t)\vert_{L^2(\Gamma^{\rm D})}\leq C\big( \vert \psi^{\rm in}\vert_{L^2}+ \int_0^t \vert \zeta(t') \vert_{L^2)}+\vert g(t') \vert_{L^2)}){\rm d}t';
$$
we can therefore use the second point of Proposition \ref{propDNHO} to complete the proof of the proposition.
\end{proof}

When $n\geq 3$, it is still possible to infer space regularity from the time regularity, but this requires a smallness assumption on the angles of the corners of the fluid domain.
\begin{corollary}\label{space regularity2} 
Let $n\in {\mathbb N}^*$ and $\Omega$, $\Gamma^{\rm D}$ and $\Gamma^{\rm N}$ be as in Assumption \ref{assconfig},  and assume moreover that the angles at each corner point between $\Gamma^{\rm N}$ and $\Gamma^{\rm D}$ are strictly less than $2\pi/n$.
Let $T>0$ and $U^{\rm in}\in {\mathbb X}^n$, $F\in {\mathbb V}^n_T$. Let also $U \in {\mathbb V}_T^n$ be the solution to \eqref{CP}-\eqref{CP0} furnished by Corollary \ref{corohigherT}.\\
- If $\Gamma^{\rm D}$ is bounded then $U$ also belongs to $C([0,T];H^{n/2}(\Gamma^{\rm D})\times \dot{H}^{(n+1)/2}(\Gamma^{\rm D}))$ and there exists  some constant $C>0$ independent of $U$ such that for all $0\leq t\leq T$,
\[
\vert U(t)\vert_{H^{n/2}(\Gamma^{\rm D})\times \dot{H}^{(n+1)/2}(\Gamma^{\rm D})}\le C
\Big(\vert U^{\rm in} \vert_{{\mathbb X}^n}+ \vertiii{F(0)}_{n-1}+\int_0^t \vertiii{F(t')}_{n}{\rm d}t'\Big).
\] 
-  If $\Gamma^{\rm D}$ is unbounded, and if moreover $\psi^{\rm in}\in L^2(\Gamma^{\rm D})$ and $g\in L^1([0,T];L^2(\Gamma^{\rm D}))$ then $U$ also belongs to $C([0,T];H^{n/2}(\Gamma^{\rm D})\times H^{(n+1)/2}(\Gamma^{\rm D}))$ and there exists  some constant $C>0$ independent of $U$ such that for all $0\leq t\leq T$,
\begin{align*}
\vert U(t)\vert_{H^{n/2}(\Gamma^{\rm D})\times H^{(n+1)/2}(\Gamma^{\rm D})} \le C
\Big(&\vert \psi^{\rm in}\vert_{L^2(\Gamma^{\rm D})}+\vert U^{\rm in} \vert_{{\mathbb X}^n}+ \vertiii{F(0)}_{n-1}\\
&+\int_0^t ( \vert g(t')\vert_{L^2(\Gamma^{\rm D})}+\vertiii{F(t')}_{n}){\rm d}t'\Big).
\end{align*}
\end{corollary}
\begin{remark}
For $n=3$, the assumption that the angle is smaller than $120$ degrees still allows a perpendicular contact. With the better regularity enjoyed by the solution, the compatibility condition at the contact point now holds in the strong sense for $\psi$, namely, $\dx\psi=0$; moreover, we also have now a compatibility condition on $\zeta$, but only in weak form, namely, $\dx\zeta \in \rho^{1/2}L^2$ near the contact point.
 With the same arguments as in Remark \ref{rem20}, if $U^{\rm in}\in {\mathbb X}^3$, then one necessarily has $\partial_x\psi^{\rm in}=0$ and $\partial_x\zeta^{\rm in}\in  \rho^{1/2}L^2(\Gamma^{\rm D})$ and these conditions are propagated by the equations. The contact angle remains therefore equal to $\pi/2$ consistently with the arguments of \cite{ABZ3}. On the contrary,  initial data, even if they are very smooth, cannot be in ${\mathbb X}^n$ with $n>2$ if the initial contact angle is not $\pi/2$, that is, if $\partial_x\zeta^{\rm in}\neq 0$ at the contact point. In this situation, further regularity far from the corner can be obtained with weighted estimates as in Sections 5 and 6. Let us mention however that, as shown by the example of the Stokes wave given in the introduction, the most relevant possible obstruction to the argument of \cite{ABZ3}  comes from the nonlinear terms not considered in this article.
\end{remark}
\begin{proof}
Proceeding as in the proof of Corollary \ref{space regularity}, the result follows easily if we can establish that the solution $u$ to the elliptic boundary value problem
$$
\begin{cases}
\Delta u=f \quad\mbox{ in }\Omega,\\
\partial_{\rm n} u =0 \quad \mbox{ on }\Gamma^{\rm N},\\
\partial_{\rm n} u=g  \quad \mbox{ on }\Gamma^{\rm D},
\end{cases}
$$
belongs to $H^{j/2+1}(\Omega)$ for all $0\leq j\leq n$ if $g$ belongs to $H^{(n-1)/2}(\Gamma^{\rm D})$. If $\Omega$ is a bounded domain, this is established as in the proof of Lemma \ref{lemmaBVP} since the fact that the mapping \eqref{mappingelliptic} is an isomorphism of Banach spaces if the angles at 
 the corners of $\Omega$ are smaller than $\pi$, is a special case (corresponding to $n=2$) of a more general theorem stating that 
 $$
 \begin{array}{lcl}
\dot{H}^{n/2+1}(\Omega)\backslash \RR & \to &  H^{n/2-1}(\Omega)\times  \prod_{j=1}^{N+1} {H}^{(n-1)/2}(\cE_j) \\
u & \mapsto &(\Delta u,(\partial_{\rm n} u)_{\vert_{\cE_1\times \{0\}}}, \dots, (\partial_{\rm n} u)_{\vert_{ \cE_{N+1}\times \{0\}}})
\end{array},
$$
provided that the angles are all strictly smaller than $2\pi/n$. The generalization to the case where $\Omega$ is unbounded is then obtained as in the proof of Proposition \ref{propDNHO}.
\end{proof}

\section{Commutator estimates with weighted derivatives}  \label{sectcommutator}
We have so far constructed, in Corollary \ref{corohigherT}, solutions that are regular in time, more precisely, that belong to ${\mathbb V}^n_T$, with ${\mathbb V}^n_T$ defined in \eqref{defWT}. We have also shown in Corollary \ref{space regularity} that for $n\leq 2$, this time regularity could be used to obtain space regularity, namely that the solution $U=(\zeta,\psi)^{\rm T}$ belongs to $C([0,T];H^{n/2}(\Gamma^{\rm D})\times \dot{H}^{((n+1)/2}(\Gamma^{\rm D}))$. We then have proved in Corollary \ref{space regularity2} that in order to extend this result to $n\geq 3$, one has to make smallness assumptions on the angles at corners of the fluid domain due to the limited elliptic regularity in corner domains.

In order to obtain space regularity without making any smallness assumption on the angles, we can use weighted estimates. Such estimates are obtained by deriving energy estimates for weighted derivatives of the solution, the weighted derivatives being of the form $\rho\dx$, where $\rho$ is a bounded function that behaves, near each boundary point of $\Gamma^{\rm D}$, as the distance function to this boundary point. 

This allows to get high order regularity estimates of the solution in the interior of $\Gamma^{\rm D}$, even if the regularity near the boundaries is no larger than $H^1(\Gamma^{\rm D})\times \dot{H}^{3/2}(\Gamma^{\rm D})$. 
Applying $(\rho\dx)^j$ to \eqref{CPcomp}, one obtains formally
$$
\dt (\rho\dx)^j U +{\bf A} (\rho\dx)^j U = (\rho\dx)^j F + \begin{pmatrix} [(\rho\dx)^j, G_0]\psi \\ 0 \end{pmatrix},
$$
so that one can expect to recover an estimate on $\vert (\rho\dx)^j U\vert_{\mathbb X}$ from Theorem \ref{theoWPbounded} or \ref{theoWPunbounded} provided that the commutator estimate $[(\rho\dx)^j, G_0]\psi$ can be controlled. This section is devoted to derive estimates for this commutator.

We introduce several notations in \S \ref{sectnotref} and then construct in \S \ref{sectweight} the weight function $\rho$ which is defined on $\overline{\Omega}$. In \S \ref{sectTN}, we construct an extension $T$ of the tangent vector ${\bf t}$ defined on $\Gamma^*$; this extension is used to construct an extension $X=\rho T\cdot \nabla$ defined on $\Omega\cup \Gamma^*$ of the weighted derivative $\rho\dx$ used to measure space regularity on $\Gamma^{\rm D}$. 
This vector field $X$ is then used in \S \ref{sectHOHE} to measure regularity of the harmonic extension $\psi^\mfh$; more precisely, we show that $X^n\psi^\mfh$ is well defined in $\dot{H}^1(\Omega)$ if $(\rho\dx)^j\psi \in \dot{H}^{1/2}(\Gamma^{\rm D})$ for $0\leq j\leq n$; moreover, we show that $X^n \psi^\mfh$ is equal to $((\rho\dx)\psi)^\mfh$ up to lower order terms in $\dot{H}^1(\Omega)$. Estimates for the commutators $[(\rho\dx)^j, G_0]\psi $ are then derived in \S \ref{sectCEDN}.

\medbreak

To simplify the computations, we assume from now on that the immersed objects and the bottom are not curved near the contact points with the surface of the water.
\begin{assumption}\label{assconfig2}
Assumption \ref{assconfig} is satisfied, and moreover, $\Gamma^{\rm N}$ is flat in the vicinity of each corner.
\end{assumption}

\subsection{Notations}\label{sectnotref}

We use the notations of Assumption \ref{assconfig}. On each point of $\Gamma^*$, we denote by ${\bf t}$ the \emph{rightwards} unit tangent vectors (on $\Gamma_j^{\rm w}$, this means that it is oriented from $C_j^{(\rm l)}$ to $C_j^{(\rm r)}$), while ${\bf n}={\bf t}^\perp$ stands for the normal vector.

Let us denote by ${\mathcal C}$ the set of all finite corners of $\Gamma$, which is the union of the set of left finite corners ${\mathcal C}^{\rm l}$ and right finite corners  ${\mathcal C}^{\rm r}$,
\begin{equation}\label{deffinC}
{\mathcal C}={\mathcal C}^{\rm l}\cup {\mathcal C}^{\rm r}
\quad\mbox{ with }\quad
\begin{cases}
{\mathcal C}^{\rm l}=\{C_j^{(\rm l)}\quad (1\leq j\leq N), \quad C_{N+1} \mbox{ if   it is finite}\} ,\\
{\mathcal C}^{\rm r}=\{C_j^{(\rm r)}\quad (1\leq j\leq N), \quad C_{0} \mbox{ if it is finite}\}.
\end{cases}
\end{equation}
For ${\rm c}\in{\mathcal C}$, we denote by $(r^\cor (x,z),\theta^\cor (x,z))$ the polar coordinates of the point $M(x,z)$ in the polar frame centered at the corner $c$ and with $\theta^\cor =0$ corresponding to the horizontal half-line starting from ${\rm c}$ and unbounded in the right-direction. For all $M\in \Omega$, we also denote by ${\bf e}_{\rm r}^\cor (M)$ the unit radial vector ${\bf e}_{\rm r}^\cor (M)=\frac{1}{\vert{{\rm c} M}\vert}\overrightarrow{{\rm c} M}$.

According to Assumption \ref{assconfig2},  in the neighborhood of each corner, the fluid domain is a portion of half-cone with interior angle $\omega^\cor$; more precisely, there exist positive real numbers $R^\cor >0$ such that $\Omega^\cor :=\Omega\cap B({\rm c} ,R^\cor )$ are given by
\begin{align*}
\Omega^{(\rm c)}&=\{(X,z), 0<r^{(\rm c)}(x,z)<R^{(\rm c)}, \quad -\omega^{(\rm c)}<\theta^{(\rm c)}(x,z)<0\} & \mbox{ if } {\rm c}\in {\mathcal C}^{\rm r},\\
\Omega^{(\rm c)}&=\{(X,z), 0<r^{(\rm c)}(x,z)<R^{(\rm c)}, \quad -\pi<\theta^{(\rm c)}(x,z)<-\omega^{(\rm c)}\}  & \mbox{ if } {\rm c}\in {\mathcal C}^{\rm l},
\end{align*}
 and where  $0<\omega^\cor <2\pi$. The numbers $R^\cor>0$ are chosen small enough, in order for the sets $\Omega^\cor$ to be disjoint.
\begin{figure}[htbp]
                \centering
                \includegraphics[width=0.6\textwidth]{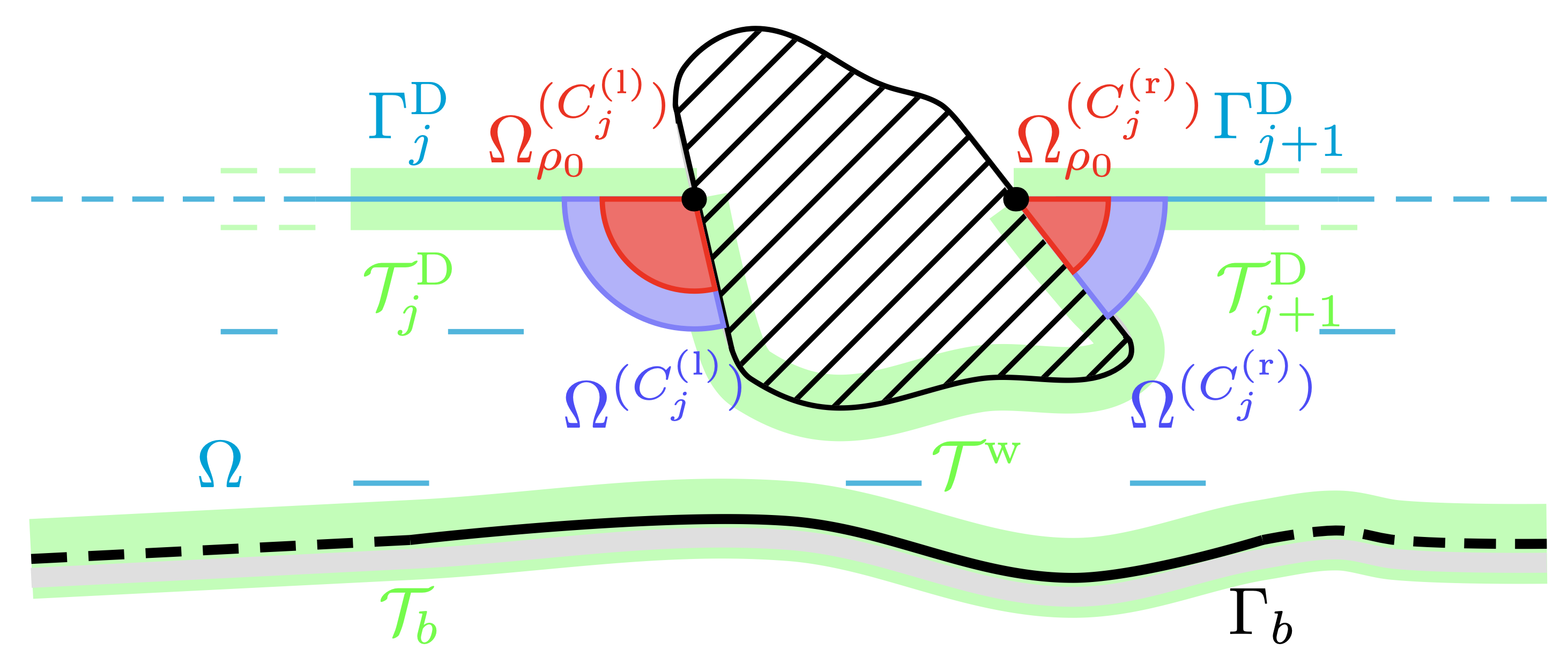}
                \caption{The sectors and the tubular neighborhoods}
        \label{fig:sectors}
\end{figure}

Let us now denote by $\chi: \RR_+\to \RR_+$  a smooth positive function, compactly supported in $[0,1)$, equal to $1$ on $[0,1/2)$ and with values in $[0,1]$. We then introduce
\begin{equation}\label{defchic}
	\chi^\cor (x,z)=\chi(\frac{r^\cor (x,z)}{R^\cor }),
\end{equation}
 with $R^\cor$ as in the previous section; taking the $R^\cor $ smaller if necessary, we can suppose that
\begin{equation}\label{asssuppchi}
 \chi^\cor \equiv 0 \quad  \mbox{ on }\quad  \mbox{\rm Supp }\chi^{(\rm c')} \quad  \mbox{ if }\quad {\rm c}\neq {\rm c'}\qquad ({\rm c}, {\rm c'}\in {\mathcal C}).
 \end{equation}
 
 For all $0<\rho_0<R^\cor$, we also denote by $\Omega^\cor_{\rho_0}\subset \Omega^\cor$ the domain
 $$
 \Omega^{(\rm c)}_{\rho_0}:=\Omega^{(\rm c)}\cap B({\rm c},\rho_0) \quad\mbox{ if }\quad {\rm c}\in {\mathcal C};
 $$ 
 if $C_0$ or $C_{N+1}$ is infinite, we then define
 \begin{align*}
 \Omega_{\rho_0}^{(C_0)}&=\Omega \cap\{(x,z)\in {\mathbb R}^2, { x}<-\frac{1}{\rho_0} \} &\mbox{ if } C_0 \mbox{ is infinite},\\
 \Omega_{\rho_0}^{(C_{N+1})}&=\Omega \cap\{(x,z)\in {\mathbb R}^2, { x}>\frac{1}{\rho_{0}} \} &\mbox{ if } C_{N+1} \mbox{ is infinite};
 \end{align*}
we choose $\rho_0>0$ small enough in order for the sets $\Omega_{\rho_0}^{(\rm c)}$ to be disjoint.

 Finally, we introduce a partition of unity $(\Theta^\cor)_{{\rm c}\in {\mathcal C}_\infty}$, where ${\mathcal C}_{\infty}$ consists of the union of ${\mathcal C}$ and $C_0$ and $C_{N+1}$ if they are infinite. More precisely, the $\Theta^\cor$ are smooth functions $\Theta^\cor$ such that $0\leq \Theta^\cor\leq 1$ and $\sum_{{\rm c}\in {\mathcal C}_\infty}\Theta^\cor=1$ on $\Omega$, and we also take it in such a way that
 \begin{equation}\label{defpartunit}
 \Theta^\cor\equiv 1 \quad\mbox{ on }\quad    \Omega^\cor 
 \quad\mbox{ and }\quad
  \Theta^\cor\equiv 0 \quad\mbox{ on }\quad \Omega^{({\rm c}')} \quad\mbox{ if }\quad {\rm c}'\neq {\rm c}
 \end{equation}
(with the convention that  $\Omega^\cor:= \Omega^\cor_{\rho_0} $ if ${\bf c}$ is infinite), 
 and 
  \begin{equation}\label{defpartunit2}
 \forall {\rm c}\in {\mathcal C}_\infty,\qquad  \partial_{\rm n} \Theta^\cor=0 \quad\mbox{ on }\quad \Gamma^*.
  \end{equation}

  We can also choose these functions in such a way that if $\Gamma^\side$ is a connected component of $\Gamma^*$ and if ${\rm c}$ is not a finite or infinite endpoint of $\Gamma^\side$, then $\Theta^{\rm c}\equiv 0$ on $\Gamma^\side$.
\subsection{Construction of the weight function}\label{sectweight}

 The following lemma provides a weight function that behaves near each corner $C^\cor$ as the distance $r^\cor$ to this corner, that is bounded in $\Omega$ and that satisfies a homogenous Neumann boundary condition on $\Gamma^*$.
 \begin{lemma}\label{LMcomm2}
Under Assumption \ref{assconfig2} and with the notations of Section \ref{sectnotref}, there exists a function $\rho\in C^\infty(\Omega\cup \Gamma^*)\cap C(\overline{\Omega})$ and a constant $0<\rho_0\leq\frac{1}{2} \min_{{\rm c}\in{\mathcal C}} R^\cor$ such that:
\begin{enumerate}
\item One has $0<\rho\leq 2\rho_0$ on $\Omega$.
\item For all ${\rm c} \in {\mathcal C}$, one has   $\rho=r^\cor $ in ${\Omega_{\rho_0}^\cor}$, while $\rho=\rho_0$ in ${\Omega} \backslash \big(\bigcup_{c\in{\mathcal C}} \Omega^\cor\big)$.
\item One has $(\partial_{\rm n} \rho)_{\vert_{\Gamma^*}}=0$.
\end{enumerate}
\end{lemma} 
\begin{proof}
If $\rho_0>0$ satisfies the smallness properties of Section \ref{sectnotref}, and if in addition $\rho_0=\frac{1}{2}\min_{{\rm c}\in {\mathcal C}} R^\cor $, we define
$$
\rho=\sum_{{\rm c}\in {\mathcal C}} \Theta^\cor \big[  r^\cor \chi\big(\frac{r^\cor }{2\rho_0}\big)   +\rho_0 \big(1-\chi\big(\frac{r^\cor }{2\rho_0}\big) \big)  \big],
$$
with the same function $\chi$ as the one used for \eqref{asssuppchi} and with $\Theta^\cor$ as in \eqref{defpartunit} and  \eqref{defpartunit2}. This choice of $\rho_0$ and this function satisfy all the requirements of the lemma.
\end{proof}

\subsection{Construction of the extensions of the tangent and normal vector fields}\label{sectTN}

The following lemma provides a vector field $T$ whose restriction to $\Gamma^*$ coincides with the unit tangent vector ${\bf t}$, that coincides with the local unit radial vector near each corner, and such that $T\cdot \nabla$ preserves homogeneous boundary conditions on $\Gamma^{\rm N}$. Note that the reason we introduce the coefficient $\sigma({\rm l})=-1$ in the second point of the lemma is because ${\bf e}_{\rm r}^{(\rm l)}$ is pointing leftwards in $\Omega^{(\rm l)}$.
\begin{lemma}\label{LMcomm1}
Under Assumption \ref{assconfig2} and with $\rho_0$ as in Lemma \ref{LMcomm2}, there exist two vector fields $T,N \in C^\infty(\Omega\cup \Gamma^* ; \RR^2)$ that satisfy the following properties:
\begin{enumerate}
\item $T={\bf t}$ and $N={\bf n} $ on  $\Gamma^*$.
\item $T=\sigma({\rm c}){\bf e}_{\rm r}^\cor $ in $\Omega^\cor_{\rho_0}$ with $\sigma({\rm c})=-1$ if ${\rm c}\in {\mathcal C}^{\rm l}$ and $\sigma({\rm c})=1$ if ${\rm c}\in {\mathcal C}^{\rm r}$.
\item For all $\phi \in C^\infty(\Omega\cup \Gamma^{\rm N})$ such that $(\partial_{\rm n} \phi)_{\vert_{\Gamma^{\rm N} }}=0$, one has $(\partial_{\rm n}   (T\cdot \nabla \phi) )_{\vert_{\Gamma^{\rm N} }}=0$.
\item One has $T\in L^\infty(\Omega)$ and, for all $n\in {\mathbb N}$, $\rho^n \partial^n T\in L^\infty(\Omega)$, where $\partial^n$ stands for any spatial derivative of order $n$.
\end{enumerate}
\end{lemma}
\begin{proof}
For each connected component $\Gamma^\side$ of $\Gamma^*$, we denote by  $({\bf n}^\side ,{\bf t}^\side )$ the restriction to $\Gamma^\side$ of the normal and tangential vectors ${\bf n}$ and ${\bf t}$ defined on $\Gamma^*$.
 There exists a tubular neighborhood ${\mathcal T}^\side $ of $\Gamma^\side $ and a smooth extension $(N^\side ,T^\side )$ of $({\bf n}^\side ,{\bf t}^\side )$ to $ \overline{\Omega}$ such that 
\begin{equation}\label{eqdercommut}
N^{(\rm s)}\cdot T^{(\rm s)}=0,\qquad \vert N^\side \vert=1 \quad\mbox{ and } \quad [N^\side \cdot \nabla,T^\side \cdot \nabla]=0 \quad\mbox{ on }\quad {\mathcal T}^\side \cap \Omega,
\end{equation}
with the neighborhoods ${\mathcal T}^\side$ chosen small enough to have
\begin{equation}\label{condsupp}
\forall {\rm s}'\neq {\rm s},\quad N^\side \equiv 0, \quad T^\side \equiv 0 \quad\mbox{ on }\quad \big({\mathcal T}^{({\rm s}')}\cap \Omega\big) \backslash \bigcup_{{\rm c}\in {\mathcal C}} \Omega^\cor_{\rho_0}.
\end{equation}

This can be achieved by taking a compactly supported extension  of the  normal and tangent vectors classically used to define normal-tangential coordinates near the boundary. In particular, we can further assume that $T^\side$ and $N^\side$ are constant in ${\mathcal T}^\side \cap \Omega^\cor$ if ${\rm c}$ is a finite endpoint of $\Gamma^\side$.

We then define
$$
T= \sum_{{\rm c}\in{\mathcal C} }\Theta^\cor \big(\sigma\cor\chi^\cor{\bf e}_{\rm r}^\cor + (1-\chi^\cor )\sum_{{\rm s}}T^\side\big) + \big(\sum_{{\rm c}\in {\mathcal C}_\infty\backslash {\mathcal C}} \Theta^\cor\big) \big(\sum_{{\rm s}}T^\side\big) ,
$$
with $\sigma({\rm c})$ as in the statement of the lemma, $\Theta^\cor$ as in \eqref{defpartunit} and \eqref{defpartunit2}, and with the summation on ${\rm s}$ being taken over all the connected components of $\Gamma^*$; for later use, we define similarly
$$
N= \sum_{{\rm c}\in{\mathcal C} } \Theta^\cor \big(\sigma\cor\chi^\cor({\bf e}_{\rm r}^\cor)^\perp + (1-\chi^\cor )\sum_{{\rm s} }N^\side\big)+  \big(\sum_{{\rm c}\in {\mathcal C}_\infty\backslash {\mathcal C}} \Theta^\cor\big) \big(\sum_{{\rm s}}N^\side\big).
$$

Let us prove the first assertion of the lemma. Let us consider a bounded connected component $\Gamma^\side$ of $\Gamma^*$  (the case where it is unbounded is treated similarly). Let ${\rm c}_{\rm s}, {\rm c}'_{\rm s}
\in {\mathcal C}$ be the two endpoints of $\Gamma^\side$. Since $\Theta^{({\rm c}_{\rm s})}=\chi^{({\rm c}_{\rm s})}=1$ on $\Gamma^\side\cap \Omega_{\rho_0}^{({\rm c}_{\rm s})}$, we obtain that $T=\sigma({\rm c}_{\rm s}){\bf e}_{\rm r}^{({\rm c}_{\rm s})}={\bf t}^\side$ on  $\Gamma^\side\cap \Omega_{\rho_0}^{({\rm c}_{\rm s})}$. The same result holds similarly on  $\Gamma^\side\cap \Omega_{\rho_0}^{({\rm c}'_{\rm s})}$. 

Finally, on $\Gamma^\side\backslash\big( \Omega_{\rho_0}^{({\rm c}_{\rm s})}\cup \Omega_{\rho_0}^{({\rm c}'_{\rm s})}\big)$ we obtain using \eqref{condsupp} that 
$$
T= \Theta^{({\rm c}_{\rm s})} \big[\sigma({\rm c}_{\rm s})\chi^{({\rm c}_{\rm s})}{\bf e}_{\rm r}^{({\rm c}_{\rm s})}+(1-\chi^{({\rm c}_{\rm s})})T^\side\big]+
\Theta^{({\rm c}_{\rm s}')} \big[ \sigma({\rm c}_{\rm s}')\chi^{({\rm c}'_{\rm s})}{\bf e}_{\rm r}^{({\rm c}_{\rm s}')}+(1-\chi^{({\rm c}'_{\rm s})})T^\side\big].
$$
Since $T^\side={\bf t}^\side$ on $\Gamma^\side$ and that $\sigma({\rm c}_{\rm s}){\bf e}_{\rm r}^{({\rm c}_{\rm s})}={\bf t}^\side$ on $\Gamma^\side \cap \mbox{Supp}(\chi^{({\rm c}_{\rm s})})$, and proceeding similarly at ${\rm s}'$, we deduce that
$$
T=(\Theta^{({\rm c}_{\rm s})}+\Theta^{({\rm c}_{\rm s}')}\big){\bf t}^\side { ={\bf t}^\side,}
$$
where we used the fact that $\Theta^\cor=0$ on $\Gamma^\side$ if ${\rm c}\not\in \{ {\rm c}_{\rm s},{\rm c}_{\rm s}'\}$.
This concludes the proof of the first assertion.

 Let us turn to the proof of the second one.
Since  $\chi^\cor \equiv 0$ on $\Omega^{({\rm c'})}$ if ${\rm c}'\neq {\rm c}$, and $\chi^\cor\equiv 1 $ in $\Omega^\cor_{\rho_0}$, and using \eqref{defpartunit}, we obtain that
$$
T=\sigma({\rm c}){\bf e}_{\rm r}^\cor+\big(\sum_{{\rm c}'\neq {\rm c}}\Theta^{({\rm c}')}\big)  \big(\sum_{{\rm s}}T^\side\big) \quad \mbox{ on }\quad \Omega_{\rho_0}^\cor;
$$
since $\sum_{{\rm c}'}\Theta^{(\rm c)'}=1$ on $\Omega$ and $\Theta^\cor=1$ on $ \Omega_{\rho_0}^\cor$, we deduce that $\sum_{{\rm c}'\neq {\rm c}}\Theta^{({\rm c}')}=0$ on $ \Omega_{\rho_0}^\cor$, from which one deduces the first assertion of the lemma.

For the third point, we prove here that if $(\partial_{\rm n}\phi)_{\vert_{\Gamma^{(\rm b)} }}=0$ then $\big[\partial_{\rm n}((T\cdot \nabla)\phi)\big]_{\vert_{\Gamma^{(\rm b)} }}=0$. The results for the other components of $\Gamma^{\rm N}$ is established in the same way. We consider the case where the left endpoint $C_0$ of $\Gamma^{(\rm b)}$ is finite, while the right-one $C_{N+1}$ is infinite (the adaptation to treat the other possible configurations is straightforward).
We want to prove that
\begin{equation}\label{commnortan}
\big[(N \cdot \nabla)(T\cdot \nabla)\phi\big]_{\vert_{\Gamma^{(\rm b)} }}=0,
\end{equation}
if $\partial_{\rm n}\phi=0$ on $\Gamma^{(\rm b)}$. 

In $ \Gamma^{(\rm b)} \cap B(C_0,\rho_0)$, and using the second point of the lemma,
 one has
 $$
 \big[(N\cdot \nabla)(T\cdot \nabla)\phi\big]_{\vert_{\Gamma^{(\rm b)}}}=\big[\frac{1}{r}\partial_\theta\partial_r \phi\big]_{\vert_{\Gamma^{(\rm b)}}},
 $$
 where for the sake of clarity we wrote $(r,\theta)$ instead of $(r^{(C_0)},\theta^{(C_0)})$. It follows that
 \begin{align*}
  \big[(N\cdot \nabla)(T\cdot \nabla)\phi\big]_{\vert_{\Gamma^{(\rm b)}}}&=\frac{1}{r}\partial_r \big[\partial_\theta \phi\big]_{\vert_{\Gamma^{(\rm b)}}}\\
  &=0,
 \end{align*}
 where we used the fact that by Assumption \ref{assconfig2}, $\theta=-\omega^{(C_0)}$ is constant on $\Gamma^{(\rm b)} \cap B(C_0,\rho_0)$ and that $\partial_\theta \phi=r\partial_{\rm n}\phi=0$ on this portion of the boundary.

 On $\Gamma^{(\rm b)} \cap \{ x> \frac{1}{\rho_0}\}$ we deduce from the properties of the family $(\Theta^\cor)_{\rm c}$ that $T=T^{(\rm b)}$ and $N=N^{(\rm b)}$, so that the result stems from \eqref{eqdercommut}.  
 
 We are therefore left to prove \eqref{commnortan} on the portion of $\Gamma^{(\rm b)}$ far from the endpoints $C_0$ and $C_{N+1}$, namely, on $\Gamma^{(\rm b)}\backslash ({B}(C_0,\rho_0)\cup  \{ x> \frac{1}{\rho_0}\})$. 
 From the properties of $(\Theta^\cor)_{\rm c}$, we know  that ${\mathcal T}^{(\rm b)}$ can be chosen thin enough so that in this region $\Theta^\cor\equiv 0$ if ${\rm c}\neq C_0,C_{N+1}$. 
 
 One then has, using also \eqref{condsupp},
 $$
T= \Theta^{(C_0)} \big(\chi^{(C_0)} {\bf e}_{\rm r}^{(C_0)}  + (1-\chi^{(C_0)}  ) T^{(\rm b)}\big) + \Theta^{(C_{N+1})}T^{(\rm b)};
 $$
moreover,  one has $T=(\Theta^{(C_0)}+\Theta^{(C_{N+1})})T^{(\rm b)}=T^{(\rm b)}$ and similarly $N=N^{(\rm b)}$ outside ${\rm Supp} \chi^{(C_0)}$ so this case is also  done.   

Remarking further that on ${\rm Supp } \chi^{(C_0)}$ one has $ \Theta^{(C_0)} \equiv 1$ and $\Theta^{(C_{N+1})}\equiv 0$, we deduce that
  $$
 T=\chi^{(C_0)}{\bf e}_{\rm r}^{(C_0)}+(1-\chi^{(C_0)})T^{(\rm b)}
 \quad\mbox{ in }\quad
 \Gamma^{(\rm b)}\backslash ({B}(C_0,\rho_0)\cup  \{ x> \frac{1}{\rho_{0}}  \}),
 $$
 and that similarly $N=({\bf e}_{\rm r}^{(C_0)})^\perp+(1-\chi^{(C_0)})N^{(\rm b)} $. 
 
 Since by construction $\partial_{\rm n}\chi^{(C_0)}=0$ on $\Gamma^{(\rm b)}\cap \mbox{\rm Supp }\chi^{(C_0)}$, we can infer that on $ \Gamma^{(\rm b)}\backslash ({B}(C_0,\rho_0)\cup  \{ x> {\frac{1}{\rho_{0}} } \})$, one has
\begin{align*}
(N\cdot \nabla)&(T\cdot \nabla)\phi=(\chi^{(C_0)})^2 (({\bf e}_{\rm r}^{(C_0)})^\perp\cdot \nabla)({\bf e}_{\rm r}^{(C_0)}\cdot \nabla)\phi\\
&+(1-\chi^{(C_0)}) \chi^{(C_0)} \big[   (N^{(\rm b)}\cdot \nabla)({\bf e}_{\rm r}^{(C_0)}\cdot \nabla)\phi+ (({\bf e}_{\rm r}^{(C_0)})^\perp\cdot \nabla)(T^{(\rm b)}\cdot \nabla)\phi \big]\\
&+(1-\chi^{(C_0)})^2(N^{(\rm b)}\cdot \nabla)(T^{(\rm b)}\cdot \nabla)\phi;
\end{align*}
the first term of the right-hand side is similar to the one handled in $ \Gamma^{(\rm b)}\cap B(C_0,\rho_0)$ and vanishes for the same reasons; the third one also vanishes because of the commutation property  \eqref{eqdercommut} in ${\mathcal T}^{(\rm b)}$.

Finally, for the second term, assuming without loss of generality that $\phi$ is supported in the tubular neighborhood ${\mathcal T}^{(\rm b)}$ where $N^{(\rm b)}$ and $T^{(\rm b)}$ are constant (because the boundary is flat in that region), we can rewrite the term between brackets as
\begin{equation}\label{polcart}
\partial_r (  (N^{(\rm b)}\cdot \nabla) \phi)+\frac{1}{r}(T^{(\rm b)}\cdot \partial_\theta\nabla)  \phi .
\end{equation}
The first term obviously vanishes on the flat portion of $\Gamma^{(\rm b)}$ near $C_0$. For the second term, we remark that
$$
\partial_\theta \nabla \phi= {\bf e}_{\rm r}^{(\rm b)} \big(-\frac 1 r \partial_\theta\phi+\partial_r\partial_\theta \phi\big) 
+({\bf e}_{\rm r}^{(\rm b)})^\perp \big(\partial_r+\frac{1}{r}\partial_\theta^2\big)\phi.
$$
The first component vanishes on this flat portion of the boundary for the same reason as above. The second component does not vanish but it is orthogonal to $T_{\rm b}$ on  $ \Gamma^{(\rm b)}\cap B(C_0,\rho_0)$ so that its corresponding contribution in \eqref{polcart} also cancels at the boundary. This concludes the proof of the third point of the lemma.

In the end, the forth point of the lemma is a consequence of the fact that near each corner ${\rm c}$, one has $(r^{(\rm c)})^n\partial^n {\bf e}_{\rm r}^{(\rm c)}\in L^\infty(\Omega^{(\rm c)})$. This completes the proof of the lemma.
\end{proof}

\subsection{Higher order weighted estimates for the harmonic extension in corner domains}\label{sectHOHE}

Recall that $\rho$ is the function constructed in Lemma \ref{LMcomm2} and which behaves, near each corner, as the distance to the corner. We know by Proposition \ref{propharmonic} that if $\psi\in \dot{H}^1(\Gamma^{(\rm D)})$, then it is possible to construct its harmonic extension $\psi^\mfh\in \dot{H}^1(\Omega)$; if the Dirichlet data $\psi$ is more regular in the sense that $(\rho\dx)^i\psi\in \dot{H}^{1/2}(\Gamma^{\rm D})$ for $0\leq i\leq n$, then one can construct the extensions $((\rho\dx)^i\psi)^\mfh\in H^1(\Omega)$.

In order to compare these extensions to derivatives of $\psi^\mfh$, we need to introduce an extension of the weighted derivative $\rho\dx$, namely, we introduce the vector field $X$ as
\[
X=\rho T\cdot\nabla;
\]
on $\Gamma^{(\rm D)}$, $X$ coincides with $\rho\dx$, and it is also tangential to all the connected components $\Gamma^\side$ of $\Gamma^*$.

The following proposition shows that $X^n\psi^\mfh$ is also well defined in $H^1(\Omega)$, and provides an upper bound for its norm. It also shows that the difference $\phi_n:=X^n\psi^\mfh-((\rho\dx)^n\psi)^\mfh$ can be controlled by lower order terms.
\begin{proposition}\label{propXnpsih}
Let Assumption \ref{assconfig2} be satisfied and $n\in {\mathbb N}\backslash\{0\}$,  and assume that one has $(\rho\dx)^i\psi\in \dot{H}^{1/2}(\Gamma^{\rm D})$ for all $0\leq i\leq n$. Then $X^n \psi^\mfh \in \dot{H}^1(\Omega)$ and there exists a constant $C>0$ independent of $\psi$ such that
$$
\Vert \nabla (X^n\psi^\mfh) \Vert_{L^2(\Omega)}\leq C \sum_{i=0}^{n}\vert (\rho \dx)^i\psi  \vert_{\dot{H}^{1/2}(\Gamma^{\rm D})}.
$$
Moreover, at leading order, $X^n\psi^\mfh $ behaves like $((\rho\dx)^n\psi)^\mfh$ in the sense that
$$
\Vert \nabla \big(X^n\psi^\mfh -((\rho\dx)^n\psi)^\mfh\big)\Vert_{L^2(\Omega)}\leq C \sum_{i=0}^{n-1} \vert (\rho \dx)^i\psi  \vert_{\dot{H}^{1/2}(\Gamma^{\rm D})}.
$$
\end{proposition}

\begin{proof}
Let us denote $\phi_n=X^n \psi^\mfh-\big((\rho\dx)^n \psi\big)^\mfh$, so that $\phi_n$ solves the elliptic boundary value problem
\begin{equation}\label{systphin}
\begin{cases}
\Delta \phi_n=[\Delta,X^n]\psi^\mfh& \quad \mbox{ in }\quad \Omega,\\
\phi =0 & \quad\mbox{ on }\quad \Gamma^{\rm D},\\
\partial_{\rm n}\phi=0& \quad\mbox{ on }\quad \Gamma^{\rm N}
\end{cases}
\end{equation}
(the fact that $\partial_{\rm n}\phi=0$ on $\Gamma^{\rm N}$ stems from the third point of Lemma \ref{LMcomm1}).

We want to prove that $\phi_n\in H^1_{\rm D}(\Omega)$ and that $\nabla\phi_n$ satisfies the estimate of the proposition. A key step is the following control on the source term in \eqref{systphin}. Note that in the left-hand side of the estimate of the lemma, the integral over $\Omega$ must be understood as a $H^1_{\rm D}(\Omega)'-H^{1/2}_{\rm D}(\Omega)$ duality product. For the sake of clarity we often use this abuse of notation.
\begin{lemma}\label{lemm2}
Let Assumption \ref{assconfig2} be satisfied and $n\in {\mathbb N}$, $n\neq 0$. If $\nabla (X^j \psi^\mfh)\in L^2(\Omega)$ for all $0\leq j\leq n-1$, then $[\Delta,X^n]\psi^\mfh\in H^1_{\rm D}(\Omega)'$ and for all $\phi \in H^1_{\rm D}(\Omega)$, one has
$$
\int_\Omega [\Delta,X^n]\psi^\mfh \phi \leq C \big(\sum_{i=0}^{n-1} \Vert \nabla (X^i \psi^\mfh)\Vert_{L^2(\Omega)}\big)\Vert \nabla \phi \Vert_2.
$$
\end{lemma}
\begin{proof}[Proof of the lemma]
Recalling that if $A$ and $B$ are two operators, one can decompose $[A,B^n]=\sum_{i=0}^{n-1}B^i[A,B]B^{n-i-1}$, and remarking that
$$
 [\Delta,\rho T]\cdot \nabla=(\Delta (\rho T))\cdot \nabla +2 \sum_{j=x,z} \partial_j(\rho T)\cdot \partial_j \nabla ,
$$
one can decompose the source term in \eqref{systphin} as
\begin{align}
\label{decompcomm}
[\Delta,X^n]\psi^\mfh&=  \sum_{i=0}^{n-1} X^{i} \big([\Delta,\rho T]\cdot \nabla X^{n-1-i}\psi^\mfh \big)\\
\nonumber
&=:  \sum_{i=0}^{n-1} (S_{i}^1+S_{i}^2),
\end{align}
with
$$
S_{i}^1=X^{i} \big(\Delta (\rho T)\cdot \nabla X^{n-1-i}\psi^\mfh \big)
\quad\mbox{ and }\quad
S_{i}^2=2  \sum_{j=x,z} X^i \big(\partial_j(\rho T)\cdot \partial_j \nabla X^{n-1-i}\psi^\mfh \big).
$$

It is therefore enough to prove that for all $\phi\in H^1_{\rm D}(\Omega)$, one has
\begin{equation}\label{estSik}
\int_\Omega S^\iota_{i,k}\phi
\leq C \big(\sum_{i=0}^{n-1} \Vert \nabla (X^i \psi^\mfh)\Vert_{L^2(\Omega)}\big)\Vert \nabla \phi \Vert_2
\end{equation}
for $\iota=1,2$ and $0\leq k\leq i$, and where $S^1_{i,k}$ and $S^2_{i,k}$ are defined as
$$
S_{i,k}^1=(X^{i-k} \Delta (\rho T)) \cdot X^k \nabla X^{n-1-i}\psi^\mfh \big)
\quad\mbox{ and }\quad
S_{i,k}^2=X^{i-k}\partial(\rho T)\cdot X^{k}\partial \nabla X^{n-1-i}\psi^\mfh,
$$
where, in each occurrence, $\partial$ can be replaced by either $\partial_x$ or $\partial_z$.

\medskip
\noindent- {\bf Control of $S_{i,k}^1$.} Let us show first that  $\rho S_{i,k}^1 \in L^2(\Omega)$ and 
\begin{equation}\label{estlmS}
\Vert \rho S_{i,k}^1\Vert_{L^2(\Omega)}\leq C \sum_{i=0}^{n-1} \Vert \nabla (X^i \psi^\mfh)\Vert_{L^2(\Omega)}.
\end{equation}
We can notice that for any smooth enough function $u$, the commutator $[X^k,\partial] u$ is a sum of terms of the form
\begin{equation}\label{exprcom1}
X^{k_1}\partial (\rho T_{\beta_1})\times \dots\times X^{k_m}\partial (\rho T_{\beta_m})\times \partial X^{k_0}u,
\end{equation}
for some $m\in {\mathbb N}$, with $\beta_1,\dots \beta_l\in\{1,2\}$, $k_0+\dots+k_m=k-1$. 

Since $X^l \partial(\rho T)$  belongs to $L^\infty(\Omega)$ for all $l\in {\mathbb N}$, 
the fact that $S_{i,k}^1$ satisfies \eqref{estlmS} can be deduced from the fact that for all $0\leq k'\leq k$, one has
$$
\Vert \rho X^{i-k}(\partial^2(\rho T)) \partial (X^{n-1-i+k'}\psi^\mfh ) \Vert_{L^2(\Omega)}
\leq C \Vert  \nabla (X^{n-1-i+k'}\psi^\mfh ) \Vert_{L^2(\Omega)},
$$
which is a direct consequence of the fact that $\rho X^{i-k}(\partial^2(\rho T))$ is bounded in $L^\infty(\Omega)$. 

We now show that the weighted estimate \eqref{estlmS} implies that $S^1_{i,k}\in H^1_{\rm D}(\Omega)'$ and satisfies \eqref{estSik}. For all $\phi \in H^1_{\rm D}(\Omega)$, one has indeed
$$
\int_\Omega S^1_{i,k}\phi \leq \Vert \rho S^1_{i,k}\Vert_{L^2(\Omega)}\Vert \frac{1}{\rho}\phi \Vert_{L^2(\Omega)};
$$
the result follows therefore from \eqref{estlmS} and Hardy's inequality (see for instance Theorem 1.4.4.4 in \cite{Grisvard}) which implies that $\Vert \frac{1}{\rho}\phi \Vert_{L^2(\Omega)}\leq C \Vert \nabla\phi \Vert_{L^2(\Omega)}$ for all $\phi\in H^1_{\rm D}(\Omega)$.

\medskip
\noindent- {\bf Control of $S_{j,k}^2$.} Using twice the fact that the commutator $[X^k,\partial]u$ is a sum of terms of the form \eqref{exprcom1}, the control of $S_{i,k}^2$ is reduced, up to terms that can be treated as $S^1_{i,k}$, to the control of terms of the form $F \partial \nabla X^{j}\psi^\mfh$, with $0\leq j\leq n-1$ and $F$ a bounded function in $\Omega$ with values in $\RR^2$ such that $\rho\partial F$ is also bounded, we have for all $\phi \in H^1_{\rm D}(\Omega)$,
\begin{equation}\label{penible1}
\int_\Omega F\cdot \partial \nabla X^j \psi^\mfh\phi= -\int_\Omega \partial (X^j\psi^\mfh)  \nabla\cdot (F \phi)+\int_{\Gamma^{\rm N}} F\cdot \widetilde{{\bf n}} \partial (X^j\psi^\mfh )\phi ,
\end{equation}
where $\widetilde{\bf n}$ stands for the outwards normal vector on $\Gamma$ (so that $\widetilde{\bf n}=-{\bf n}$ on $\Gamma^{(\rm b)}$).

Using the assumption that $\rho\partial F$ is bounded and Hardy's inequality as above, we have $\Vert \nabla\cdot (F\phi)\Vert_{L^2(\Omega)}\leq C \Vert \nabla\phi\Vert_{L^2(\Omega)}$ and the first term of the right-hand side satisfies the estimate \eqref{estSik}. We therefore focus on the boundary integral.

Since we know by Lemma \ref{LMcomm1} that $\partial_{\rm n} X^j\psi^\mfh=0$ on $\Gamma^{\rm N}$, we can replace $\partial={\bf e}\cdot \nabla$ by $({\bf e}\cdot \widetilde{\bf n}^\perp)\widetilde{\bf n}^\perp\cdot \nabla$ in this boundary integral. Writing $\widetilde{f}=({\bf e}\cdot N^\perp)F\cdot N$,  where $N$ is the extension of ${\bf n}$ provided by Lemma \ref{LMcomm1}, we must therefore control
\begin{equation}\label{penible2}
\int_{\Gamma^{\rm N}} \widetilde{f} \widetilde{\bf n}^\perp\cdot \nabla (X^j\psi^\mfh )\phi=
\int_\Omega \nabla(X^j \psi^\mfh)\cdot \nabla^\perp (\widetilde{f}\phi),
\end{equation}
the right-hand side stemming from Green's identity. 

Again, since $f$ and $\rho \nabla\widetilde{f}$ are in $L^\infty(\Omega)$, Hardy's inequality implies that 
\[
\Vert \nabla^\perp (\widetilde{f}\phi)\Vert_{L^2(\Omega)}\leq C \Vert \nabla \phi\Vert_{L^2(\Omega)},
\]
so that it follows easily that the boundary integral also satisfies the upper bound \eqref{estSik}, which concludes the proof of the lemma.
\end{proof}
If $\nabla (X^i\psi^\mfh)\in L^2(\Omega)$ for $0\leq i\leq n-1$ than the  variational inequality associated with \eqref{systphin} and the lemma yield
\begin{align*}
\Vert \nabla \phi_n\Vert_{L^2(\Omega)}^2 &\leq \big\vert \int_\Omega [\Delta,X^n]\psi^\mfh \phi_n \big\vert\\
&\leq C \sum_{i=0}^{n-1} \Vert \nabla X^i\psi^\mfh \Vert_{L^2(\Omega)} \Vert \nabla \phi_n\Vert_{L^2(\Omega)},
\end{align*}
and therefore
\begin{equation}\label{lmestphin}
\Vert \nabla \phi_n\Vert_{L^2(\Omega)}\leq C \sum_{i=0}^{n-1} \Vert \nabla X^i\psi^\mfh \Vert_{L^2(\Omega)}.
\end{equation}
We can now prove the following lemma.
\begin{lemma}\label{lemmeborneXj}
Under the assumptions of the proposition, one has for all $0\leq j\leq n$ that $\nabla (X^j\psi^\mfh)\in L^2(\Omega)$ with
$$
\Vert \nabla  (X^j\psi^\mfh)\Vert_{L^2(\Omega)}\leq C \sum_{k=0}^{j}\vert (\rho\dx)^k \psi  \vert_{\dot{H}^{1/2}(\Gamma^{\rm D})},
$$
for some constant $C$ independent of $\psi$.
\end{lemma}
\begin{proof}[Proof of the lemma]
The proof is done by induction. The result is obviously true for $j=0$. We assume that it is true for all $0\leq j'\leq j\leq n-1$ and show that it also holds for $j+1$. One has $X^{j+1}\psi^\mfh=((\rho\dx)^{j+1}\psi)^\mfh+\phi_{j+1}$; since \eqref{lmestphin} holds for $\phi_{j+1}$ with $n$ replaced by $j+1\leq n$, the result follows directly since by \eqref{G0kin} and Proposition \ref{propequivnorms2} one has $\Vert\nabla \psi^\mfh\Vert_{L^2(\Omega)}\sim \vert \psi \vert_{\dot{H}^{1/2}(\Gamma^{\rm D})}$.
\end{proof}
Plugging the estimate of the lemma into \eqref{lmestphin} completes the proof of  the proposition.
\end{proof}

\subsection{Commutator estimates for the Dirichlet-Neumann operator}\label{sectCEDN}
%
%
We derive in this section two commutator estimates that will be crucial to  obtain space regularity on the solutions to the Cauchy problem \eqref{CP}-\eqref{CP0}. The first one is stated in the following proposition.

\begin{proposition}\label{propcomm1}
Let Assumption \ref{assconfig2} be satisfied and $n\in {\mathbb N}$, $n\neq 0$.  Let also  $\psi\in \dot{H}^{1/2}(\Gamma^{\rm D})$ be such that $(\rho\dx)^j\psi\in \dot{H}^{1/2}(\Gamma^{\rm D})$ for $0\leq j\leq n$.
Then 
there is a constant $C>0$ independent of $\psi$ such that for 
 all $\tpsi\in \dot{H}^{1/2}(\Gamma^{\rm D})$, one has
$$
\int_{\Gamma^{\rm D}} [(\rho\dx)^n,G_0]\psi \rho \dx  \tpsi
 \leq C\times \big(\sum_{j=0}^{n}\vert (\rho\dx)^j\psi))^\mfh\vert_{\dot{H}^{1/2}(\Gamma^{\rm D})}\big) \vert \tpsi\vert_{\dot{H}^{1/2}(\Gamma^{\rm D})}. 
$$
\end{proposition}


\begin{proof}
Thanks to Proposition \ref{propdense2}, it is enough to prove the estimate of the lemma for $\tpsi\in  {\mathcal D}(\Gamma^{\rm D})$ (in which case the integral of the left-hand side is well defined).
The strategy of the proof is to decompose the commutator $[ (\rho\dx)^n,G_0]\psi$ into a sum of terms involving $(\dz\phi_j)_{\vert_{\Gamma^{\rm D}}}$ ($0\leq j\leq n$) and $G_0 ((\rho\dx)^j \psi)$ ($0\leq j\leq n-1$). 

Using Green's identity, we  can write
\begin{equation}\label{eqBrux}
 \int_{\Gamma^{\rm D}}\dz \phi_j  \rho \dx\tpsi 
=\int_\Omega \Delta \phi_j (\rho\dx \tpsi)^\mfh+\int_\Omega \nabla\phi_j\cdot \nabla(\rho\dx \tpsi)^\mfh,
\end{equation}
for all $\tpsi \in {\mathcal D}(\Gamma^{\rm D})$.

We first give in Lemma \ref{lemm2var} a variant of Lemma \ref{lemm2} that allows one to control the first term of the right-hand side. The control of $(\dz\phi_j)_{\vert_{\Gamma^{\rm D}}}$ is then deduced in Lemma \ref{compcomdual}, while the control of  $G_0 ((\rho\dx)^j \psi)$ is established in Lemma \ref{compcomdual1}. Finally, the afore mentioned decomposition of the commutator is given in Lemma \ref{compcomdual2} after which the end of the proof follows easily.   

The following Lemma is a variant of Lemma \ref{lemm2}, in which we derived an upper bound for $\int_\Omega \Delta\phi_j \phi $ when $\phi\in H^1_{\rm D}(\Omega)$. In the lemma below, $\phi$ is replaced by $X \tpsi^\mfh$, which does not belong to $H^1_{\rm D}(\Omega)$, so that the proof needs to be modified. The assumption that $\phi\in H^1_{\rm D}(\Omega)$ was crucial in two places in the proof of Lemma \ref{lemm2}. 
The first place was that it allowed to use Hardy's inequality $\Vert \frac{1}{\rho}\phi\Vert_{L^2(\Omega)}\leq C \Vert \nabla \phi\Vert_{L^{2}(\Omega)}$. This step is replaced by the observation that, from the definition of $X=\rho T\cdot \nabla$, one has $\Vert \frac{1}{\rho} X\tpsi^\mfh \Vert_{L^2(\Omega)}\leq C \Vert \nabla \tpsi^\mfh\Vert_{L^{2}(\Omega)}$. 
The second place was that since $\phi$ vanishes on $\Gamma^{\rm D}$, there was no boundary integral on this portion of the boundary in \eqref{penible1}. The adaptation requires an analysis of this new boundary integral whose control exploits the specific structure of $X\tpsi^\mfh$.
\begin{lemma}\label{lemm2var}
Under the assumptions of the proposition,  let $0\leq j\leq n$, and  $\psi\in \dot{H}^{1/2}(\Gamma^{\rm D})$ be such that for all $0\leq k\leq j$, one has $(\rho\dx)^k\psi\in \dot{H}^{1/2}(\Gamma^{\rm D})$. Then for all $\tpsi\in \dot{H}^{1/2}(\Gamma^{\rm D})$.
$$
\int_\Omega \Delta\phi_j X\tpsi^\mfh \leq C\times \big(\sum_{k=0}^{j}\vert  (\rho\dx)^k\psi)\vert_{\dot{H}^{1/2}(\Gamma^{\rm D})}\big) \vert \tpsi\vert_{\dot{H}^{1/2}(\Gamma^{\rm D})},
$$
where we recall that $\phi_j=X^j\psi^\mfh-((\rho\dx)^j\psi)^\mfh$.
\end{lemma}
\begin{proof}[Proof of the lemma]
As above, it is sufficient to prove the estimate for $\tpsi\in  {\mathcal D}(\Gamma^{\rm D})$. We actually prove a slightly more general result, namely, that for all $0\leq j \leq n$ and all $f\in L^\infty(\Omega)$ such that $\rho\nabla f\in L^\infty(\Omega)$, one has
\begin{equation}\label{claim2}
\int_\Omega f \Delta\phi_j X\tpsi^\mfh \leq C(\Vert f\Vert_{L^\infty(\Omega)},\Vert \rho\nabla f\Vert_{L^\infty(\Omega)})\big(\sum_{k=0}^{j}\Vert \nabla( (\rho\dx)^k\psi))^\mfh\Vert_2\big) \Vert \nabla\tpsi^\mfh\Vert_2.
\end{equation}

Since $\phi_0=0$, the claim \eqref{claim2} is satisfied for $j=0$. In order to prove that it holds for all $1\leq j\leq n$, we proceed by a finite induction. Let us assume that \eqref{claim2} is satisfied for $0\leq i \leq j-1$, and let us prove that it holds for $j$ also. 

Decomposing $\Delta\phi_j=[\Delta,X^j]\tpsi^\mfh$ as in \eqref{decompcomm}, we are led to prove that the upper bound of \eqref{claim2} holds for
$$
\int_\Omega f S^\iota_{i,k} X\widetilde{\psi}^\mfh,
$$
with $\iota=1,2$ and  $0\leq k\leq i\leq j-1$, and where we recall that
\begin{align*}
S_{i,k}^1&=(X^{i-k} \Delta (\rho T)) \cdot X^k \nabla X^{j-1-i}\psi^\mfh \big),\\
S_{i,k}^2&=X^{i-k}\partial(\rho T)\cdot X^{k}\partial \nabla X^{j-1-i}\psi^\mfh.
\end{align*}

When $\iota=1$, we write
$$
\int_\Omega f S^1_{i,k}X\widetilde{\psi}^\mfh=\int_\Omega f (\rho S^1_{i,k}) T\cdot \nabla\tpsi^\mfh,
$$
so that the upper bound \eqref{claim2} stems from Cauchy-Schwarz inequality and the upper bound \eqref{estlmS} on $\Vert \rho S^1_{i,k}\Vert_{L^2(\Omega)}$. 

For $\iota=2$ and proceeding as in the proof of Lemma \ref{lemm2} one sees that it is enough to establish the desired upper bound for terms of the form
$$
I_l:= \int_\Omega F \cdot (\nabla \partial  X^{l}\psi^\mfh) X\tpsi^\mfh \qquad(0\leq l\leq j-1),
$$
where  $\Vert F\Vert_{L^\infty(\Omega)}+\Vert \rho \nabla F\Vert_{L^\infty(\Omega)}\leq C\big( \Vert f\Vert_{L^\infty(\Omega)}+\Vert \rho \nabla f\Vert_{L^\infty(\Omega)} )$.  

Integrating by parts, and using the fact that $\widetilde{\bf n}={\bf n}$ on $\Gamma^{\rm D}$, we get
$$
I_l=-\int_\Omega  (\partial  X^{l}\psi^\mfh)\nabla\cdot (F X\tpsi^\mfh)+\int_{\Gamma^{\rm N}} F\cdot \widetilde{\bf n} \partial  X^{l}\psi^\mfh X\tpsi^\mfh+\int_{\Gamma^{\rm D}} F\cdot {\bf n} \partial  X^{l}\psi^\mfh X\tpsi^\mfh
$$
(note that the boundary integral on $\Gamma^{\rm D}$ was not present in the computations of the proof of Lemma \ref{lemm2} because $\phi$ vanished on $\Gamma^{\rm D}$). We use the notation $A\sim B$ if $\vert A-B\vert$ is bounded from above by a term of the same type as the right-hand side \eqref{claim2}.

For the first term in the right-hand side of $I_l$, one has
\[
\begin{split}
\int_\Omega  (\partial  X^{l}\psi^\mfh)\nabla\cdot (F X\tpsi^\mfh)=& \int_\Omega  (\partial  X^{l}\psi^\mfh\sum_{m=1,2}\nabla\cdot (F \rho T_m) \partial_m\tpsi^\mfh+\int_\Omega  F(\partial  X^{l}\psi^\mfh) \cdot X\nabla\tpsi^\mfh\\
\sim& \int_\Omega  F(\partial  X^{l}\psi^\mfh)\cdot X\nabla\tpsi^\mfh,
\end{split}
\]
where we used Lemma \ref{lemmeborneXj} to bound $\Vert \partial X^l \psi^\mfh\Vert_2$ from above. 
Integrating by parts on the last integral leads to
\begin{align}
\label{simplIPP1}
\int_\Omega  (\partial  X^{l}\psi^\mfh)\nabla\cdot (F X\tpsi^\mfh)&\sim \int_\Omega  X^*\big(F(\partial  X^{l}\psi^\mfh)\big) \cdot\nabla\tpsi^\mfh\\
\nonumber
&\sim  \int_\Omega  F(\partial  X^{l+1}\psi^\mfh) \cdot\nabla\tpsi^\mfh,
\end{align}
which is controlled by the desired upper bound owing to Lemma \ref{lemmeborneXj}.

The second term of the right-hand side of $I_1$ is controlled exactly as the boundary integral in \eqref{penible1}.
It remains therefore to control the third term, that is, the integral on the component $\Gamma^{\rm D}$ of the boundary. We need to distinguish whether $\partial=\partial_x$ or $\partial=\partial_y$.
\begin{itemize}
\item If $\partial=\partial_x$. Since on $\Gamma^{\rm D}$, one has $\dx=-{\bf n}^\perp\cdot \nabla$, we get by Green's identity that
\begin{equation}\label{ifdx}
\int_{\Gamma^{\rm D}} F\cdot \widetilde{\bf n} \partial_x  X^{l}\psi^\mfh X\tpsi^\mfh=-\int_\Omega  \nabla(X^{l}\psi^\mfh )\cdot \nabla^\perp\big( g X\tpsi^\mfh\big),
\end{equation}
where $g=(F\cdot N)\chi_{\rm D}$ and $\chi_{\rm D}\in C^\infty(\Omega\cup \Gamma^*)\cap L^\infty(\Omega)$ is a cutoff function such that $\chi_{\rm D}=1$ on $\Gamma^{\rm D}$ and $\chi_{\rm N}=0$ on $\Gamma^{\rm N}$ and satisfying $\rho\nabla\chi_{\rm D}\in L^\infty(\Omega)$. The right-hand side in the above equality satisfies therefore
\begin{align*}
-\int_\Omega  \nabla(X^{l}\psi^\mfh )\cdot \nabla^\perp\big( g X\tpsi^\mfh\big)
=&-\int_\Omega  \nabla(X^{l}\psi^\mfh )\cdot (\rho \nabla^\perp g ) T\cdot \nabla \tpsi^\mfh\\
&-\int_\Omega  \nabla(X^{l}\psi^\mfh )\cdot \big(g \nabla^\perp (X\tpsi^\mfh)\big)\\
\sim & -\int_\Omega  \nabla(X^{l}\psi^\mfh )\cdot \big(g X\nabla^\perp \tpsi^\mfh\big),
\end{align*}
where we used the control on $\Vert \nabla (X^l \psi^\mfh)\Vert_{L^2(\Omega)}$ provided by Proposition \ref{propXnpsih} for the first integral and also used the fact that $X$ commutes with $\nabla^\perp$ in the second integral, up to terms that satisfy the upper bound stated in the lemma. 

We can now conclude by observing that
\begin{align}
\label{simplIPP2}
\int_\Omega  \nabla(X^{l}\psi^\mfh )\cdot \big(g X\nabla^\perp \tpsi^\mfh\big)&=
\int_\Omega  X^*\big( g \nabla(X^{l}\psi^\mfh )\big)\cdot \nabla^\perp \tpsi^\mfh\\
\nonumber
&\sim \int_\Omega  \big( g \nabla(X^{l+1}\psi^\mfh )\big)\cdot \nabla^\perp \tpsi^\mfh.
\end{align}
Since $l\leq j-1$, this last term also satisfies the desired upper bound.
\item If $\partial=\partial_z$. Since on $\Gamma^{\rm D}$ one has $\dz={\bf n}\cdot \nabla$, and using the fact that $\partial_{\rm n} (X^l \psi^\mfh)=0$ on $\Gamma^{\rm N}$ by the third point of Lemma \ref{LMcomm1},  Green's identity gives this time,
\begin{align*}
\int_{\Gamma^{\rm D}} 
F\cdot {\bf n} \partial_z  X^{l}\psi^\mfh X\tpsi^\mfh=&\int_\Omega \Delta (X^l\psi^\mfh) (
F\cdot N) X\tpsi^\mfh\\
&+\int_\Omega  \nabla(X^{l}\psi^\mfh )\cdot \nabla\big( (F\cdot N) X\tpsi^\mfh\big).
\end{align*}
The first term of the right-hand side is of the form considered in the claim \eqref{claim2} (with $F\cdot N$ playing the role of $f$); since $l\leq j-1$, the induction assumption ensures that it satisfies the upper bound of the claim. The second term is treated exactly as the right-hand side of  \eqref{ifdx} above.
\end{itemize}
The induction is therefore complete and \eqref{claim2} is proved. The lemma follows by taking $f=1$.
\end{proof}

We can now provide a control for the left-hand side of \eqref{eqBrux}.
\begin{lemma}\label{compcomdual}
Under the assumptions of the proposition, let $0\leq j\leq n$, and  $\psi\in \dot{H}^{1/2}(\Gamma^{\rm D})$ be such that for all $0\leq k\leq j$, one has $(\rho\dx)^k\psi\in \dot{H}^{1/2}(\Gamma^{\rm D})$. Then 
  there exists a constant
 $C>0$ independent of $\psi$ such that for all $\tpsi\in \dot{H}^{1/2}(\Gamma^{\rm D})$,
$$
\int_{\Gamma^{\rm D}}  \dz\phi_j \rho\dx \tpsi \leq  C\times \big(\sum_{k=0}^{j}\vert (\rho\dx)^k\psi)\vert_{\dot{H}^{1/2}(\Gamma^{\rm D})}\big) \vert \tpsi \vert_{\dot{H}^{1/2}(\Gamma^{\rm D})}.  
$$
\end{lemma}   
%
\begin{proof}[Proof of the lemma]
As above, we can work with $\tpsi \in {\mathcal D}(\Gamma^{\rm D})$. We recall that by \eqref{eqBrux} we have
\begin{align*}
\int_{\Gamma^{\rm D}}\dz \phi_j  \rho \dx\tpsi 
&=\int_\Omega \Delta \phi_j (\rho\dx \tpsi)^\mfh+\int_\Omega \nabla\phi_j\cdot \nabla(\rho\dx \tpsi)^\mfh\\
&=:K_1+K_2,
\end{align*}
for all $\tpsi\in {\mathcal D}(\Gamma^{\rm D})$.
We now turn to control $K_1$ and $K_2$. \\
- Control of $K_1$.  We can decompose 
$$
K_1=-\int_\Omega \Delta \phi_j  
\widetilde{\phi}_1
+\int_\Omega (\Delta \phi_j )  
X\widetilde{\psi}^\mfh,
$$
with $\widetilde{\phi}_1=-((\rho\dx) \widetilde{\psi})^\mfh+X \widetilde{\psi}^\mfh $.
The first term can be controlled by Lemma \ref{lemm2} because $\widetilde{\phi}_1\in H^1_{\rm D}(\Omega)$, while the second one is controlled by its variant Lemma \ref{lemm2var}.\\
- Control of $K_2$. We have  
\begin{align*}
K_2
  &= \int_\Omega \nabla\phi_j\cdot 
  \nabla\big((\rho\dx \tpsi)^\mfh-X\tpsi^\mfh\big)
  +\int_\Omega \nabla\phi_j\cdot \nabla X\tpsi^\mfh\\
  &=\int_\Omega \nabla\phi_j\cdot 
  \nabla\big((\rho\dx \tpsi)^\mfh-X\tpsi^\mfh\big)
  +\int_\Omega \nabla\phi_j\cdot [\nabla, X]\tpsi^\mfh+\int_\Omega X^*\nabla\phi_j\cdot \nabla \tpsi^\mfh,
\end{align*}
so that the desired upper bound easily follows from Proposition \ref{propXnpsih} and the equivalence of norms $\Vert\nabla \tpsi^\mfh\Vert_{L^2(\Omega)}\sim \vert \tpsi \vert_{\dot{H}^{1/2}(\Gamma^{\rm D})}$.


This completes the proof of the lemma.
\end{proof}

To complement Lemma \ref{compcomdual}, we provide a similar control for $G_0((\rho \dx)^j\psi)$ (note however that the sum on the right-hand side takes its range in $0\leq k\leq j+1$ rather than $0\leq k\leq j$).
\begin{lemma}\label{compcomdual1}
Under the assumptions of the proposition, let $0\leq j\leq n-1$, and  $\psi\in \dot{H}^{1/2}(\Gamma^{\rm D})$ be such that for all $0\leq k\leq j+1$, one has $(\rho\dx)^k\psi\in \dot{H}^{1/2}(\Gamma^{\rm D})$. Then there exists a constant $C>0$ independent of $\psi$ such that for all $\tpsi\in \dot{H}^{1/2}(\Gamma^{\rm D})$,
$$
\int_{\Gamma^{\rm D}} G_0((\rho\dx)^{j}\psi) (\rho\dx\tpsi) \leq C\times \big(\sum_{k=0}^{j+1}\vert (\rho\dx)^k\psi) \vert_{\dot{H}^{1/2}(\Gamma^{\rm D})}\big) \vert \tpsi\vert_{\dot{H}^{1/2}(\Gamma^{\rm D})}.
$$
\end{lemma}  
\begin{proof}
As above, we can work with $\tpsi\in {\mathcal D}(\Gamma^{\rm D})$. Throughout this proof, we write $A\sim B$ if $\vert A-B\vert$ is bounded from above by the right-hand side of the estimate stated in the lemma. 

From the variational definition of the Dirichlet-Neumann operator, we have
\begin{align*}
\int_{\Gamma^{\rm D}} G_0((\rho\dx)^{j}\psi) (\rho\dx\tpsi )&=\int_\Omega \nabla ((\rho\dx)^{j}\psi)^\mfh\cdot \nabla (\rho\dx\tpsi )^\mfh\\
&\sim \int_\Omega \nabla ((\rho\dx)^{j}\psi)^\mfh\cdot X \nabla \tpsi ^\mfh,
\end{align*}
the second line stemming as usual from Proposition \ref{propXnpsih}. Integrating by parts and recalling that $X^*=-X-\nabla\cdot (\rho T)$, it follows that
\begin{align*}
\int_{\Gamma^{\rm D}} G_0((\rho\dx)^{j}\psi) (\rho\dx\tpsi )&\sim - \int_\Omega X\nabla ((\rho\dx)^{j}\psi)^\mfh\cdot  \nabla \tpsi ^\mfh\\
&\sim  - \int_\Omega \nabla ((\rho\dx)^{j+1}\psi)^\mfh\cdot  \nabla \tpsi ^\mfh.
\end{align*}
The result then easily follows from Cauchy-Schwartz inequality.
\end{proof}
The following lemma now provides a decomposition of the commutator $[(\rho\dx)^n,G_0]\psi$ into several terms that can be controlled using the above results.

\begin{lemma}\label{compcomdual2}
Under the assumptions of the proposition, let $n\in {\mathbb N}$, and assume that $\psi$ is such that  $(\rho\dx)^i\psi\in \dot{H}^{1/2}(\Gamma^{\rm D})$ for all $0\leq i\leq n$. Then there are functions $F_i^{(n)}\in W^{1,\infty}(\Gamma^{\rm D})$ ($0\leq i\leq n-1$) that do not depend on $\psi$ and such that
$$
[(\rho\dx)^n,G_0]\psi=(\dz \phi_n)_{\vert_{\Gamma^{\rm D}}}+\sum_{i=0}^{n-1}F_i^{(n)} \big[G_0 ((\rho\dx)^i\psi)+(\dz \phi_{i})_{\vert_{\Gamma^{\rm D}}}\big].
$$
\end{lemma}  
\begin{proof}[Proof of the lemma]
Since by definition of $G_0$  we have $G_0\psi=(\partial_z \psi^\mfh)_{\vert_{\Gamma^{\rm D}}}$, and because $X \psi^\mfh=\rho T\cdot\nabla \psi^\mfh=\rho\dx \psi^\mfh$ on $\Gamma^{\rm D}$,  we can write
\begin{align}
\nonumber
[(\rho\dx)^n,G_0]\psi&= \big(X^n \dz \psi^\mfh\big)_{\vert_{\Gamma^{\rm D}}}-\big(\partial_z ((\rho\dx)^n\psi)^\mfh\big)_{\vert_{\Gamma^{\rm D}}}\\
\label{eqcom1}
&= ([X^n,\dz]  \psi^\mfh)_{\vert_{\Gamma^{\rm D}}}+\big(\partial_z \phi_n\big)_{\vert_{\Gamma^{\rm D}}}.
\end{align}

In order to understand the structure of the commutator in the right-hand side, let us recall that $X=\rho T\cdot \nabla$, so that we have for all $\varphi\in {\mathcal D}(\Omega\cup \Gamma^*)$,
$$
X \dz \varphi=\dz  (\rho T\cdot \nabla\varphi)
-\dz (\rho T) \cdot \nabla \varphi.
$$
Since $\dz \rho$ vanishes on $\Gamma^{\rm D}$, we deduce that
$$
\big([X,\dz]\varphi\big)_{\vert_{\Gamma^{\rm D}}}=-\rho (\dz T\cdot \nabla\varphi\big)_{\vert_{\Gamma^{\rm D}}}.
$$

Choosing $T$  as in the proof of Lemma \ref{LMcomm1} and using the fact that $\partial_z \chi^\cor=0$ and $\partial_z(\sigma\cor {\bf e}_{\rm r}^\cor)=\frac{1}{r^\cor}{\bf e}_z$ on $\Gamma^{\rm D}$, one gets that
 $$
 (\dz T)_{\vert_{\Gamma^{\rm D}}}=\big(\sum_{{\rm c}\in {\mathcal C}}\frac{\chi^\cor}{r^\cor}\big){\bf e}_z,
 $$
so that we can conclude that
$$
\big([X, \dz ]\varphi\big)_{\vert_{\Gamma^{\rm D}}}={F_0^{(1)}}  (\dz \varphi)_{\vert_{\Gamma^{\rm D}}},
\quad\mbox{ with }\quad
F_0^{(1)}=- \big(\sum_{{\rm c}\in {\mathcal C}}\chi^\cor\frac{\rho}{r^\cor}\big)_{\vert_{\Gamma^{\rm D}}} .
$$

Using this identity repeatedly, one can prove by a finite induction that for all $n\in {\mathbb N}$, $n\geq 1$, one has
$$
\big([X^n, \dz ]\varphi\big)_{\vert_{\Gamma^{\rm D}}}=\sum_{i=0}^{n-1}{F_i^{(n)}} (\dz X^i\varphi)_{\vert_{\Gamma^{\rm D}}},
$$
where the functions $F^{(j)}_i$ belong to  $W^{1,\infty}(\Gamma^{\rm D})$. Since $(\dz X^i\psi^\mfh)_{\vert_{\Gamma^{\rm D}}}=G_0 ((\rho\dx)^i\psi)+(\dz \phi_{i})_{\vert_{\Gamma^{\rm D}}}$, the above formula applied with $\varphi=\psi^\mfh$ yields
$$
([X^n, \dz ]  \psi^\mfh)_{\vert_{\Gamma^{(\rm D)}}}=\sum_{i=0}^{n-1}{F_i^{(n)}} \big[G_0 ((\rho\dx)^i\psi)+(\dz \phi_{i})_{\vert_{\Gamma^{\rm D}}}\big].
$$
Together with \eqref{eqcom1},  this directly implies the result.
\end{proof}
The proposition is now a direct consequence of Lemmas  \ref{compcomdual},  \ref{compcomdual1}  and  \ref{compcomdual2}. 
\end{proof}

When $n=1$, a variant of the commutator estimate given in Proposition \ref{propcomm1} is the following (it requires a control of $\rho\dx \tpsi$ in $H^{1/2}(\Gamma^{\rm D})$ instead of $\rho\dx \psi \in \dot{H}^{1/2}(\Gamma^{\rm D})$).
\begin{proposition}\label{propcomm2}
Let Assumption \ref{assconfig2} be satisfied. There exists $C>0$ such that for all  $\psi\in \dot{H}^{1/2}(\Gamma^{\rm D})$
and
 all $\tpsi\in \dot{H}^{1/2}(\Gamma^{\rm D})$ such that $\rho\dx\tpsi\in {H}^{1/2}(\Gamma^{\rm D})$,  one has
$$
\int_{\Gamma^{\rm D}} [\rho\dx,G_0]\psi \rho \dx  \tpsi
 \leq C \vert \psi \vert_{\dot{H}^{1/2}} \big(  \vert \tpsi \vert_{\dot{H}^{1/2}}  +  \vert \rho\dx \tpsi \vert_{{H}^{1/2}}  \big).
 $$
\end{proposition}
\begin{remark}\label{remcommutateur}
Remarking that  $([\dx(\rho^2 \dx \cdot),G_0]\psi,\tpsi)_{L^2}=(\rho\dx \psi,[\rho\dx,G_0]\tpsi)_{L^2}
-([\rho\dx,G_0]\psi,\rho\dx \tpsi)_{L^2}$, we can combine Propositions \ref{propcomm1} and \ref{propcomm2} to obtain
$$
\big( [\dx(\rho^2 \dx \cdot),G_0]\psi,\tpsi\big)_{L^2(\Gamma^{\rm D})} \lesssim \big( \vert \psi \vert_{\dot{H}^{1/2}(\Gamma^{\rm D})}+\vert \rho \dx \psi \vert_{{H}^{1/2}(\Gamma^{\rm D})} \big) \vert \tpsi \vert_{\dot{H}^{1/2}(\Gamma^{\rm D})}.
$$
\end{remark}
\begin{proof}
By Lemma \ref{compcomdual2} we can decompose
$$
[\rho\dx,G_0]\psi  = (\partial_z \phi_1)_{\vert_{\Gamma^{\rm D}}}+F G_0 \psi,
$$
with $F\in W^{1,\infty}(\Gamma^{\rm D})$ and $\phi_1=X \psi^\mfh -(\rho\dx\psi)^\mfh$. 

Since 
\[
\vert \int_{\Gamma^{\rm D}} F G_0 \psi  \rho \dx  \tpsi \vert \lesssim \vert \psi \vert_{\dot{H}^{1/2}(\Gamma^{\rm D})} \vert \rho\dx \tpsi\vert_{{H}^{1/2}(\Gamma^{\rm D})},
\]
we are left to control the contribution of $ (\partial_z \phi_1)_{\vert_{\Gamma^{\rm D}}}$. 

The proof is very similar to the proof of Lemma \ref{lemm2var} (it is actually simpler since it consists in skipping the integration by parts of \eqref{simplIPP1} and \eqref{simplIPP2}), and is therefore omitted.
\end{proof}

\section{Well-posedness in partially weighted spaces}

We have seen that Corollary \ref{corohigherT} allows us to construct solutions in $U\in {\mathbb V}^n_T$, with ${\mathbb V}^n_T$ defined in \eqref{defWT} and $n\in {\mathbb N}$, to the evolution equation
\begin{equation}\label{CPcompbis}
\partial_t U +{\bf A}U=F,
\end{equation}
with $U=(\zeta,\psi)^{\rm T}$, $F=(f,g)^{\rm T}$ and ${\bf A}=\begin{pmatrix} 0 & -G_0  \\  \gr  & 0\end{pmatrix}$, 
with initial condition
\begin{equation}\label{CP0bis}
U_{\vert_{t=0}}=U^{\rm in}.
\end{equation}
This implies that for all $0\leq j \leq n$, one has $(-{\bf A})^j U \in C([0,T];{\mathbb X})$, where we recall that ${\mathbb X}$ is the energy space defined as ${\mathbb X}={\mathcal L}^2(\Gamma^{\rm D})\times \dot{{\mathcal H}}^{1/2}(\Gamma^{\rm D})$ if $\Gamma^{\rm D}$ is bounded and ${\mathbb X}=L^2(\Gamma^{\rm D})\times \dot{H}^{1/2}(\Gamma^{\rm D})$ otherwise. 

From this result, it is possible to deduce that  $U=(\zeta,\psi)^{\rm T}$ belongs to $C([0,T];H^{n/2}(\Gamma^{\rm D})\times \dot{H}^{((n+1)/2}(\Gamma^{\rm D}))$ if $n\leq 2$, as shown in Corollary \ref{space regularity}, and for $n\geq 3$ under an additional smallness assumption on the angles at the corners of the fluid domain, as shown in  Corollary \ref{space regularity2}. 

We also derived in Section \ref{sectcommutator} some estimates on the commutators $[(\rho\dx)^j, G_0]\psi$ of the Dirichlet-Neumann operator with weighted derivatives which are classically used in the analysis of elliptic PDEs in corner domains \cite{Grisvard,Dauge,MR}. The motivation to derive these commutator estimates was to use them to control $(\rho\dx)^j U$, and more generally $(\rho\dx)^j (-{\bf A})^kU$, with $0\leq j+k\leq n$,   in $C([0,T];{\mathbb X})$ and without any smallness assumption on the angles. 

This suggests that a functional space containing all the functions $V\in {\mathbb X}$ such that all the partially weighted quantities $(\rho\dx)^j (-{\bf A})^k V$, with $0\leq j+k\leq n$, are in ${\mathbb X}$ is appropriate for a good well-posedness theory for the initial value problem \eqref{CPcompbis}-\eqref{CP0bis}. Such partially weighted functional spaces are introduced in \S \ref{sectmixspaces}. 

In \S \ref{sectkeyprop}, we then use the commutator estimates of Section \ref{sectcommutator} to establish a key result of transfer of regularity, namely, that a control of $(\rho\dx)^j (-{\bf A})^{k+1}U$ in $C([0,T];{\mathbb X})$ allows one to control $(\rho\dx)^{j+1} (-{\bf A})^{k}U$ in the same space. Starting from Corollary  \ref{corohigherT} that provides a control of $(-{\bf A})^n U$, this transfer property allows one to control all the partially weighted norm of order $n$ and to prove, in \S \ref{sectMR}, a well-posedness result in partially weighted functional spaces. This result allows one to go beyond the $H^1\times \dot{H}^{3/2}$ regularity threshold in the interior of $\Gamma^{\rm D}$, without any smallness assumption on the angles.

\subsection{The spaces ${\mathbb Y}^n$ and ${\mathbb W}^n_T$}\label{sectmixspaces}
Recalling that ${\mathbb X}$ is defined by \eqref{defXbounded} if $\Gamma^{\rm D}$ is bounded and by \eqref{defXunbounded} otherwise,  we introduce the spaces ${\mathbb Y}^n$ defined for all $n\in {\mathbb N}$ as
\begin{equation}\label{defV0}
{\mathbb Y}^n:= \lbrace U\in {\mathbb X}, \, \vert U \vert_{{\mathbb Y}^n}<\infty \rbrace,
\end{equation}
with
$$
\vert U \vert_{{\mathbb Y}^n}:=\sum_{0\leq j+k\leq n} \vert (\rho \dx )^{j}(-{\bf A})^{k} U\vert_{\mathbb X}.
$$
For time dependent functions, we also need to introduce the spaces ${\mathbb W}^n_T$ for $T>0$ by
\begin{equation}\label{defWT2}
{\mathbb W}^n_T:= \bigcap_{l=0}^n C^l([0,T];{\mathbb Y}^{n-l}), 
\end{equation}
endowed with the norm
$
\Vert U \Vert_{{\mathbb W}^n_T}:=\sup_{t\in [0,T]} {\mathcal N}^n(U(t))<0$, where
\begin{equation}\label{defN}
	{\mathcal N}^n(U(t)):= 
	\sum_{l=0}^n \vert  \dt^{l} U (t)\vert_{{\mathbb Y}^{n-l}}.
\end{equation}
Throughout this section, we use the following notation,
$$
U_{j,k}:=(\rho\dx)^j(-{\bf A})^k U.
$$


\subsection{A key proposition}\label{sectkeyprop}

The following proposition shows that one can control $(\rho\dx)^{j+1}(-{\bf A})^k U$ in terms of $(\rho\dx)^{j}(-{\bf A})^{k+1}U$, which is the key step to deduce our main result from Theorem \ref{theoWPbounded} or \ref{theoWPunbounded}. 
Thanks to this property, it will be possible to deduce weighted estimates from the time regularity provided by Corollary \ref{corohigherT}.
\begin{proposition}\label{keyprop}
Let Assumption \ref{assconfig2} be satisfied.  
Let $T>0$, $n\in {\mathbb N}$ and $U\in C([0,T];{\mathbb Y}^n)$ solve \eqref{CPcompbis}-\eqref{CP0bis}. Let also $0\leq j,k\leq n$ be such that $j+k=n$. If moreover $U_{j,k+1}\in C([0,T];{\mathbb X})$, $U_{j+1,k}^{\rm in}\in {\mathbb X}$ and $F_{j+1,k}\in C([0,T];{\mathbb X})$ then  one has $U_{j+1,k}\in C([0,T];{\mathbb X})$ and the following energy estimate holds
$$
 \vert U_{j+1,k}(t) \vert_{\mathbb X}^2\leq  e^t \vert U_{j+1,k}^{\rm in} \vert_{\mathbb X}^2+ C \int_0^t e^{t-s}\big( \vert F_{j+1,k}(s,\cdot) \vert_{\mathbb X}^2 + \vert U(s,\cdot) \vert_{{\mathbb Y}^n}^2 + \vert U_{j,k+1} (s,\cdot)\vert_{\mathbb X}^2 \big){\rm d}s,
$$
for some constant $C>0$ independent of $U$ and $F$.
\end{proposition}
\begin{proof}
We first establish in Step 1 the energy estimate stated in the lemma for more regular solutions $U$. In order to show that the energy estimate remains valid at the level of regularity assumed in the statement of the proposition, we approximate $U_{j+1,k}$ by a regularizing sequence $(V_\epsilon)$, where the $V_\epsilon$ satisfy the energy estimate, and then pass to the limit. The regularization process is studied in Step 2, and the energy estimates for $V_\epsilon$, which involve additional commutator estimates compared to the computations of Step 1 due to the presence of smoothing operators, are performed in Step 3. The convergence step and conclusion is then done in Step 4.

\medskip

\noindent{\bf Step 1.} A priori estimates for more regular solutions. Under the assumption $U\in C([0,T];{\mathbb Y}^n)$, we know that $U_{j,k}\in C([0,T];{\mathbb X})$. The idea is to write the system satisfied by $\rho\dx U_{j,k}$ and to perform energy estimates, assuming that we have enough regularity on $U_{j,k}$ to do so. 

Applying $(\rho\dx)^{j+1}(-{\bf A})^k$ to \eqref{CPcompbis}, one gets
$$
\dt (\rho \dx U_{j,k}) +{\bf A} (\rho \dx U_{j,k}) =F_{j+1,k} +  \begin{pmatrix} [(\rho \dx)^{j+1}, G_0] \psi_{0,k} \\0 \end{pmatrix}.
$$
Note that if $\rho\dx U_{j,k}\in C^1([0,T];{\mathbb X})$ then 
\[
\langle \dt (\rho \dx U_{j,k}),\rho\dx U_{j,k}\rangle_X=\frac{1}{2}\frac{\rm d}{{\rm d}t} \vert \rho\dx U_{j,k}\vert_{\mathbb X}^2,
\]
and that if $\rho\dx U_{j,k}\in C([0,T];H^{1/2}\times \dot{H}^{1/2})$ then we have
\[
\langle \rho \dx U_{j,k},(-{\bf A})\rho\dx U_{j,k}\rangle_{\mathbb X}=0.
\]

Therefore, if we assume the additional regularity $\rho\dx U_{j,k} \in C^1([0,T];{\mathbb X})\cap C([0,T];H^{1/2}\times \dot{H}^{1/2})$ one gets by 
taking the scalar product of the above equation with  $\rho\dx U_{j,k}$,
$$
\frac{\rm d}{{\rm d}t}\frac{1}{2} \vert (\rho\dx) U_{j,k} \vert_{\mathbb X}^2 = \langle F_{j+1,k},\rho\dx U_{j,k} \rangle_{\mathbb X}+ \big( [(\rho \dx)^{j+1}, G_0]\psi_{0,k}, \rho\dx \zeta_{j,k} \big)_{L^2(\Gamma^{\rm D})}.
$$

Using Proposition \ref{propcomm1} we get that
\begin{align*}
\big( [(\rho \dx)^{j+1}, G_0]\psi_{0,k}, \rho\dx \zeta_{j,k} \big)_{L^2(\Gamma^{\rm D})} &\leq C \big( \sum_{l=0}^{j+1} \vert (\rho \dx )^l\psi_{0,k} \vert_{\dot{H}^{1/2}}\big)\vert \zeta_{j,k} \vert_{\dot{H}^{1/2}}\\
&\leq C \big( \vert U \vert_{{\mathbb Y}^n}+\vert \rho \dx U_{j,k} \vert_{{\mathbb X}}\big)\vert U_{j,k+1} \vert_{{\mathbb X}},
\end{align*}
where we used the fact that $(\rho\dx)^l\psi_{0,k}$ is the second component of $U_{l,k}$ and ${{\mathtt g}}\zeta_{j,k}$ is the second component of $U_{j,k+1}=(\rho\dx)^j (-{\bf A}) U_{0,k}$.
We deduce that
$$
\frac{\rm d}{{\rm d}t}\frac{1}{2} \vert (\rho\dx) U_{j,k} \vert_{\mathbb X}^2 \lesssim  \big(\vert F_{j+1,k} \vert_{\mathbb X} + \vert U_{j,k+1} \vert_{\mathbb X}   \big) \vert \rho \dx U_{j,k} \vert_{\mathbb X}+\vert U \vert_{{\mathbb Y}^n} \vert U_{j,k+1}\vert_{\mathbb X} ,
$$
and therefore
$$
 \vert U_{j+1,k}(t) \vert_{\mathbb X}^2\leq  e^t \vert U_{j+1,k}^{\rm in} \vert_{\mathbb X}^2+ C \int_0^t e^{t-s}\big( \vert F_{j+1,k}(s,\cdot) \vert_{\mathbb X}^2 + \vert U(s,\cdot) \vert_{{\mathbb Y}^n}^2 + \vert U_{j,k+1} (s,\cdot)\vert_{\mathbb X}^2 \big){\rm d}s.
$$

\noindent{\bf Step 2.} Regularization of $U_{j,k}$. Since, under the assumptions of the proposition, $U_{j,k}$ does not have the required regularity to apply the computations of the previous steps, we consider the regularization $U_{j,k}^\epsilon$ defined for $\epsilon>0$ as
$$
U_{j,k}^\epsilon:=K_\epsilon U_{j,k}\quad\mbox{ with }\quad K_\epsilon:=\big(1-\epsilon \partial_x(\rho^2 \partial_x \cdot ) \big)^{-1};
$$
the smoothing operator $K_\epsilon$ is well defined and acts on $L^2(\Gamma^{\rm D})$, as proved in the following lemma which gathers several properties of  $K_\epsilon$ that we shall need throughout this proof.
\begin{lemma}\label{lemmaregul}
{\bf i.} There is a constant $C>0$ such that for all $f\in L^2(\Gamma^{\rm D})$, one has 
\[
\vert K_\epsilon f\vert_{L^2(\Gamma^{\rm D})}+\sqrt{\epsilon}\vert \rho\dx (K_\epsilon f)\vert_{L^2(\Gamma^{\rm D})}\leq C \vert f\vert_{L^2(\Gamma^{\rm D})}.
\]
{\bf ii.} There is $\epsilon_0>$ and a constant $C>0$ such that for all $0<\epsilon<\epsilon_0$ and $f\in \dot{H}^{1/2}(\Gamma^{\rm D})$, one has 
\[
\vert K_\epsilon f\vert_{\dot{H}^{1/2}(\Gamma^{\rm D})}+\sqrt{\epsilon}\vert \rho\dx (K_\epsilon f)\vert_{\dot{H}^{1/2}(\Gamma^{\rm D})}\leq C \vert f\vert_{\dot{H}^{1/2}(\Gamma^{\rm D})}.
\]
{\bf iii.} Let $j\in {\mathbb N}^*$. There is a constant $C>0$ such that for all $V\in {\mathbb X}$ such that $(\rho\dx)^l V \in {\mathbb X}$ for  $l\leq j-1$, one has
$$
\vert [(\rho\dx)^j,K_\epsilon ]V \vert_{{\mathbb X}}
+\sqrt{\epsilon}\vert \rho\dx [(\rho\dx)^j,K_\epsilon ] V \vert_{{\mathbb X}} \leq C \sqrt{\epsilon}
\sum_{l=0}^{j-1}  \vert (\rho\dx)^l   V \vert_{{\mathbb X}}.
$$
{\bf iv.} For all $f\in L^2(\Gamma^{\rm D})$, $K_\epsilon f$ converges to $f$ in $L^2(\Gamma^{\rm D})$ as $\epsilon \to 0$.\\
{\bf v.} For all $f\in \dot{H}^{1/2}(\Gamma^{\rm D})$,  $K_\epsilon f$ converges to $f$ in $\dot{H}^{1/2}(\Gamma^{\rm D})$ as $\epsilon \to 0$.
\end{lemma}
\begin{proof}[Proof of the lemma]
{\bf i.} The result is a simple consequence of the variational estimates associated with the identity $(1-\epsilon \partial_x(\rho^2 \partial_x \cdot))K_\epsilon f =f$.\\
{\bf ii.} Taking the $L^2$-scalar product of  $(1-\epsilon \partial_x(\rho^2 \partial_x \cdot))K_\epsilon f =f$ with $G_0 K_\epsilon f$, one gets
$$
\vert K_\epsilon f \vert_{\dot{H}^{1/2}}^2+\epsilon \vert \rho\dx K_\epsilon f \vert_{\dot{H}^{1/2}}^2 \lesssim \vert f \vert_{\dot{H}^{1/2}(\Gamma^{\rm D})}\vert K_\epsilon f \vert_{\dot{H}^{1/2}}
+\epsilon\big\vert (\rho\dx K_\epsilon f, [\rho\partial_x,G_0] K_\epsilon f)_{L^2}\big\vert.
$$

Using Proposition \ref{propcomm1}, this implies 
\begin{align*}
\vert K_\epsilon f \vert_{\dot{H}^{1/2}}^2+\epsilon \vert \rho\dx K_\epsilon f \vert_{\dot{H}^{1/2}}^2 \lesssim& \vert f \vert_{\dot{H}^{1/2}(\Gamma^{\rm D})}\vert K_\epsilon f \vert_{\dot{H}^{1/2}}\\
&+\epsilon \vert K_\epsilon f \vert_{\dot{H}^{1/2}}\big(  \vert K_\epsilon f \vert_{\dot{H}^{1/2}}+\vert \rho\dx K_\epsilon f \vert_{\dot{H}^{1/2}}\big),
\end{align*}
from which the result follows by Young's inequality if $\epsilon$ is mall enough. 

{\bf iii.} We need to prove that the estimate of the lemma holds with ${\mathbb X}$ replaced by $L^2(\Gamma^{\rm D})$ and $\dot{H}^{1/2}(\Gamma^{\rm D})$. We start with $f\in  L^2(\Gamma^{\rm D})$.

We observe that  $w_j:=[K_\epsilon, (\rho\dx)^j] f$ solves the elliptic system
\begin{equation}\label{eqw}
(1-\epsilon \dx (\rho^2 \dx \cdot ))w_j=-\epsilon  [(\rho\dx)^j , \dx\rho ] \rho\dx K_\epsilon f ;
\end{equation}
multiplying by $w_j$, integrating by parts and using Young's inequality we deduce  that
$$
\vert w_j\vert_{L^2}^2+\epsilon \vert \rho\dx w_j \vert_{L^2}^2 \lesssim  \sum_{l=0}^{j-1} \epsilon^2 \vert \rho\dx (\rho\dx)^l K_\epsilon f \vert_{L^2}^2.
$$

If $j=1$, then we can use the first point of the lemma to bound the right-hand side from above by $\vert f \vert_{L^2}^2$; for $j>1$, we write $ (\rho\dx)^l K_\epsilon f =K_\epsilon (\rho\dx)^l f -w_l$ and apply the second point proved above to obtain
$$
\vert w_j\vert_{L^2}^2+\epsilon \vert \rho\dx w_j \vert_{L^2}^2\lesssim \epsilon \sum_{l=0}^{j-1} \big( \vert (\rho\dx)^l  f \vert_{L^2}^2+ \sqrt \epsilon \vert \rho\dx w_{l} \vert_{L^2}^2\big), 
$$
so that a simple induction yields 
$$
\vert [(\rho\dx)^j,K_\epsilon ]f \vert_{L^2}
+\sqrt{\epsilon}\vert \rho\dx [(\rho\dx)^j,K_\epsilon ] f \vert_{L^2} \leq C \sqrt{\epsilon}
\sum_{l=0}^{j-1} \big( \vert (\rho\dx)^l   f \vert_{L^2},
$$
which is the desired result. 

In order to prove that a similar results holds in $\dot{H}^{1/2}(\Gamma^{\rm D})$, we multiply \eqref{eqw} by $G_0 w_j$ and proceed as in the case $f\in L^2(\Gamma^{\rm D})$ and point {\bf ii} of the lemma. We omit the details.

{\bf iv.} The result stems directly from functional calculus since the operator $-\partial_x(\rho^2 \partial_x\cdot)$ is positive self-adjoint on the Hilbert space $\{f\in L^2(\Gamma^{\rm D}), \rho\partial_x f \in L^2(\Gamma^{\rm D})\}$ (endowed with its canonical scalar product). 

{\bf v.}  We first  prove that if $\dx f \in L^2(\Gamma^{\rm D})$ then $\vert \partial_x K_\epsilon f\vert_{L^2(\Gamma^{\rm D})}\leq C \vert \partial_x f\vert_{L^2(\Gamma^{\rm D})}$ for some constant $C>0$ independent of $f$. 

Applying $\partial_x$ to the equation $(1-\partial_x(\rho^2\partial_x))K_\epsilon f = f$, we get
$$
(1-\epsilon\partial_x(\rho^2\partial_x))(\partial_x K_\epsilon f) = \partial_x f+2\epsilon \partial_x(\rho (\partial_x\rho) (\partial_x K_\epsilon f)).
$$
Multiplying by $(\partial_x K_\epsilon f) $ and integrating by parts easily yields, for $\epsilon$ small enough
\begin{equation}\label{ineqdxf}
\vert (\partial_x K_\epsilon f)  \vert_{L^2(\Gamma^{\rm D})}+\sqrt{\epsilon} \vert \rho\partial_x (\partial_x K_\epsilon f)  \vert_{L^2(\Gamma^{\rm D})}
\lesssim \vert \partial_x f\vert_{L^2(\Gamma^{\rm D})},
\end{equation}
from which the desired inequality follows. 

Let us now prove the result stated in the lemma, namely, that for all $\varepsilon >0$ and all $f\in \dot{H}^{1/2}(\Gamma^{\rm D})$, one has $\vert K_\epsilon f -f \vert_{\dot{H}^{1/2}} \lesssim \varepsilon $ for $\epsilon$ small enough. By density of ${\mathcal D}(\Gamma^{\rm D})$ in  $\dot{H}^{1/2}(\Gamma^{\rm D})$, we can find $\widetilde{f}\in {\mathcal D}(\Gamma^{\rm D})$ such that $\vert f-\widetilde{f}\vert_{\dot{H}^{1/2}}<\varepsilon/M$ for some $M>0$ to be fixed below. 

We then write
\begin{align}
\nonumber
\vert K_\epsilon f - f \vert_{\dot{H}^{1/2}} &\leq \vert K_\epsilon \widetilde{f} - \widetilde{f} \vert_{\dot{H}^{1/2}}+ \vert K_\epsilon (f-\widetilde{f})\vert_{\dot{H}^{1/2}} + \vert f-\widetilde{f}\vert_{\dot{H}^{1/2}} \\
\label{boundKB}
&\leq \vert K_\epsilon \widetilde{f} - \widetilde{f} \vert_{\dot{H}^{1/2}}+\epsilon/2,
\end{align}
provided that $M\geq 2\big(1+ \vert K_\epsilon \vert_{\dot{H}^{1/2}\to \dot{H}^{1/2}})$. Since moreover, we have by interpolation that
$$
\vert K_\epsilon \widetilde{f} - \widetilde{f} \vert_{\dot{H}^{1/2}}\lesssim \vert K_\epsilon \widetilde{f} - \widetilde{f}\vert_{L^2}^{1/2}\vert K_\epsilon \widetilde{f} - \widetilde{f}\vert_{H^1}^{1/2}
$$
(this interpolation makes sense by using Proposition \ref{propequivter} and a standard interpolation) and because we can infer from the bound on $\dx K_\epsilon$ proved above that $\vert K_\epsilon \widetilde{f} - \widetilde{f}\vert_{H^1}\lesssim \vert \widetilde{f}\vert_{H^1}$, we obtain that
$$
\vert K_\epsilon \widetilde{f} - \widetilde{f} \vert_{\dot{H}^{1/2}}\lesssim \vert K_\epsilon \widetilde{f} - \widetilde{f}\vert_{L^2}^{1/2}\vert \widetilde{f}\vert_{H^1}^{1/2}.
$$
Since we have seen that $K_\epsilon f$ converges to $f$ in $L^2(\Gamma^{\rm D})$, we can bound the right-hand side of this inequality by $\epsilon/2$ for $\epsilon$ small enough. Plugging this into \eqref{boundKB} then yields the desired result.
\end{proof}
Let us note that under the assumption that $U_{j,k+1}\in C([0,T];{\mathbb X})$ and $F_{j,k} \in C([0,T];{\mathbb X})$ and remarking that $\dt U_{j,k}=-U_{j,k+1}+F_{j,k}$ we deduce that $U_{j,k}\in C^1([0,T];{\mathbb X})$. Since by the first two points of Lemma \ref{lemmaregul}, $\rho\dx K_\epsilon$ is bounded on ${\mathbb X}$, this implies that $V_\epsilon =\rho\dx K_\epsilon U_{j,k}\in C^1([0,T];{\mathbb X})$.

In order to justify the computations leading to the energy estimates, we still need to prove that $V_\epsilon\in C([0,T];H^{1/2}\times \dot{H}^{1/2})$; using the fact that $\rho\dx K_\epsilon$ is bounded on $H^{1/2}(\Gamma^{\rm D})\times \dot{H}^{1/2}(\Gamma^{\rm D})$, it is enough to check that $U_{j,k} \in C([0,T];H^{1/2}\times \dot{H}^{1/2})$. Since  $U_{j,k}\in C([0,T];{\mathbb X})$, the only thing to prove is that $\zeta_{j,k} \in  \dot{H}^{1/2}(\Gamma^{\rm D})$. This is a consequence of the assumption that $U_{j,k+1}=(-{\bf A})U_{j,k}\in C([0,T];{\mathbb X})$ since the second component of $(-{\bf A})U_{j,k}$ is equal to $-{\mathtt g}\zeta_{j,k}$. 

We have therefore enough regularity on $V_{\epsilon}$ to justify the computations leading to the energy estimates; the difficulty is now to control the additional commutator terms with the smoothing operator $K_\epsilon$.

\medskip
\noindent{\bf Step 3.} Energy estimates for $V_\epsilon:=\rho\dx K_\epsilon U_{j,k}$. Since $U_{j,k}$ solves
$$
\dt U_{j,k}+{\bf A}U_{j,k}=F_{j,k}+\begin{pmatrix} [(\rho \dx)^{j}, G_0] \psi_{0,k} \\0 \end{pmatrix},
$$
we get that
\begin{equation}\label{eqVepsilon}
\dt V_\epsilon+{\bf A}V_\epsilon=\rho\dx K_\epsilon F_{j,k}+\begin{pmatrix} C_1+C_2+C_3\\0 \end{pmatrix}
\end{equation}
with 
\begin{align*}
C_1= \rho\dx K_\epsilon [(\rho \dx)^{j}, G_0] \psi_{0,k}, \qquad C_2 = \rho \dx [K_\epsilon, G_0] \psi_{j,k}, \qquad C_3= [\rho\dx, G_0]K_\epsilon \psi_{j,k}.
\end{align*}

Taking the scalar product  in $\mathbb X$ of \eqref{eqVepsilon} with $V_\epsilon$ yields as above
\begin{align}
\nonumber
\frac{\rm d}{{\rm d}t}\frac{1}{2} \vert V_{\epsilon} \vert_{\mathbb X}^2 &\leq
\vert \rho\dx K_\epsilon F_{j,k}\vert_{\mathbb X}\vert V_\epsilon \vert_{\mathbb X}
+\sum_{m=1}^3 (C_m, \rho\dx K_\epsilon \zeta_{j,k})_{L^2} \\
\label{NRJestreg}
& \lesssim 
\vert F_{j+1,k}\vert_{\mathbb X}\vert V_\epsilon \vert_{\mathbb X}
+\sum_{m=1}^3 (C_m, \rho\dx K_\epsilon \zeta_{j,k})_{L^2} ,
\end{align}
where we used the first three points of the lemma. 

We now turn to control the three commutator terms. More precisely, we prove below that for $m=1,2,3$, we have
\begin{equation}
\label{boundCm}
\big\vert (C_m, \rho\dx K_\epsilon \zeta_{j,k})_{L^2} \big\vert  
\lesssim  \vert U\vert_{{\mathbb Y}^n}^2 +\vert V_\epsilon \vert_{\mathbb X}^2 + \vert U_{j,k+1}\vert_{\mathbb X}^2.
\end{equation}
\begin{itemize}
\item Contribution of $C_1$. We decompose $C_1=C_{11}+C_{12}+C_{13}$ with
\begin{align*}
C_{11}&=(\rho \dx) [(\rho\dx)^j,G_0]K_\epsilon \psi_{0,k},\\
C_{12}&=(\rho\dx)^{j+1}[K_\epsilon,G_0]\psi_{0,k} -(\rho\dx)[K_\epsilon,G_0](\rho\dx)^j \psi_{0,k},\\
C_{13}&=(\rho\dx)[K_\epsilon,(\rho\dx)^j]G_0\psi_{0,k}-(\rho\dx)G_0[K_\epsilon,(\rho\dx)^j]\psi_{0,k},
\end{align*}
and consider these components separately. For $C_{11}$, we notice that
$$
C_{11}=[(\rho\dx)^{j+1},G_0]K_\epsilon \psi_{0,k}-[\rho\dx,G_0](\rho\dx)^jK_\epsilon\psi_{0,k},
$$
so that we get from Proposition \ref{propcomm1}
\begin{align}
\nonumber
\big\vert (C_{11}, \rho\dx K_\epsilon \zeta_{j,k}  )_{L^2}\big\vert &\lesssim \big( \sum_{l=0}^{j+1}\vert (\rho\dx)^l K_\epsilon \psi_{0,k}\vert_{\dot{H}^{1/2}}\big) \vert K_\epsilon \zeta_{j,k}\vert_{\dot{H}^{1/2}}  \\
\label{boundC11}
&\lesssim \big( \vert U\vert_{{\mathbb Y}^n}+\vert V_\epsilon\vert_{\mathbb X}\big)\vert U_{j,k+1}\vert_{\mathbb X},
\end{align}
which shows that $C_{11}$ satisfies the upper bound \eqref{boundCm}.

For $C_{12}$, we introduce the notation ${\bf E}=-\dx (\rho^2 \dx \cdot) $; since  $[K_\epsilon,G_0]=-\epsilon K_\epsilon [{\bf E},G_0]K_\epsilon$, we deduce that $(C_{12}, \rho\dx K_\epsilon \zeta_{j,k} )_{L^2}=A_1+A_2$ with
\begin{align*}
A_1=&\epsilon  \big( {\bf E} [(\rho\dx)^{j} , K_\epsilon ][{\bf E},G_0]K_\epsilon \psi_{0,k} , K_\epsilon \zeta_{j,k}\big)_{L^2}\\
&+\epsilon \big( {\bf E} K_\epsilon  (\rho\dx)^{j} [{\bf E},G_0]K_\epsilon \psi_{0,k} , K_\epsilon \zeta_{j,k}\big)_{L^2},\\
A_2=&-\epsilon  \big( {\bf E} K_\epsilon [{\bf E},G_0] K_\epsilon (\rho\dx)^j \psi_{0,k}, K_\epsilon \zeta_{j,k}\big)_{L^2}.
\end{align*}
Since moreover $\epsilon {\bf E}K_\epsilon=\epsilon K_\epsilon{\bf E}=1-K_\epsilon$, we can write
\begin{align*}
A_1=&\epsilon\big( [(\rho\dx)^{j} , K_\epsilon ] [{\bf E},G_0]K_\epsilon \psi_{0,k} , (1-K_\epsilon)K_\epsilon \zeta_{j,k}\big)_{L^2}\\
&+\epsilon\big(   (\rho\dx)^{j} [{\bf E},G_0]K_\epsilon \psi_{0,k} , (1-K_\epsilon)K_\epsilon \zeta_{j,k}\big)_{L^2},
\end{align*}
and
$$
{A}_2= -\epsilon \big(  [{\bf E},G_0] K_\epsilon \psi_{j,k}, (1-K_\epsilon) K_\epsilon \zeta_{j,k}\big)_{L^2}.
$$

For $A_1$, we remark that if $j=0$ then the first component of the right-hand side vanishes while the second one gets using Proposition \ref{propcomm2} and Remark \ref{remcommutateur} that
\begin{align*}
\vert A_1\vert &\lesssim \big( \vert K_\epsilon \psi_{0,k}\vert_{\dot{H}^{1/2}}+\vert \rho \dx K_\epsilon \psi_{0,k}\vert_{H^{1/2}} \big) \vert (1-K_\epsilon)K_\epsilon \zeta_{0,k} \vert_{\dot{H}^{1/2}} \\
&\lesssim \big( \vert  U \vert_{{\mathbb Y}^n}+\vert V_\epsilon\vert_{\mathbb X}+  \vert \rho \dx K_\epsilon \psi_{0,k}\vert_{L^2} \big) \vert U_{0,k+1} \vert_{{{\mathbb X}}}.
\end{align*}

Since $\rho$ is bounded, one can use \eqref{ineqdxf} and Proposition \ref{propDNell} to obtain that for all smooth enough function $f$, one has
\begin{equation}\label{contsup}
\vert \rho \dx K_\epsilon f \vert_{L^2}\lesssim  \vert G_0 f\vert_{L^2}+ \vert f \vert_{\dot{H}^{1/2}},
\end{equation}
so that
$$
\vert A_1\vert \lesssim \big( \vert  U \vert_{{\mathbb Y}^n}+\vert V_\epsilon\vert_{\mathbb X}+  \vert U_{0,k+1} \vert_{{{\mathbb X}}} \big) \vert U_{0,k+1} \vert_{{{\mathbb X}}},
$$
so that $\vert A_1\vert$ can be bounded from above by the right-hand side of \eqref{boundCm}.

For $j\geq 1$, 
we notice that since $ [(\rho\dx)^{j} , K_\epsilon ] f=-\epsilon K_\epsilon [(\rho\dx)^j,\dx\rho]\rho\dx K_\epsilon f$, the first component consists of lower order terms so that we can focus on the second one; using also the fact that ${\bf E}=-(\rho\dx)^2 -(\dx\rho)\rho\dx$ we obtain 
\begin{align*}
{A}_1 \sim &\epsilon\big( [(\rho\dx)^{j-1},G_0] (\rho\dx)^2 K_\epsilon \psi_{0,k}, \rho\dx  ((1-K_\epsilon) K_\epsilon \zeta_{j,k})\big)_{L^2} \\
&+\epsilon \big( [(\rho\dx)^{j+1},G_0]K_\epsilon \psi_{0,k}, \rho\dx  ((1-K_\epsilon) K_\epsilon \zeta_{j,k})\big)_{L^2},
\end{align*}
where the notation $\sim$ means equality up to lower order terms that can be controlled by the right-hand side of \eqref{boundCm}.   

It then follows from Proposition \ref{propcomm1} and the first two points of Lemma \ref{lemmaregul} that ${A}_1$ satisfies the same upper bound as \eqref{boundCm}. This is also the case for  ${A}_2$ as a direct consequence of Proposition \ref{propcomm2} and Remark \ref{remcommutateur},  and \eqref{contsup}.
We can therefore conclude that $C_{12}$ satisfies the upper bound \eqref{boundCm}.

For $C_{13}$, we write 
$(C_{13},\rho\dx K_\epsilon \zeta_{j,k}  )_{L^2}=B_1+B_2$ with
\begin{align*}
B_1&=\big((\rho\dx)[K_\epsilon,(\rho\dx)^j]G_0\psi_{0,k} , \rho\dx K_\epsilon \zeta_{j,k}\big)_{L^2}\\
B_2&= -\big( (\rho\dx)G_0[K_\epsilon,(\rho\dx)^j]\psi_{0,k},\rho\dx K_\epsilon \zeta_{j,k}\big)_{L^2}.
\end{align*}
We directly get from Cauchy-Schwarz inequality and the third point of Lemma \ref{lemmaregul} that 
\begin{align*}
\vert B_1\vert &\lesssim \big(\sum_{l=0}^{j-1 } \vert (\rho\dx)^l G_0\psi_{0,k}\vert_{L^2}\big) \vert  \rho\dx K_\epsilon \zeta_{j,k}\vert_2 \\
&\lesssim \vert U\vert_{{\mathbb Y}^n} \vert V_\epsilon \vert_{\mathbb X}.
\end{align*}

For $B_2$, we decompose
\begin{align*}
B_2=&-\big( G_0 \rho\dx[K_\epsilon,(\rho\dx)^j]\psi_{0,k},\rho\dx K_\epsilon \zeta_{j,k}\big)_{L^2}\\
&-\big( [ \rho\dx, G_0]\, [K_\epsilon,(\rho\dx)^j]\psi_{0,k},\rho\dx K_\epsilon \zeta_{j,k}\big)_{L^2},
\end{align*}
so that
\begin{align*}
\vert B_2\vert \lesssim& \frac{1}{\sqrt{\epsilon}} \vert \rho\dx[K_\epsilon,(\rho\dx)^j]\psi_{0,k} \vert_{\dot{H}^{1/2}} \times \sqrt{\epsilon} \vert \rho\dx K_\epsilon \zeta_{j,k} \vert_{\dot{H}^{1/2}} \\
&+ \big( \sum_{l=0}^j \vert (\rho\dx)^l \psi_{0,k}\vert_{\dot{H}^{1/2}} \big)\times \vert K_\epsilon \zeta_{j,k}\vert_{\dot{H}^{1/2}},
\end{align*}
the second line being a consequence of Proposition \ref{propcomm1}  and Lemma \ref{lemmaregul}. Decomposing $ \rho\dx[K_\epsilon,(\rho\dx)^j]=-[(\rho\dx)^{j+1},K_\epsilon]+[\rho\dx,K_\epsilon](\rho\dx)^j$, we can therefore deduce from the first and  third point of Lemma \ref{lemmaregul} that
$$
\vert B_2\vert \lesssim \vert U \vert_{{\mathbb Y}^n} \vert U_{j,k+1}\vert_{{\mathbb X}},
$$
and we can conclude that $C_{13}$, and therefore $C_1$, satisfy the upper bound \eqref{boundCm}.
\item Contribution of $C_2$. Using again the relation $[K_\epsilon, G_0]=-\epsilon K_\epsilon [ {\bf E}, G_0]K_\epsilon$, we have
\begin{align*}
(C_2, \rho\dx K_\epsilon \zeta_{j,k})_{L^2}&=-\epsilon \big( \rho \dx K_\epsilon [ {\bf E}, G_0]K_\epsilon\psi_{j,k},\rho\dx K_\epsilon \zeta_{j,k}\big)_{L^2(\Gamma^{\rm D})} \\
&=- \big(  [ {\bf E}, G_0]K_\epsilon\psi_{j,k},\epsilon{\bf E} K_\epsilon^2 \zeta_{j,k} \big)_{L^2(\Gamma^{\rm D})}. 
\end{align*}
Since $\epsilon{\bf E} K_\epsilon=1-K_\epsilon$,  it follows from Proposition \ref{propcomm2} and Remark \ref{remcommutateur}, as well as \eqref{contsup}, that
\begin{align*}
(C_2, \rho\dx K_\epsilon \zeta_{j,k})_{L^2}& \lesssim  \big( \vert K_\epsilon \psi_{j,k}\vert_{\dot{H}^{1/2}} +  \vert \rho\dx K_\epsilon \psi_{j,k}\vert_{{H}^{1/2}} \big) \times  \vert (1-K_\epsilon)K_\epsilon\zeta_{j,k}\vert_{\dot{H}^{1/2}} \\
&\lesssim \big( \vert U\vert_{{\mathbb Y}^n} +\vert V_\epsilon \vert_{\mathbb X} +\vert U_{j,k+1}\vert_{\mathbb X} \big) \vert U_{j,k+1}\vert_{\mathbb X},
\end{align*}
so that $C_2$ satisfies the upper bound \eqref{boundCm}.
\item Contribution of $C_3$. We directly get from Proposition \ref{propcomm1} that $C_3$ satisfies the upper bound \eqref{boundCm}.
\end{itemize}
We can then conclude from \eqref{NRJestreg} and \eqref{boundCm} that
$$
 \vert V_\epsilon(t) \vert_{\mathbb X}^2\leq  e^t \vert U_{j+1,k}^{\rm in} \vert_{\mathbb X}^2+ C \int_0^t e^{t-s}\big( \vert F_{j+1,k}(s,\cdot) \vert_{\mathbb X}^2 + \vert U(s,\cdot) \vert_{{\mathbb Y}^n}^2 + \vert U_{j,k+1} (s,\cdot)\vert_{\mathbb X}^2 \big){\rm d}s.
$$

\noindent{\bf Step 4.} Conclusion. The energy estimate of the previous point together with the last two points of Lemma \ref{lemmaregul} can be used to show that $V_\epsilon$ is a Cauchy sequence in $C([0,T];{\mathbb X})$ and therefore converges in this space (up to a constant for the second component). Since this limit coincides (up to a constant) with $\rho\dx U_{j,k}$, this proves the result.
\end{proof}

\subsection{Main result}\label{sectMR}
We recall that the spaces ${\mathbb Y}^n$ and ${\mathbb W}^n_T$ are defined in \eqref{defV0} and \eqref{defWT2}, while ${\mathcal N}^n(U(t))$ is defined in \eqref{defN}.
\begin{theorem}\label{maintheo}
Let $\Omega$, $\Gamma^{\rm D}$ and $\Gamma^{\rm N}$ be as in Assumption \ref{assconfig} and Assumption \ref{assconfig2}.  
Let also $n\in {\mathbb N}$, $T>0$, and $U^{\rm in} \in {\mathbb Y}^n$ and $F\in {\mathbb W}^n_T$. Then there is a unique solution $U\in {\mathbb W}^n_T$ to \eqref{CPcompbis}-\eqref{CP0bis}. Moreover, there exist two continuous functions of times $c_1$ and $c_2$ such that
$$
\forall t\geq 0,\qquad  {\mathcal N}^n (U(t))\leq \big(1+t c_1(t)\big) {\mathcal N}^n (U(0)) + c_2(t) \int_0^t {\mathcal N}^n (F(t')){\rm d}t'.
$$
\end{theorem}
\begin{remark}
Similarly to Remark \ref{exprCI}, we can use the equation to bound ${\mathcal N}^n (U(0)) $ from above
in terms of $U^{\rm in}$ and $F$, namely,
$$
{\mathcal N}^n(U(0))\leq  \vert U^{\rm in}\vert_{{\mathbb Y}^n}+{\mathcal N}^{n-1}(F(0)).
$$
\end{remark}

\begin{remark}
For {the order} $n=2$, the quantity $ {\mathcal N}^n (U(t))$ controls $\zeta$ in $H^1(\Gamma^{\rm D})$ and $\dx \psi$ in $H^{1/2}(\Gamma^{\rm D})$ as in Corollary \ref{space regularity2}, but also $(\rho\dx)^2\zeta$ in $L^2(\Gamma^{\rm D})$ and $(\rho\dx)^2\psi$ in $\dot{H}^{1/2}(\Gamma^{\rm D})$, that is, away from the corners, we control one derivative more than in Corollary \ref{space regularity2}.
For $n=3$ and angles smaller than $2\pi/3$, we get as for Corollary \ref{space regularity2}  that $ {\mathcal N}^n (U(t))$ controls $\zeta$ in $H^{3/2}(\Gamma^{\rm D})$ and $\dx \psi$ in $H^{1}(\Gamma^{\rm D})$, but also 
$(\rho\dx)^3\zeta$ in $L^2(\Gamma^{\rm D})$ and $(\rho\dx)^3\psi$ in $\dot{H}^{1/2}(\Gamma^{\rm D})$, that is, $3/2$ more derivatives away from the boundary.
In  de Poyferr\'e's nonlinear a priori estimates (see \cite{Poyferre}) for the case of the emerging bottom, the author had to assume that the contact angle was smaller than $\pi/3$ to get $\dx\psi \in H^s$, with $s>3/2$. Here, such a control is obtained away from the boundary for $n=3$; according to Corollary \ref{space regularity2}, we get a control of $\dx \psi$ in $H^2(\Gamma^{\rm D})$ for $n=5$ (and hence also $(\rho\dx)^5\psi$ in $\dot{H}^{1/2}(\Gamma^{\rm D})$), provided that the angle is smaller than $2\pi/5>\pi/3$.
\end{remark}
\begin{remark}
We have characterized the domain $D({\bf A})=H^{1/2}(\Gamma^{\rm D})\times \dot{H}^1(\Gamma^{\rm D})$ and proved that $D({\bf A})=H^1(\Gamma^{\rm D})\times \dot{H}^{3/2}_*(\Gamma^{\rm D})$, where $H^{3/2}_*(\Gamma^{\rm D})$ is subset of $\dot{H}^{3/2}_*(\Gamma^{\rm D})$ (we have seen that in the case of a vertical contact at least, the inclusion is strict). The characterization of $D({\bf A}^n)$ for all $n\in {\mathbb N}^*$ requires the characterization of $D(G_0^n)$, which is a difficult open problem.
\end{remark}
\begin{remark}
When $\Omega$ is bounded, we can define the Neumann-Dirichlet map $H_0: {\mathcal L}^2(\Gamma^{\rm D})\to {\mathcal H}^1(\Gamma^{\rm D})$ (recall that we use calligraphic letters to denote spaces of functions whose integral over $\Gamma^{\rm D}$ vanishes); we have of course $G_0H_0={\rm Id}$. Defining on ${\mathbb X}={\mathcal L}^2(\Gamma^{\rm D})\times {\mathcal H}^{1/2}(\Gamma^{\rm D})$ the operator ${\mathcal B}$ by the matricial formula
$$
{\bf B}=\begin{pmatrix}
0 & 1/{\mathtt g} \\
-H_0 & 0
\end{pmatrix},
$$
so that ${\bf A}{\bf B}={\rm Id}$, and by $\widetilde{X}^n=\{ U\in {\mathbb X}, (\rho\dx)^j U \in {\mathbb X}, 0\leq j\leq n\}$, one has ${\bf B}^n \widetilde{X}^n \subset {\mathbb Y}^n$.
\end{remark}

\begin{proof}
The first step is to prove the property ${\mathcal P}(n')$ for $n'=n$, where ${\mathcal P}(n')$ states that $U_{j,k}\in C([0,T];{\mathbb X})$ for all $0\leq j+k\leq n'$, where we recall the notation $U_{j,k}=(\rho\dx)^j(-{\bf A})^k U$. This is done by a finite induction on $0\leq n'\leq n$.

For $n'=0$ this is a direct consequence of Theorem \ref{theoWPbounded} or \ref{theoWPunbounded}. For $1\leq n'\leq n-1$ we assume that ${\mathcal P}(n'')$ is proved for $0\leq n''\leq n'$ and show that this implies that ${\mathcal P}(n'+1)$ holds.

Under the assumptions of the theorem, we know by Corollary \ref{corohigherT} that $U_{0,n'}$ and $U_{0,n'+1}$ are in $C([0,T];{\mathbb X})$. Since the assumptions of the theorem also imply that $U_{1,n'}^{\rm in} \in {\mathbb X}$ and $F_{1,n'}\in C([0,T];{\mathbb X})$ we can use Proposition \ref{keyprop} to deduce that $U_{1,n'}\in C([0,T];{\mathbb X})$. Since by ${\mathcal P}(n'-1)$ we know that that $U_{1,n'-1}\in C([0,T];{\mathbb X})$ and that the assumptions of the theorem imply that 
$U_{2,n'-1}^{\rm in}$ and $F_{2,n'-1}$ are in $C([0,T];{\mathbb X})$, we use Proposition \ref{keyprop} again to get that $U_{2,n'-1} \in C([0,T];{\mathbb X})$. By a finite induction, we thus get that $U_{j,k}\in C([0,T];{\mathbb X})$ for all $j+k=n'+1$, hereby proving ${\mathcal P}(n'+1)$. The induction step is therefore complete and ${\mathcal P}(n)$ follows.

Since we now know that $\partial_t^l U_{j+k}\in C([0,T];{\mathbb X})$ for $l=0$ and $0\leq j+k\leq n$, we just need to extend this properties to $j \geq 1$ such that $0\leq l+j+k\leq n$. 

For such a $j$, we can write, using the equation
$$
\dt^l U=(-{\bf A})^l U +\sum_{l'=0}^{l-1}(-{\bf A})^{l-l'}\dt^{l'}F,
$$
and therefore
$$
\dt^l U_{j,k}=U_{j,k+l}+\sum_{l'=0}^{l-1}\dt^{l'}F_{j,k+l-l'}.
$$
The first component of the right-hand side belongs to $C([0,T];{\mathbb X})$ from the case $l=0$ proved above, and the second component is also in this space by assumption. This proves that $U\in {\mathbb W}^n_T$; the energy estimates easily follows from the energy estimates of Corollary \ref{corohigherT} and Proposition \ref{keyprop}.
\end{proof}

\thanks{This work was supported by the BOURGEONS project, grant ANR-23-CE40-0014-01 of the French National Research Agency (ANR),  the project Climath of the PEPR Math-VivEs, ANR-23-EXMA-0003, and by National Natural Science Foundation of China no.12071415 and no.12222116.}

\end{document}